\documentclass{siamart0516}

\usepackage{amssymb}
\usepackage{graphicx}
\usepackage{color}
\usepackage{url}
\usepackage{hyperref}
\usepackage{algorithm}
\usepackage{algpseudocode}
\usepackage{pdflscape}

\usepackage[left=3.25cm,top=3.1cm,right=3.0cm,bottom=3.25cm,bindingoffset=0.5cm]{geometry}
\newtheorem{assumption}{Assumption A.\!\!}
\theoremstyle{remark}
\newtheorem*{remark}{Remark}

\newcommand{\R}{\mathbb{R}}
\newcommand{\Rext}{\mathbb{R}\cup\{+\infty\}}

\newcommand{\snorm}[1]{\Vert#1\Vert}
\newcommand{\abs}[1]{\left\vert#1\right\vert}
\newcommand{\set}[1]{\left\{#1\right\}}
\newcommand{\norm}[1]{\Vert #1\Vert}
\newcommand{\argmin}{\mathrm{arg}\!\min}
\newcommand{\Eproof}{\hfill$\square$}
\newcommand{\ri}[1]{\mathrm{ri}\left(#1\right)}
\newcommand{\xb}{x} 
\newcommand{\yb}{y} 
\newcommand{\zb}{z} 
\newcommand{\ub}{u}
\newcommand{\vb}{v} 
 
\newcommand{\ab}{a}

\newcommand{\Zb}{Z}
\newcommand{\cb}{c}
\newcommand{\Nc}{\mathcal{N}}
\newcommand{\Kc}{\mathcal{K}}
\newcommand{\Xc}{\mathcal{X}}
\newcommand{\Yc}{\mathcal{Y}}
\newcommand{\Zc}{\mathcal{Z}}
\newcommand{\Wc}{\mathcal{W}}

\newcommand{\Db}{D}
\newcommand{\Sb}{\mathbf{S}}

\newcommand{\dom}[1]{\mathrm{dom}\left(#1\right)}

\renewcommand{\vec}[1]{\mathrm{vec}\left(#1\right)}

\newcommand{\Lc}{\mathcal{L}}
\newcommand{\Tc}{\mathcal{T}}

\newcommand{\rb}{r} 
\newcommand{\wb}{w} 
\newcommand{\bb}{b} 
\newcommand{\Ab}{A}

\newcommand{\Bc}{\mathcal{B}}
\newcommand{\Cc}{\mathcal{C}}
\newcommand{\iprod}[1]{\left\langle #1\right\rangle}
\newcommand{\iprods}[1]{\langle #1\rangle}
\newcommand{\dist}[1]{\mathrm{dist}\left(#1\right)}

\newcommand{\prox}{\textrm{prox}}

\newcommand{\xopt}{x^{\star}}
\newcommand{\yopt}{y^{\star}}

\newcommand{\wopt}{w^{\star}}
\newcommand{\Xopt}{\mathcal{X}^{\star}}
\newcommand{\Yopt}{\mathcal{Y}^{\star}}

\newcommand{\Wopt}{\mathcal{W}^{\star}}
\newcommand{\fopt}{f^{\star}}

\newcommand{\Popt}{P^{\star}}
\newcommand{\Dopt}{D^{\star}}

\newcommand{\xbar}{\bar{x}}
\newcommand{\ybar}{\bar{y}}
\newcommand{\zbar}{\bar{z}}
\newcommand{\wbar}{\bar{w}}
\newcommand{\xhat}{\hat{x}}
\newcommand{\yhat}{\hat{y}}

\newcommand{\xtilde}{\tilde{x}}

\newcommand{\xast}{x^{\ast}}
\newcommand{\yast}{y^{\ast}}

\newcommand{\Id}{\mathbb{I}}

\newcommand{\Fc}{\mathcal{F}}
\newcommand{\Sc}{\mathcal{S}}

\newcommand{\TheTitle}{A Smooth Primal-Dual Optimization Framework for Nonsmooth Composite Convex Minimization} 
\newcommand{\TheRunningTitle}{A Smooth Primal-Dual Optimization Framework}
\newcommand{\TheAuthors}{Quoc Tran-Dinh, Olivier Fercoq, Volkan Cevher}
\newcommand{\TheAbAuthors}{Q. Tran-Dinh, O. Fercoq, V. Cevher}

\headers{\TheRunningTitle}{\TheAbAuthors}

\title{{\TheTitle}}

\author{
Quoc Tran-Dinh\thanks{Department of Statistics and Operations Research, University of North Carolina at Chapel Hill (UNC), USA. \newline Email: {\tt quoctd@email.unc.edu}.} 
\and Olivier Fercoq\thanks{LTCI, CNRS, T\'el\'ecom ParisTech, Universit\'e Paris-Saclay, 75013, Paris, France. \newline E-mail: {\tt olivier.fercoq@telecom-paristech.fr}.}
\and Volkan Cevher\thanks{Laboratory  for Information and Inference Systems (LIONS), \'{E}cole Polytechnique F\'{e}d\'{e}rale de Lausanne (EPFL),  CH1015 - Lausanne, Switzerland. E-mail: {\tt volkan.cevher@epfl.ch}.}
}

\ifpdf
\hypersetup{
  pdftitle={\TheTitle},
  pdfauthor={\TheAuthors}
}
\fi

\begin{document}
\maketitle

\begin{abstract}
We propose a new and low per-iteration complexity first-order primal-dual optimization framework for a convex optimization template with broad applications. 
Our analysis relies on a novel combination of three classic ideas applied to the primal-dual gap function: smoothing, acceleration, and homotopy. 
The algorithms due to the new approach achieve the best-known convergence rate results, in particular when the template consists of only non-smooth functions.  
We also outline a restart strategy for the acceleration to significantly enhance the practical performance. 
We demonstrate relations with the augmented Lagrangian method and show how to exploit the strongly convex objectives with rigorous convergence rate guarantees.
We provide representative examples to illustrate that the new methods can outperform the state-of-the-art,  including  Chambolle-Pock, and the alternating direction method-of-multipliers algorithms.   
We also compare our algorithms with the well-known Nesterov's smoothing method.
\vspace{1ex}

\end{abstract}

\vspace{-1ex}
\noindent
\vspace{-1ex}

\section{Introduction}\label{sec:intro}
We introduce a new analysis framework for designing primal-dual optimization algorithms to obtain numerical solutions to the following convex optimization template  described in the primal space:  
\begin{equation}\label{eq:primal_cvx}
P^{\star} := \min_{\xb\in\R^n}\Big\{ P(\xb) := f(\xb) + g(\Ab\xb) \Big\},
\end{equation}
where $f : \R^n \to \Rext$ and $g : \R^m \to\Rext$ are proper, closed and convex functions, and $\Ab\in\R^{m\times n}$ is given.
For generality, we do not impose any smoothness assumption on $f$ and $g$. In particular, we refer to \eqref{eq:primal_cvx} as a nonsmooth composite minimization problem. 

Associated with the primal problem~\eqref{eq:primal_cvx}, we define the following dual formulation:
\begin{equation}\label{eq:dual_cvx}
D^{\star} := \max_{\yb\in\R^m}\Big\{ D(\yb) := -f^{\ast}(-\Ab^{\top}\yb) - g^{\ast}(\yb) \Big\},
\end{equation}
where $f^{\ast}$ and $g^{\ast}$ are the Fenchel conjugate of $f$ and $g$, respectively.
Clearly, \eqref{eq:dual_cvx} has  the same form as \eqref{eq:primal_cvx} in the dual space.

The templates \eqref{eq:primal_cvx}-\eqref{eq:dual_cvx} provide a unified formulation for a broad set of applications in various disciplines, see, e.g., 
\cite{Bertsekas1989b,Boyd2004,Cevher2014,Chandrasekaranm2012,McCoy2014,Nocedal2006,Wainwright2014}. 
While problem \eqref{eq:primal_cvx} is presented in the unconstrained form, it automatically covers constrained settings by means of indicator functions.
For example, \eqref{eq:primal_cvx}  covers the following prototypical optimization template via $g(\zb) := \delta_{\set{\cb}}(\zb)$ (i.e., the indicator function of the convex set $\set{\cb}$):
\begin{equation}\label{eq:constr_cvx}
\fopt := \min_{\xb\in\R^n} \big\{ f(\xb) + \delta_{\set{\cb}}(\Ab\xb)  \big\} \equiv \min_{\xb\in\R^n} \big\{ f(\xb) \mid \Ab\xb = \cb  \big\},
\end{equation}
where $f$ is a proper, closed and convex function as in \eqref{eq:primal_cvx}. 
Note that \eqref{eq:constr_cvx} is sufficiently general to cover standard convex optimization subclasses, such as conic programming, monotropic programming, 
and geometric programming, as specific instances \cite{BenTal2001,Bertsekas1996d,Boyd2011}. 

Among classical convex optimization methods, the primal-dual approach is perhaps one of the best candidates to solve the primal-dual pair \eqref{eq:primal_cvx}-\eqref{eq:dual_cvx}. 
Theory and methods along this approach have been developed for several decades and have led to a diverse set of algorithms, see, e.g., 
\cite{Bauschke2011,Boyd2011,Chambolle2011,chen2016direct,Chen1994,combettes2012primal,Davis2014a,Davis2014,Davis2014b,deng2013parallel,Esser2010a,Goldstein2013,He2012a,He2012,Lan2013,Lan2013b,lin2015sublinear,Monteiro2011,Monteiro2012b,Monteiro2010,Nesterov2005d,Ouyang2013,Ouyang2014,Shefi2014,Tseng1991a}, and the references quoted therein.
A more thorough comparison between existing primal-dual methods and our approach in this paper is postponed to Section~\ref{sec:comparison}.
There are several reasons for our emphasis on first-order primal-dual methods for \eqref{eq:primal_cvx}-\eqref{eq:dual_cvx}, with the most obvious one being their scalability. 
Coupled with recent demand for low-to-medium accuracy solutions in applications, these methods indeed provide important trade-offs between the per-iteration complexity and the iteration-convergence rate along with the ability to distribute and decentralize the computation.

Unfortunately, the newfound popularity of primal-dual optimization has lead to an explosion in the number of different algorithmic variants, each of which requires different set of assumptions on problem settings or methods, such as strong convexity, error bound conditions, metric regularity, Lipschitz gradient, Kurdyka-{\L}ojasiewicz conditions or penalty parameter tuning \cite{cai2016convergence,lin2015global,lin2015iteration}. 
As a result, the optimal choice of the algorithm for a given application is often unclear as it is not guided by theoretical principles, but rather trial-and-error procedures, which can incur unpredictable computational costs. 
A vast list of key references can be found, e.g., in \cite{Chambolle2011,Shefi2014}.

To this end, we address the following key question: ``Can we construct heuristic-free, accelerated first-order primal-dual methods for nonsmooth composite minimization that have the best-known convergence rate guarantees?'' 
To our best knowledge, this question has never been addressed fully in a unified fashion in this generality. 
Intriguingly, our theory  is still applicable to the smooth cases of $f$ without requiring neither Lipschitz gradient nor strongly convex-type assumption. 
Such a model covers serval important applications, such as graphical learning models and Poisson imaging reconstruction \cite{Tran-Dinh2013a}.

\subsection{Our approach} 
Associated with the primal problem \eqref{eq:primal_cvx} and the dual one \eqref{eq:dual_cvx}, we define
\begin{equation}\label{eq:gap_fun}
G(\wb) := P(\xb) - D(\yb),
\end{equation}
as a primal-dual gap function, where $\wb := (\xb, \yb)$ is the concatenated primal-dual variable. 
The gap function $G$ in \eqref{eq:gap_fun} is convex in terms of $w$.  Under strong duality, we have $G(\wb^{\star}) = 0$ if and only if $\wb^{\star} := (\xb^{\star}, \yb^{\star})$ is a primal-dual solution of \eqref{eq:primal_cvx} and \eqref{eq:dual_cvx}. 

The gap function \eqref{eq:gap_fun} is widely used in convex optimization and variational inequalities, see, e.g., \cite{Facchinei2003}.
Several researchers have already used the gap function as a tool to characterize the convergence of  optimization algorithms, e.g., within a variational inequality framework \cite{Chambolle2011,He2012,Ouyang2014}.

In stark contrast with the existing literature, our analysis relies on a novel combination of three ideas applied to the primal-dual gap function: \textit{smoothing, acceleration, and homotopy}. 
While some combinations of these techniques have already been studied in the literature, their full combination is important for the desiderata and has not been studied yet. 

\textit{Smoothing:} 
We can obtain a smoothed estimate of the gap function within Nesterov's smoothing technique applied to $f$ and $g$ \cite{Beck2012a,Nesterov2005c}. In the sequel, we denote the smoothed  gap function by $G_{\gamma\beta}(\wb) := P_{\beta}(\xb) - D_{\gamma}(\yb)$ to approximate the primal-dual gap function $G(\wb)$, where $P_{\beta}$ is a smoothed approximation to $P$ depending on the smoothness parameter $\beta > 0$, and $D_{\gamma}$ is a smoothed approximation to $D$ depending on the smoothness parameter $\gamma > 0$. 
By smoothed approximation, we mean the same max-form approximation as \cite{Nesterov2005c}.
However, it is still unclear how to properly update these smoothness parameters in primal-dual methods.

\textit{Acceleration:} 
Using an accelerated scheme, we will design new primal-dual decomposition methods that satisfy the following smoothed gap reduction model: 
\begin{equation}\label{eq:gap_reduction_model}
G_{\gamma_{k+1}\beta_{k+1}}(\wbar^{k+1}) \leq  (1-\tau_k)G_{\gamma_k\beta_k}(\wbar^k) + ~\psi_k,
\end{equation}
where  $\{\wbar^k\}$ and the parameters are generated by the algorithms with $\tau_k \in [0, 1)$ and $\set{\max\set{\psi_k, 0}}$ converges to zero. Similar ideas have been proposed before; for instance, Nesterov's excessive gap technique \cite{Nesterov2005d} is a special case of the gap reduction model \eqref{eq:gap_reduction_model} when $\psi_k \leq 0$ (see \cite{TranDinh2014b}).

\textit{Homotopy:} 
We will design algorithms to maintain \eqref{eq:gap_reduction_model} while simultaneously updating $\beta_k$, $\gamma_k$ and $\tau_k$ to zero to achieve the best-known convergence rate based on the assumptions imposed on the problem template. 
This strategy will also allow our theoretical guarantees not to depend on the diameter of the feasible set  of \eqref{eq:constr_cvx}. 
A similar technique is also proposed in \cite{Nesterov2005d}, but only for symmetric primal-dual methods. It is also used in conjunction with Nesterov's smoothing technique in \cite{boct2012variable} for unconstrained problem but had only an $\mathcal{O}(\ln(k)/k)$ convergence rate.

Note that without homotopy, we can directly apply Nesterov's accelerated methods to minimize the smoothed gap function $G_{\gamma\beta}$ for given $\gamma > 0$ and $\beta > 0$. In this case, these smoothness parameters must be fixed a priori depending on the desired accuracy and the prox-diameter of both the primal and dual problems, which may not be applicable to  \eqref{eq:constr_cvx} due to the unboundedness of the dual feasible domain.

\vspace{-1ex}
\subsection{Our contributions}
Our main contributions  can be summarized as follows:
\begin{itemize}
\item[$\mathrm{(a)}$] (\textit{Theory}) 
We propose to use differentiable smoothing prox function to smooth both primal and dual objective functions, which allows us to update the smoothness parameters in a heuristic-free manner.
We introduce a new model-based gap reduction condition for constructing novel first-order primal-dual methods that can operate in a black-box fashion (in the sense of \cite{Nesterov2004}). 
Our analysis technique unifies several classical concepts in convex optimization, from  Auslander's gap function \cite{Auslender1976} and Nesterov's smoothing technique \cite{Beck2012a,Nesterov2005c} to the accelerated proximal gradient descent method, in a nontrivial manner. 
We also prove a fundamental bound on the primal objective residual and the feasibility violation for \eqref{eq:constr_cvx}, which leads to the main results of our convergence guarantees.

\item[$\mathrm{(b)}$] (\textit{Algorithms and convergence theory}) We propose two novel primal-dual first-order algorithms for solving \eqref{eq:primal_cvx} and \eqref{eq:constr_cvx}. 
The first algorithm requires to perform only one primal step and one dual step without using any primal averaging scheme.
The second algorithm needs one primal step and two dual steps but using a weighted averaging scheme on the primal.
We prove an $\mathcal{O}(1/k)$ convergence rate on the objective residual $P(\xbar^k) - P^{\star}$ of \eqref{eq:primal_cvx} for both algorithms, which is the best-known in the literature for the fully nonsmooth setting. 
For the constrained case \eqref{eq:constr_cvx}, we also prove the convergence of both algorithms in terms of the primal objective residual and the feasibility violation, both achieve an $\mathcal{O}(1/k)$ convergence rate, and are independent of the prox-diameters unlike existing smoothing techniques \cite{Beck2012a,Nesterov2005d,Nesterov2005c}.

\item[$\mathrm{(c)}$] (\textit{Special cases}) We illustrate that the new techniques enable us to exploit additional structures, including the augmented Lagrangian smoothing scheme, and the strong convexity of the objectives. 
We show the flexibility of our framework by applying it to different constrained settings including conic programs.
\end{itemize}
Let us emphasize some key aspects of this work in detail.
First, our characterization is radically different from existing results such as  \cite{Beck2014,Chambolle2011,Deng2012,He2012a,He2012,Ouyang2014,Shefi2014} thanks to the separation of the convergence rates for primal objective residual and the feasibility gap for \eqref{eq:constr_cvx}. 
We believe that this is important since the separated constraint feasibility guarantee can be interpreted as a consensus rate in distributed optimization. 
Second, our assumptions cover a broader class of problems: we can trade-off the primal objective residual and the feasibility gap without any heuristic strategy on the algorithmic parameters while maintaining the best-known convergence rate for a class of fully nonsmooth convex problems in \eqref{eq:constr_cvx}.
Third, our augmented Lagrangian algorithm generates simultaneously both the primal-dual sequence compared to existing augmented Lagrangian algorithms, while it maintains its $\mathcal{O}\left(\frac{1}{k^2}\right)$-worst-case convergence rate both on the objective residual and on the feasibility gap. 
Fourth, we also describe how to adapt known structures on the objective and the constraint components, such as strong convexity to obtain new variants of our methods.  
Fifth, this work significantly expands on our earlier conference work \cite{TranDinh2014b} not only with new methods but also by demonstrating the impact of warm-start and restart. 
Finally, our forthcoming paper \cite{Tran-Dinh2015}  also demonstrates how our analysis framework and gap reduction model  extend to cover alternating direction optimization methods. 

\subsection{Paper organization}
In Section \ref{sec:smoothing}, we propose a smoothing technique with proximity functions for \eqref{eq:primal_cvx}-\eqref{eq:constr_cvx} to estimate the primal-dual gap. 
We also investigate the properties of smoothed gap function and introduce the model-based gap reduction condition. 
Section \ref{sec:alg_scheme} presents the first primal-dual algorithmic framework using accelerated (proximal-) gradient schemes for solving \eqref{eq:primal_cvx}-\eqref{eq:constr_cvx} and its convergence theory. 
Section \ref{sec:algorithm3} provides  the second primal-dual algorithmic framework using averaging sequences for solving \eqref{eq:primal_cvx}-\eqref{eq:constr_cvx} and its convergence theory. 
Section \ref{sec:variants} specifies different instances of our algorithmic framework for \eqref{eq:primal_cvx}-\eqref{eq:constr_cvx} under other common optimization structures and generalizes it to the cone constraint $\Ab\xb - \cb \in \mathcal{K}$. 
Numerical examples are presented in Section~\ref{sec:num_experiments}.
A comparison between our approach and existing methods is given in Section~\ref{sec:comparison}.
For clarity of exposition, technical proofs are moved to the appendix.

\section{Smoothed gap function and optimality characterization}\label{sec:smoothing}
We propose to smooth the primal-dual gap function defined by \eqref{eq:gap_fun} by proximity functions.
Then, we provide a key lemma to characterize the optimality condition for \eqref{eq:primal_cvx} and \eqref{eq:dual_cvx}.

\subsection{Basic notation}
We  use $\iprods{\cdot,\cdot}$ for the standard inner product and $\norm{\xb}_2$ for the Euclidean norm.
Given a matrix $\Sb$, we define a semi-norm of $\xb$ as $\snorm{\xb}_{\Sb} := \sqrt{\iprods{\Sb\xb,\Sb\xb}}$.
When $\Sb$ is the identity matrix $\Id$, we recover the standard Euclidean norm.
When $\Sb^{\top}\Sb$ is positive definite, the semi-norm becomes  a weighted-norm. In this case, its dual norm exists and is defined by $\snorm{\ub}_{\Sc,\ast} = \max\set{\iprods{\ub,\vb} \mid \snorm{\vb}_{\Sb}=1}$. 
When $\Sb^{\top}\Sb$ is not positive definite, we still consider the quantity $\snorm{\ub}_{\Sc,\ast} = \max\set{\iprods{\ub,\vb} \mid \snorm{\vb}_{\Sb}=1}$, although
$\snorm{\ub}_{\Sc,\ast}$ is finite if and only if $\ub \in \mathrm{Ran}(\Sb^\top)$. 

We also use $\Vert\cdot\Vert_{\Xc}$ (respectively, $\Vert\cdot\Vert_{\Yc}$) and $\Vert\cdot\Vert_{\Xc, \ast}$ (respectively, $\Vert\cdot\Vert_{\Yc,\ast}$) for the norm and the corresponding dual norm in the primal space $\Xc$ (respectively, the dual space $\Yc$) induced by the above standard inner product in $\Xc$ (respectively, in $\Yc$).
Given a proper, closed, and convex function $f$, we use $\dom{f}$ and $\partial{f}(\xb)$ to denote its domain and its subdifferential at $\xb$, respectively. 
If $f$ is differentiable, then we use $\nabla{f}(\xb)$ for its gradient at $\xb$.
For a given set $\Cc$, $\delta_{\Cc}(\xb) := 0$ if $\xb\in \Cc$, and $\delta_{\Cc}(\xb) := +\infty$, otherwise, denotes the indicator function of $\Cc$. In addition, $\ri{\Cc}$ denotes the relative interior of $\Cc$.

For a smooth function $f: \Zc \to \mathbb{R}$, we say that $f$ has the $L_f$-Lipschitz gradient with respect to the norm $\norm{\cdot}_\Zc$ if for any $\zb, \tilde{\zb}\in\dom{f}$, we have $\Vert\nabla{f}(\zb) - \nabla{f}(\tilde{\zb})\Vert_{\Zc, \ast} \leq L_f\norm{\zb - \tilde{\zb}}_\Zc$, where $L_f \equiv L(f) \in [0, \infty)$. 
We denote by $\Fc_{L_f}^{1,1}$ the class of all convex functions $f$ with the $L_f$-Lipschitz gradient.
We also use $\mu_f \equiv \mu(f)$ for the strong convexity parameter of a convex function $f$
with respect to the semi-norm $\snorm{\cdot}_{\Zc}$, i.e., $f(\cdot) - (\mu_f/2)\snorm{\cdot}_{\Zc}^2$ is convex.
For a proper, closed and convex function $f$, we use $\prox_f$ to denote its proximal operator, which is defined as $\prox_f(\zb) := \argmin\set{ f(\ub) + (1/2)\Vert\ub - \zb\Vert_\Zc^2 \mid \ub\in\dom{f} }$.

\subsection{Smooth proximity functions and Bregman distance}
We use the following two mathematical tools in the sequel.

\subsubsection{Proximity functions} 
Given a nonempty, closed and convex set $\Zc$ in the primal space or in the dual space, a  continuous, and $\mu_p$-strongly convex function $p$ is called a \textit{proximity} function (or a prox-function) of $\Zc$ if $\Zc\subseteq \dom{p}$. 
We also denote 
\vspace{-0.5ex}
\begin{equation}\label{eq:center_point}
\zbar^c := \argmin\set{p(\zb) \mid \zb\in\dom{p}}~~~\text{and}~~~D_{\Zc} := \sup\set{ p(\zb) \mid \zb\in\Zc},
\vspace{-0.5ex}
\end{equation} 
as the prox-center of $p$ and the prox-diameter of $\Zc$, respectively.
Without loss of generality, we can assume that $\mu_p = 1$ and $p(\zbar^c) = 0$. Otherwise, we can shift and rescale the function $p$. Moreover, $D_{\Zc} \geq 0$, and it is finite if $\Zc$ is bounded.

In addition to the strong convexity, we also limit our class of prox-functions to the smooth ones, which have a Lipschitz gradient with the Lipschitz constant $L_p \geq 1$. We denote the class of prox-functions whose gradient has Lipschitz constant $L$ by $\Sc^{1,1}_{L,1}$.
For example, $p_{\Zc}(\zb) := (1/2)\snorm{\zb}_{\Sb}^2$ is a simple prox-function in $\Zc = \R^{n_z}$, i.e., $\frac{1}{2}\snorm{\cdot}_{\Sb}^2 \in\Sc^{1,1}_{\norm{\Sb}^2,1}(\Zc)$. 

\subsubsection{Bregman distance}
Instead of smoothing the primal and dual problems \eqref{eq:primal_cvx}-\eqref{eq:dual_cvx} by smooth proximity functions, we use a Bregman distance defined via $p_{\Zc}$ as
\begin{equation}\label{eq:bregman_dist}
b_{\Zc}(\zb, \dot{\zb}) := p_{\Zc}(\zb) - p_{\Zc}(\dot{\zb}) - \iprods{\nabla{p}_{\Zc}(\dot{\zb}), \zb - \dot{\zb}},~~\forall \zb,\dot{\zb}\in\Zc,
\end{equation}
where $p_{\Zc}\in\Sc_{L,1}^{1,1}(\Zc)$.
Clearly, if we fix $\dot{\zb} = \bar{\zb}^c$ at the center point of $p_{\Zc}$, then $b_{\Zc}(\zb,\bar{\zb}^c) = p_{\Zc}(\zb)$.
In addition, $\nabla_1{b_{\Zc}}(\zb,\zb) = 0$ for all $\zb\in\Zc$. We use in the sequel $\nabla{b_{\Zc}}$ for $\nabla_1{b_{\Zc}}$.

\subsection{Basic assumption}
Our main assumption for  problems \eqref{eq:primal_cvx}-\eqref{eq:dual_cvx} is to guarantee the strong duality, which essentially requires the following assumption (see, \cite[Proposition 15.22]{Bauschke2011}).

\begin{assumption}\label{as:A1}
The solution set $\Xopt$ of the primal problem \eqref{eq:primal_cvx} $($or \eqref{eq:constr_cvx}$)$ is nonempty.
In addition, the following assumption holds for either \eqref{eq:primal_cvx} or \eqref{eq:constr_cvx}:
\begin{itemize}
\item[$\mathrm{(a)}$] The condition $0 \in \ri{\dom{g} - \Ab(\dom{f})}$ for \eqref{eq:primal_cvx} holds.
\item[$\mathrm{(b)}$] The Slater condition $\ri{\dom{f}}\cap\set{\xb\in\R^n \mid \Ab\xb = \cb} \neq\emptyset$ for \eqref{eq:constr_cvx} holds.
\end{itemize}
\end{assumption}
Now, we define $\Xc := \overline{\dom{f}}$, $\Yc := \overline{\dom{g^{\ast}}}$ and $\Wc := \Xc\times\Yc$.
Note that if the function $g$ in \eqref{eq:primal_cvx} is Lipschitz continuous on $\Yc$, then Assumption~A.\ref{as:A1} holds.

Under Assumption~A.\ref{as:A1}, the strong duality for  \eqref{eq:primal_cvx}-\eqref{eq:dual_cvx} holds, see, e.g.,~\cite{Bauschke2011}.
The solution set $\Yopt$ of the dual problem \eqref{eq:dual_cvx} is nonempty, and
\begin{equation}\label{eq:strong_duality}
\Popt = f(\xopt) + g(\Ab\xopt) = \Dopt = -f^{\ast}(-\Ab^{\top}\yopt) - g^{\ast}(\yopt),~\forall \xopt \in\Xopt,~\forall\yopt\in\Yopt.
\end{equation}
Let $\Wopt := \Xopt\times\Yopt$ be the primal-dual (or the saddle point) set of \eqref{eq:primal_cvx}-\eqref{eq:dual_cvx}.
Then, \eqref{eq:strong_duality} is equivalent to $f(\xopt) + g(\Ab\xopt) + f^{\ast}(-\Ab^{\top}\yopt) + g^{\ast}(\yopt) = 0$ for all $(\xopt,\yopt)\in\Xopt\times\Yopt$.
In addition, we can write the optimality condition of \eqref{eq:primal_cvx}-\eqref{eq:dual_cvx} as follows:
\begin{equation}\label{eq:opt_cond}
- \Ab^{\top}\yopt \in \partial{f}(\xopt)~~\text{and}~~~\yopt \in \partial{g}(\Ab\xopt).
\end{equation}
Note that this condition can be written as $0 \in \partial{f}(\xopt) + \Ab^{\top}\partial{g}(\Ab\xopt)$ for the primal problem \eqref{eq:primal_cvx}, and $0 \in \partial{g^{\ast}}(\yopt) - \Ab\partial{f^{\ast}}(-\Ab^{\top}\yopt)$ for the dual problem \eqref{eq:dual_cvx}.

\subsection{Smoothed primal-dual gap function}\label{subsec:excessive_gap}
The gap function $G$ defined in \eqref{eq:gap_fun} is \textit{convex} but generally \textit{nonsmooth}. 
This subsection introduces a smoothed primal-dual gap function that approximates $G$ using smooth prox-functions. 

\subsubsection{The first smoothed approximation}
Let $b_{\Xc}$ be a Bregman distance defined on $\Xc$, and $\dot{\xb}\in\Xc$ be given, we consider an approximation to the dual objective function $D(\cdot)$ as
\begin{equation}\label{eq:D_gamma}
D_{\gamma}(\yb; \dot{\xb}) := \displaystyle\min_{\xb\in\Xc}\set{ f(\xb) + \iprods{\yb, \Ab\xb} + \gamma b_{\Xc}(\xb,\dot{\xb})} - g^{\ast}(\yb) \equiv -f_{\gamma}^{\ast}(-\Ab^{\top}\yb;\dot{\xb}) - g^{\ast}(\yb),
\end{equation}
where $\gamma > 0$ is a dual smoothness parameter.
The minimization subproblem in \eqref{eq:D_gamma} always admits a solution, which is denoted by
\vspace{-1ex}
\begin{equation}\label{eq:x_ast}
\xast_{\gamma}(\yb;\dot{\xb}) := \displaystyle\argmin_{\xb\in\Xc}\set{f(\xb) + \iprods{\yb,\Ab\xb} + \gamma b_{\Xc}(\xb,\dot{\xb})}.
\vspace{-1ex}
\end{equation}

We emphasize that our algorithms presented in the next sections support parallel and distributed computation for the \textit{decomposable setting} of \eqref{eq:primal_cvx} or \eqref{eq:constr_cvx}, where $f$ is decomposed into $N$ terms as $f(\xb) := \sum_{i=1}^Nf_i(\xb_i)$ with the $i$-th block being in $\R^{n_i}$ such that $\sum_{i=1}^Nn_i = n$.
In this case, we can choose a separable prox-function to generate a decomposable Bregman distance $b_{\Xc}(\xb,\dot{\xb}) := \sum_{i=1}^Nb_i(\xb_i, \dot{\xb}_i)$ to approximate the dual function $D$ defined in \eqref{eq:dual_cvx}.
By exploiting this decomposable structure, we can evaluate the smoothed dual function and its gradient in a parallel or distributed fashion.
We will discuss the detail of this setting in the sequel, see, Section~\ref{sec:variants}.

\subsubsection{The second smoothed approximation}
Let $b_{\Yc}$ be a Bregman distance defined on $\Yc$ the feasible set of the dual problem \eqref{eq:dual_cvx} and $\dot y \in \Yc$.
We consider an approximation to the objective $g(\cdot)$ in \eqref{eq:primal_cvx} as
\begin{equation}\label{eq:g_beta}
g_{\beta}(\ub;\dot{\yb}) := \max_{\yb\in\Yc}\set{ \iprods{\ub,\yb} - g^{\ast}(\yb) - \beta b_{\Yc}(\yb,\dot{\yb})},
\end{equation}
where $\beta >0$ is a primal smoothness parameter.
We also denote the solution of the maximization problem in \eqref{eq:g_beta} by $\yb^{\ast}_{\beta}(\ub; \dot{\yb})$, i.e.:
\begin{equation}\label{eq:yast_beta}
\yb^{\ast}_{\beta}(\ub;\dot{\yb}) := \arg\max_{\yb\in\Yc}\set{ \iprods{\ub,\yb} - g^{\ast}(\yb) - \beta b_{\Yc}(\yb,\dot{\yb})}.
\end{equation}
We consider an approximation to the primal objective function $P$ as
\begin{equation}\label{eq:F_beta}
P_{\beta}(\xb;\dot{\yb}) := f(\xb) + g_{\beta}(\Ab\xb; \dot{\yb}).
\end{equation}
This function is the second smoothed approximation for the primal problem.
We note that if $g(\cdot) := \delta_{\set{\cb}}(\cdot)$ and $p_{\Yc}(\cdot) := (1/2)\norm{\cdot }_2^2$, then $\yb^{\ast}_{\beta}(\ub;\dot{\yb}) = \dot{\yb} + \beta^{-1}(\ub - \cb)$, which has a closed form.

\subsection{Smoothed gap function and its properties}
Given $D_{\gamma}$ and $P_{\beta}$ defined by \eqref{eq:D_gamma} and \eqref{eq:F_beta}, respectively, and the primal-dual variable $\wb := (\xb, \yb)$,  the smoothed primal-dual gap (or the smoothed gap) $G_{\gamma \beta}$ is now defined as
\begin{equation}\label{eq:smoothed_gap_func}
G_{\gamma\beta}(\wb;\dot{\wb}) :=  P_{\beta}(\xb;\dot{\yb}) - D_{\gamma}(\yb;\dot{\xb}),
\end{equation}
where $\gamma$ and $\beta$ are two smoothness parameters, and $\dot{\wb} := (\dot{\xb}, \dot{\yb})$.

The following lemma provides  fundamental bounds of the objective residual $P(\xb) - \Popt $ for the unconstrained form \eqref{eq:primal_cvx}, and the objective residual $f(\xb) - \fopt$ and the feasibility gap $\norm{\Ab\xb - \cb}_{\Yc,\ast}$ for the constrained form \eqref{eq:constr_cvx}.
For clarity of exposition, we move its proof to Appendix \ref{apdx:le:excessive_gap_aug_Lag_func}.

\begin{lemma}\label{le:excessive_gap_aug_Lag_func}
Let  $G_{\gamma\beta}$ be the smoothed gap function defined by  \eqref{eq:smoothed_gap_func} and $S_\beta(\xb;\dot{\yb}) := P_{\beta}(\xb;\dot{\yb}) - \Popt = f(\xb) +  g_\beta(\Ab\xb; \dot{\yb}) - \Popt$ be the smoothed objective residual.
Then, we have  
\begin{equation}\label{eq:S_bound}
S_\beta(\xb;\dot{\yb}) \leq G_{\gamma\beta}(\wb; \dot{\wb}) + \gamma b_{\Xc}(\xopt, \dot{\xb})~~~\text{and}~~~\frac{1}{2} \norm{\yb^{*}_\beta(\Ab\xb; \dot{\yb}) - \yopt}^2_{\Yc,\ast}  \leq  b_{\Yc}(\yopt, \dot y) + \frac{1}{\beta} S_\beta(\xb;\dot{\yb}). 
\end{equation}
Suppose that $g(\cdot) := \delta_{\set{\cb}}(\cdot)$. 
Then, for any $\yopt\in\Yopt$ and $\xb\in\Xc$, one has
\begin{equation}\label{eq:lower_bound}
-\Vert\yopt\Vert_{\Yc}\Vert\Ab\xb - \cb\Vert_{\Yc, *} \leq f(\xb) - \fopt
\end{equation}
and the following primal objective residual and feasibility gap estimates hold for \eqref{eq:constr_cvx}:
\begin{equation}\label{eq:main_bound_gap_3}
\left\{\begin{array}{ll}
f(\xb) - \fopt &\leq S_\beta(\xb;\dot{\yb}) -\iprods{\yopt, \Ab\xb -\cb} + \beta b_{\Yc}(\yopt, \dot{\yb}), \vspace{1ex}\\
\norm{\Ab\xb - \cb}_{\Yc, *} &\leq  \beta L_{b_{\Yc}} \Big[ \norm{\yopt - \dot{\yb}}_{\Yc}  + \big(\norm{\yopt - \dot{\yb}}_{\Yc}^2 + 2L_{b_{\Yc}}^{-1}\beta^{-1}S_{\beta}(\xb;\dot{\yb}) \big)^{1/2}\Big],
\end{array}\right.
\end{equation}
where the quantity in the square root is always nonnegative.
\end{lemma}

The estimates \eqref{eq:lower_bound} and \eqref{eq:main_bound_gap_3} are independent of optimization methods used to construct $\{\wbar^k\}$ for the primal-dual variable $\wb = (\xb, \yb)$.
However, their convergence guarantee  depends on the smoothness parameters $\gamma_k$ and $\beta_k$. 
Hence, the convergence rate of the objective residual $f(\xbar^k) - \fopt$ and feasibility gap $\norm{\Ab\xbar^k - \cb}_{\Yc,\ast}$ depends on the rate of $\{(\gamma_k, \beta_k)\}$. 

The first inequality in~\eqref{eq:S_bound} is more precise than what~\cite{Nesterov2005c} tells us. 
It holds even if $\Xc$ is unbounded and it shows that if $\dot x$ is close to $\xopt$, then the smoothed function is more accurate.
The second inequality in~\eqref{eq:S_bound} shows that the distance between $\yb^{\ast}_\beta(\Ab\xb; \dot{\yb})$ and $\yopt$ is controlled by quantities that will remain bounded. In practice, we 
observe that $\yb^{\ast}_\beta(\Ab\xb; \dot{\yb})$ cconverges to $\yopt$. 
Hence, restarting the algorithm with $\dot{\yb}^{\prime} = \yb^{\ast}_\beta(\Ab\xb; \dot{\yb})$ gives us a chance to accelerate the actual performance of our algorithms while does not hurt the convergence guarantee \cite{fercoq2016restarting}.

\section{The accelerated primal-dual gap reduction algorithm}\label{sec:alg_scheme}
Our new scheme builds upon Nesterov's acceleration idea \cite{Nesterov1983,Nesterov2004}.
At each iteration, we apply an accelerated proximal-gradient step to minimize $f + g_{\beta}$. 
Since $f + g_{\beta}$ is nonsmooth, we use the proximal operator of $f$ to generate a proximal-gradient step.
As a key feature, we must update the parameters $\tau_k$ and $\beta_k$ simultaneously at each iteration with analytical updating formulas.

\subsection{The  method}
Let $\xbar^k \in\Xc$ and $\xtilde^k \in\Xc$ be given.
The \textit{Accelerated Smoothed GAp ReDuction} (ASGARD) scheme generates a new point $(\xbar^{k+1}, \xtilde^{k+1})$ as
\begin{equation}\label{eq:ac_pd_scheme}
 \left\{\begin{array}{lll}
&\xhat^k            &:= (1-\tau_k)\xbar^k + \tau_k\xtilde^k, \vspace{0.5ex}\\
&\yast_{\beta_{k+1}}(\Ab\hat{\xb}^k;\dot{\yb}) & := \arg\displaystyle\max_{\yb \in \Yc}\set{ \iprods{\Ab \xhat^k, \yb} - g^{*}(\yb) - \beta_{k+1} b_{\Yc}(\yb, \dot{\yb}) },\\
&\xbar^{k+1}      &:= \prox_{\beta_{k+1}\bar L_\Ab^{-1} f}\left(\xhat^k - \beta_{k+1}\bar L_\Ab^{-1}  \Ab^{\top}\yast_{\beta_{k+1}}(\Ab\hat{\xb}^k;\dot{\yb}) \right), \vspace{0.75ex}\\
&\xtilde^{k+1}   &:= \xtilde^k -  \tau_k^{-1} (\xhat^k - \xbar^{k+1}). 
\end{array}\right.
\tag{\textrm{ASGARD}}
\end{equation}
where  $\tau_k \in (0, 1]$ and $\beta_k > 0$ are parameters that will be defined in
the sequel.
The constant $\bar L_A$ is defined as 
\begin{equation}\label{eq:LA_def}
\bar{L}_{\Ab} := \norm{\Ab}^2 = \displaystyle\max_{\xb \in \R^n} \Big\{ \frac{\norm{\Ab\xb}_{\Yc, *}^2}{\norm{\xb}_\Xc^2}\Big\}.
\end{equation}
The \ref{eq:ac_pd_scheme} scheme requires a mirror step with the conjugate $g^{*}$ of $g$ to get $\yb^{\ast}_{\beta_{k+1}}(\cdot; \dot{\yb})$ in \eqref{eq:yast_beta} and a proximal step of $f$ in the third line. 

The following lemma shows that $\xbar^{k+1}$ updated by \ref{eq:ac_pd_scheme} decreases the smoothed objective residual $P_{\beta_k}(\xbar^k;\dot{\yb}) - \Popt$,  
whose proof can  be found in Appendix \ref{apdx:lem:choose_parameters_asgard}.

\begin{lemma} \label{lem:choose_parameters_asgard}
Let us choose $\tau_0 := 1$. 
If $\tau_k \in (0, 1)$ is the unique positive root of the cubic polynomial equation $p_3(\tau) := \tau^3/L_{b_{\Yc}} + \tau^2 +  \tau_{k-1}^2 \tau -  \tau_{k-1}^2 = 0$ for $k\geq 1$, and $\beta_k := \frac{\beta_{k-1}}{1+\tau_{k-1}/L_{b_{\Yc}}}$, then $\beta_k = \mathcal{O}\big(\frac{1}{k^{1/L_{b_{\mathcal Y}}}}\big)$ as $k\to\infty$, and
\begin{equation}\label{eq:key_est_a0}
P_{\beta_{k+1}}(\xbar^{k+1};\dot{\yb}) - \Popt + \frac{\tau_k^2}{\beta_{k\!+\!1}} \frac{\bar{L}_{\Ab}}{2} \norm{\tilde{x}^{k\!+\!1} \!\!- \xopt}_{\Xc}^2 \leq \frac{\tau_k^2}{\beta_{k\!+\!1}} \frac{\bar{L}_\Ab}{2} \norm{\tilde{\xb}^0 - \xopt}^2_{\Xc}. 
\end{equation}
 Moreover, if  $L_{b_{\Yc}} = 1$, then $\frac{1}{k+1} \leq \tau_k \leq \frac{2}{k+2}$, $\frac{\tau_k^2}{\beta_{k+1}} \leq \frac{\tau_0^2}{\beta_1(k+1)} = \frac{1}{\beta_1(k+1)}$, and  $\beta_k \leq \frac{2 \beta_1}{k+1}$. 
\end{lemma}

\subsection{The primal-dual algorithmic template}\label{subsec:algorithm1}
Similar to the accelerated scheme \cite{Beck2009,Nesterov1983}, we can eliminate  $\xtilde^k$ in \ref{eq:ac_pd_scheme} by combining its first line and last line to obtain
\begin{equation*}
\xhat^{k+1} = \xbar^{k+1} + \frac{(1-\tau_k)\tau_{k+1}}{\tau_k}(\xbar^{k+1}-\xbar^k).
\end{equation*}
Now, we combine all the ingredients presented previously and this step to obtain a primal-dual algorithmic 
template for solving \eqref{eq:primal_cvx} as in Algorithm \ref{alg:pd_alg1} below.

\begin{algorithm}[!ht]\caption{(\textit{Accelerated Smoothed GAp ReDuction \eqref{eq:ac_pd_scheme} algorithm})}\label{alg:pd_alg1}
\begin{normalsize}
\begin{algorithmic}[1]
\Statex{\hskip-4ex}\textbf{Initialization:} 
\State Choose $\beta_1 > 0$ (e.g., $\beta_1 := 0.5\sqrt{\bar{L}_{\Ab}}$, where $\bar{L}_{\Ab}$ is given in \eqref{eq:LA_def}) and set $\tau_0 := 1$. 
\State Choose $\xbar^0 \in \Xc$ arbitrarily, and set $\xhat^0 := \xbar^0$.
\Statex{\hskip-4ex}\textbf{For}~{$k=0$ {\bfseries to} $k_{\max}$, \textbf{perform:}}
\State\label{stp:update_tau} Compute $\tau_{k+1} \in (0, 1)$ the unique positive root of $\tau^3/L_{b_{\Yc}} + \tau^2 +  \tau_{k}^2 \tau -  \tau_{k}^2 = 0$. 
\State\label{stp:ystar} Compute the dual step by solving
\begin{equation*}
\yast_{\beta_{k+1}}(\Ab\xhat^k; \dot{\yb}) := \arg\max_{\hat{\yb} \in \Yc}\set{  \iprods{\Ab \xhat^k, \hat{\yb}} - g^{\ast}(\hat{\yb}) - \beta_{k+1} b_{\Yc}(\hat{\yb}, \dot{\yb}) }.
\end{equation*}
\State\label{stp:xprox} Compute the primal step $\xbar^{k+1}$ using the $\prox_{f}$ of $f$ as
\begin{equation*}
\xbar^{k+1} :=  \prox_{\beta_{k+1}\bar{L}_\Ab^{-1} f}\left(\xhat^k - \beta_{k+1}\bar{L}_\Ab^{-1}  \Ab^{\top}\yast_{\beta_{k+1}}(\Ab\xhat^k; \dot{\yb})\right).
\end{equation*}
\State\label{stp:update_xbeta} Update $\xhat^{k+1} = \xbar^{k+1} + \frac{\tau_{k+1}(1-\tau_k)}{\tau_k}(\xbar^{k+1}-\xbar^k)$ and $\beta_{k+2} := \frac{\beta_{k+1}}{1 + L_{b_{\Yc}}^{-1} \tau_{k+1}}$.
\Statex{\hskip-4ex}\textbf{End~for}
\end{algorithmic}
\end{normalsize}
\end{algorithm}

\paragraph{Per-iteration complexity of Algorithm~\ref{alg:pd_alg1}}
The computationally heavy steps of Algorithm \ref{alg:pd_alg1} are Steps \ref{stp:ystar} and \ref{stp:xprox}.
The per-iteration complexity of Algorithm~\ref{alg:pd_alg1} consists of 
\begin{itemize}
\item One matrix-vector multiplication $\Ab\xb$, and one mirror step of $g^{\ast}$ at Step \ref{stp:ystar} to compute $\yast_{\beta_{k+1}}(\Ab\xhat^k;\dot{\yb})$. 
If $g(\cdot) := \delta_{\set{\cb}}(\cdot)$ and $p_{\Yc}(\cdot) := (1/2)\norm{\cdot}_2^2$, then $\yast_{\beta_{k+1}}(\Ab\xhat^k;\dot{\yb}) = \dot{\yb} + \beta_{k+1}^{-1}(\Ab\xhat^k - \cb)$, which only requires one matrix-vector multiplication $Ax$.
\item One adjoint matrix-vector multiplication $\Ab^{\top}\yb$, and one proximal step of $f$ at Step \ref{stp:xprox}.
If $f$ is decomposable, evaluating $\prox_f$ can be  implemented \textit{in parallel}.
\end{itemize}
We note that if  $p_{\Yc}(\cdot) := (1/2)\norm{\cdot}_2^2$, the the mirror step in $g^{\ast}$ becomes a proximal step $\prox_{g^{\ast}}$.

\subsection{Convergence analysis}
Our first main result is the following two theorems, which show an $\mathcal{O}(1/k)$ convergence rate of Algorithm~\ref{alg:pd_alg1} for both the unconstrained problem \eqref{eq:primal_cvx} and the constrained setting \eqref{eq:constr_cvx}.

\begin{theorem}\label{th:convergence_accPD_alg_const}
Suppose that $g = \delta_{\set{\cb}}$. 
Let $\beta_1 > 0$ and $b_{\Yc}$ be chosen such that $L_{b_{\Yc}} = 1$.
Let $\{\xbar^k\}$  be the primal sequence generated by Algorithm~\ref{alg:pd_alg1}.
Then the following bounds hold for \eqref{eq:constr_cvx}:
\begin{equation}\label{eq:convergence_accPD1}
{\!\!\!\!\!}\left\{\begin{array}{ll}
f(\xbar^k) \!-\! f^{\star} &\geq - \norm{\yopt}_\Yc \norm{A \bar x^k - \cb}_{\Yc,*},\vspace{1ex}\\
f(\xbar^k) \!-\! f^{\star} &\leq \frac{1}{k}\frac{\bar{L}_\Ab}{2\beta_1} \norm{\bar{\xb}^0 - \xopt}^2_{\Xc} + \norm{\yopt}_{\Yc} \norm{\Ab \bar{\xb}^k -\cb}_{\Yc,*} + \frac{2\beta_1}{k+1}b_{\Yc}(\yopt, \dot{\yb}),\vspace{1ex}\\
 \norm{\Ab \bar{\xb}^k -\cb}_{\Yc,*}  &\leq  \frac{\beta_1}{k+1}\Big[ \norm{\yopt - \dot{\yb}}_{\Yc}  + \big(\norm{\yopt - \dot{\yb}}_{\Yc}^2 + \beta_1^{-2}\bar{L}_\Ab \norm{\bar{\xb}^0 - \xopt}^2_{\Xc} \big)^{1/2}\Big].
\end{array}\right.
\end{equation}
\end{theorem}

\begin{proof}
If $g = \delta_{\{\cb\}}$, then we apply Lemma~\ref{le:excessive_gap_aug_Lag_func} and use the bound on the  smoothed optimality gap given by Lemma~\ref{lem:choose_parameters_asgard} with noting that $\tilde x^0 = \bar x^0 = \hat x^0$ to get the bounds as in \eqref{eq:main_bound_gap_3}.
Moreover, since $\beta_k \leq \frac{2\beta_1}{k+1}$ and 
\begin{equation*}
\frac{\tau_k^2}{\beta_{k+1}^2} = \frac{\tau_k^4}{\beta_{k+1}^2}\frac{1}{\tau_k^2} \leq 
\Big(\frac{1}{\beta_1 (k+1)} \Big)^2 (k+1)^2 = \frac{1}{\beta_1^2},
\end{equation*}
using these estimates into the resulting bounds, we obtain \eqref{eq:convergence_accPD1}.
\end{proof}
Note that if we choose $\dot{\yb} := \boldsymbol{0}^m$ and $b_\Yc(y,\dot y) = \frac 12 \norm{y - \dot y}^2_\Yc$, then the bounds \eqref{eq:convergence_accPD1} can be further simplified as
\begin{equation}\label{eq:simple_bound1}
{\!\!\!\!}\left\{\!\!\begin{array}{ll}
\abs{f(\xbar^k) \!-\! \fopt} &{\!\!\!\!\!\!\!}\leq \frac{1}{k}\left(\frac{L_{\Ab}}{2\beta_1}\norm{\xbar^0 \!-\! \xopt}_{\Xc}^2 + 3\beta_1\norm{\yopt}_{\Yc}^2 + \frac{\sqrt{L_{\Ab}}}{\beta_1}\norm{\xbar^0 \!-\! \xopt}_{\Xc}\norm{ \yopt}_{\Yc}\right),\vspace{1ex}\\
\Vert\Ab\xbar^k \!-\! \cb\Vert_{\Yc,\ast} &{\!\!\!\!\!\!\!}\leq \frac{\beta_1}{k+1}\left(2\norm{\yopt}_{\Yc} + \frac{\sqrt{L_{\Ab}}}{\beta_1}\norm{\xbar^0 - \xopt}_{\Xc}\right).
\end{array}\right.{\!\!\!\!\!\!\!\!\!}
\end{equation}
Clearly, the choice of $\beta_1$ in Theorem~\ref{th:convergence_accPD_alg_const} trades off between $\norm{\bar{\xb}^0 - \xopt}^2_{\Xc}$ and $\norm{\yopt - \dot{\yb}}_{\Yc}^2$ on the primal objective residual $f(\xbar^k) -  f^{\star}$ and on the feasibility gap $\Vert\Ab\xbar^k - \cb\Vert_{\Yc,\ast}$.

\begin{theorem}\label{th:convergence_accPD_alg}
Suppose that $g$ is Lipschitz continuous and $\dot y = \bar y^c$ is fixed. Then, $D_{\Yc} := \sup_{y}\set{ b_{\Yc}(y, \dot{\yb}) \mid y \in\dom{g^{\ast}}} < +\infty$ $($i.e., $D_{\Yc}$ is bounded$)$. 
Let $\beta_1 > 0$ and $b_{\Yc}$ be chosen such that $L_{b_{\Yc}} = 1$.
Let $\{\xbar^k\}$  be the primal sequence generated by Algorithm~\ref{alg:pd_alg1}.
Then, the primal objective residual of \eqref{eq:primal_cvx} satisfies
\begin{equation}\label{eq:convergence_accPD1a}
P(\xbar^k) - \Popt \leq \frac{\bar{L}_{\Ab}}{2\beta_1 k}\norm{\xbar^0 - \xopt}_{\Xc}^2 + \frac{2\beta_1}{k+1}D_\Yc, ~~\text{for all}~k\geq 1.
\end{equation}
\end{theorem}

\begin{proof}
If $g$ is Lipschitz continous, then, for all $x\in\Xc$, $\partial g(\Ab \xb) \neq \emptyset$ and $\dom{g^*}$ is bounded. 
For $y \in \dom{g^*}$, we have $b_{\Yc}(y, \dot{\yb}) \leq D_\Yc < +\infty$.  
Let $\yb^{\ast}(\xb)\in\partial{g}(\Ab\xb)$. Then, we can show that
\begin{align*}
g(\Ab\xb) &= \iprods{\Ab\xb,\yb^{\ast}(\xb)} - g^{\ast}(\yb^{\ast}(\xb)) \leq \max_{\yb\in\Yc}\set{\iprods{\Ab\xb,\yb} - g^{\ast}(\yb) - \beta b_{\Yc}(\yb,\dot{\yb})} + \beta b_{\Yc}(\yb^{\ast}(\xb), \dot{\yb})\\
&  \leq g_{\beta}(\Ab\xb) + \beta b_{\Yc}(\yb^{\ast}(\xb), \dot{\yb}) \leq g_{\beta}(\Ab\xb) + \beta D_\Yc.
\end{align*}
Therefore, the bound \eqref{eq:convergence_accPD1a} follows directly from \eqref{eq:key_est_a0} and this inequality.
\end{proof}

\begin{remark}\label{re:optimal_choice_of_beta_1}
If, in addition to $g^{*}$, $f$ also has a bounded domain, we recover the assumptions of~\cite{Nesterov2005c}. 
By choosing $\bar{\xb}^c := \bar{\xb}^0$, and  $\beta_1 := \sqrt{\bar L_A \frac{D_\Xc}{2D_\Yc}}$, we get a convergence  bound as $P(\bar{\xb}^k) - P^{\star} \leq \frac{2\sqrt{2}\sqrt{\bar L_A D_\Xc D_\Yc }}{k}$.
The worst-case convergence rate bound \eqref{eq:convergence_accPD1a} is the same as the one in \cite{Nesterov2005c} $($up to a small constant factor$)$. 
However, $\beta_1$ does not depend on the tolerance $\varepsilon$ as in \cite{Nesterov2005c}.
\end{remark}

\begin{remark}\label{re:update_tau}
When Algorithm~\ref{alg:pd_alg1} is applied to solve the constrained convex problem~\eqref{eq:constr_cvx} using $p_{\Yc}(\cdot) := \frac{1}{2}\Vert\cdot\Vert_2^2$, we can simplify the update rule for $\tau_k$ at Step~\ref{stp:update_tau} and $\beta_k$ at Step~\ref{stp:update_xbeta} as follows:
\begin{equation}\label{eq:update_tau_new}
\beta_{k+1} := (1-\tau_k)\beta_k,~~~~~\text{and}~~~~\tau_{k+1} := \frac{\tau_k}{\tau_k + 1} = \frac{1}{k+2}.
\end{equation}
This update rule does not improve the worst-case convergence guarantee in Theorem~\ref{th:convergence_accPD_alg}, but it is simple.
The detail analysis can be found in Appendix~\ref{apdx:new_update_tau}.
\end{remark}

\section{The accelerated dual smoothed gap reduction method}\label{sec:algorithm3}
Algorithm~\ref{alg:pd_alg1} can be viewed as an accelerated proximal scheme applying to minimize the function $P_{\gamma}(\cdot;\dot{\yb})$ defined in \eqref{eq:F_beta}.
Now, we exploit the smoothed gap function $G_{\gamma\beta}$ defined by \eqref{eq:smoothed_gap_func} to develop a novel primal-dual method for solving \eqref{eq:primal_cvx} and \eqref{eq:dual_cvx}.
Our goal is to design a new scheme to compute a primal-dual sequence $\{\wbar^k\}$ and a parameter sequence $\set{(\gamma_k, \beta_k)}$ such that  $\max\set{0, G_{\gamma_k\beta_k}(\wbar^k;\dot{\wb})}$ converges to zero.

\subsection{The method}
Given $\bar{\wb}^k := (\xbar^k,\ybar^k) \in\Wc$, we derive a scheme to compute a new point $\bar{\wb}^{k+1} := (\xbar^{k+1}, \ybar^{k+1})$ as follows:
\begin{equation}\label{eq:pd_scheme_2d}
\left\{\begin{array}{ll}
\yhat^k        &  := (1-\tau_k)\ybar^k + \tau_k\yb^{\ast}_{\beta_k}(\Ab\xbar^k;\dot{\yb}),\vspace{0.75ex}\\
\ybar^{k  + 1} & := \prox_{\gamma_{k+1}\bar{L}_{\Ab}^{-1}g^{\ast}}\left( \yhat^k +  \gamma_{k+1}\bar{L}_{\Ab}^{-1}\Ab\xb^{\ast}_{\gamma_{k+1}}(\yhat^k; \dot{\xb})\right),\vspace{0.75ex}\\
\xbar^{k + 1} & := (1  - \tau_k)\xbar^k + \tau_k\xb^{\ast}_{\gamma_{k+1}}(\yhat^k;\dot{\xb}),
\end{array}\right.\tag{\textrm{ADSGARD}}
\end{equation}
where $\tau_k \in (0, 1)$ and the parameters $\beta_k > 0$ and $\gamma_{k+1} > 0$ will be updated in the sequel.
The points  $\xb^{\ast}_{\gamma_{k+1}}(\yhat^k;\dot{\xb})$ and $\yb^{\ast}_{\beta_k}(\Ab\xbar^k;\dot{\yb})$ are computed by \eqref{eq:x_ast} and \eqref{eq:yast_beta}, respectively.
This scheme requires one primal step for $\xb^{\ast}_{\gamma_{k+1}}(\yhat^k;\dot{\xb})$, one dual step for $\yb^{\ast}_{\beta_k}(\Ab\xbar^k;\dot{\yb})$, and one dual proximal-gradient step for $\ybar^{k+1}$.
Since the accelerated step is applied to $g_{\gamma}$, we call this scheme the  \textit{Accelerated Dual Smoothed GAp ReDuction} \eqref{eq:pd_scheme_2d} scheme.

The following lemma, whose proof is in Appendix \ref{apdx:le:maintain_excessive_gap2}, shows that $\wbar^{k+1}$ updated by \ref{eq:pd_scheme_2d} decreases the smoothed gap $G_{\gamma_k\beta_k}(\wbar^k)$ with at least a factor of $(1-\tau_k)$.

\begin{lemma}\label{le:maintain_excessive_gap2}
Let   $\wbar^{k+1} := (\xbar^{k+1}, \ybar^{k+1})$ be updated by the \ref{eq:pd_scheme_2d} scheme.
Then, if $\tau_k \in (0, 1]$, $\beta_k$ and $\gamma_k$ are chosen such that $\beta_1 \gamma_1 \geq \bar L_\Ab$ and 
\begin{equation}\label{eq:pd_condition2}
\big(1 + \tau_k/L_{b_{\Xc}}\big)\gamma_{k\!+\!1} \geq \gamma_k, ~~~~~~~ \beta_{k\!+\!1} \geq (1 \!-\! \tau_k)\beta_k,  ~~~~\text{and}~~~~ \frac{\bar L_\Ab}{\gamma_{k+1}} \leq \frac{(1-\tau_k) \beta_k}{\tau_k^2},
\end{equation}
then  $\wbar^{k+1} \in \Wc$ and satisfies $G_{\gamma_{k\!+\!1}\beta_{k\!+\!1}}(\wbar^{k\!+\!1};\dot{\wb}) \leq (1-\tau_k)G_{\gamma_k\beta_k}(\wbar^k;\dot{\wb})\leq 0$.

Let $\tau_0 := 1$. Then, for all $k\geq 1$, if we choose $\tau_k\in (0, 1)$ to be the unique positive solution of the cubic equation $p_3(\tau) := \tau^3/L_{b_{\Xc}} + \tau^2 + \tau_{k-1}^2\tau - \tau_{k-1}^2 = 0$, then $\frac{1}{k+1} \leq \tau_k \leq \frac{2}{k+2}$ for $k\geq 1$.
The parameters $\beta_k$ and $\gamma_k$ computed by $\beta_1 \gamma_1 = \bar L_\Ab$ and 
\begin{equation}\label{eq:update_gamma_beta}
\gamma_{k+1} := \frac{\gamma_k}{1 + \tau_k/L_{b_{\Xc}}}~~~\text{and}~~~\beta_{k+1} := (1-\tau_k)\beta_k,
\end{equation}
satisfy the conditions in \eqref{eq:pd_condition2}.

In addition, if $L_{b_{\Xc}} = 1$, then $\gamma_k \leq \frac{2\gamma_1}{k+1}$ and $\frac{\bar L_\Ab}{2\gamma_1(k+1)} \leq \beta_{k+1} \leq \frac{\beta_1}{k+1}$ for $k\geq 1$.
\end{lemma}

\subsection{The primal-dual algorithmic template}\label{subsec:algorithm}
We combine all the ingredients presented in the previous subsection to obtain a primal-dual algorithmic template for solving either \eqref{eq:primal_cvx} or \eqref{eq:constr_cvx} as shown in Algorithm \ref{alg:pd_alg2}.

\begin{algorithm}[!ht]\caption{(\textit{Accelerated Dual Smoothed GAp ReDuction  \eqref{eq:pd_scheme_2d}})}\label{alg:pd_alg2}
\begin{normalsize}
\begin{algorithmic}[1]
\Statex{\hskip-4ex}\textbf{Initialization:} 
\State Choose $\gamma_1 > 0$ (e.g., $\gamma_1 := \sqrt{\bar{L}_{\Ab}}$, where $\bar{L}_{\Ab}$ is given by \eqref{eq:LA_def}). Set  $\beta_1 := \frac{ \bar{L}_{\Ab}}{\gamma_1}$ and $\tau_0 := 1$. 
\State Take  an initial point $\bar{\yb}^{\ast}_0 := \dot{\yb} \in\Yc$.
\Statex{\hskip-4ex}\textbf{For}~{$k=0$ {\bfseries to} $k_{\max}$, \textbf{perform:}}
\State\label{eq:update_yhat}Update $\yhat^k := (1 - \tau_k)\ybar^k + \tau_k\ybar^{\ast}_k$.
\State\label{eq:2d_step5}Compute $\xhat^{\ast}_{k+1}$ \textit{in parallel} with
\begin{equation*} 
\xhat^{\ast}_{k+1} := \displaystyle\argmin_{\xb\in\Xc}\big\{f(\xb) + \iprods{\Ab^{\top}\yhat^k, \xb} + \gamma_{k+1} b_{\Xc}(\xb,\dot{\xb}) \big\}.
\end{equation*}
\State\label{eq:2d_step6}Update the dual vector
\begin{equation*}
\ybar^{k  + 1} := \prox_{\gamma_{k+1} \bar L_\Ab^{-1} g^{\ast}}\left( \yhat^k +  \gamma_{k+1} \bar L_\Ab^{-1} \Ab\xhat^{\ast}_{k+1}\right).
\end{equation*}
\State\label{eq:2d_step7}Update the primal vector: $\xbar^{k+1} := (1 - \tau_k)\xbar^k + \tau_k\xhat^{\ast}_{k+1}$.
\State\label{eq:2d_step3}Compute 
\begin{equation*}
\ybar^{\ast}_{k+1} := \mathrm{arg}\max_{\yb\in\Yc} \set{ \iprods{\Ab\xbar^{k+1}, \yb}  - g^{\ast}(\yb) - \beta_{k+1} b_{\Yc}(\yb,\dot{\yb}) }. 
\end{equation*}
\State Compute $\tau_{k+1} \in (0, 1)$ the unique positive root of $\tau^3/L_{b_{\Yc}} + \tau^2 +  \tau_{k}^2 \tau -  \tau_{k}^2 = 0$. 
\State Update $\gamma_{k+2} := \frac{\gamma_{k+1}}{1 + L_{b_{\Xc}}^{-1} \tau_{k+1}}$, and  $\beta_{k+2} := (1-\tau_{k+1})\beta_{k+1}$. 
\vspace{0.5ex}
\Statex{\hskip-4ex}\textbf{End for}
\end{algorithmic}
\end{normalsize}
\end{algorithm}

Since $\tau_0 = 1$, Step~\ref{eq:update_yhat} shows that $\hat{\yb}^0 = \bar{\yb}^{\ast}_0$, and while Step~\ref{eq:2d_step7} leads to $\bar{\xb}^1 = \hat{\xb}^{\ast}_1$.
The main steps of Algorithm \ref{alg:pd_alg2} are Steps \ref{eq:2d_step5}, \ref{eq:2d_step6} and~\ref{eq:2d_step3}, where we need to solve the subproblem \eqref{eq:x_ast}, and to update two dual steps, respectively.
The first dual step requires the proximal operator $\prox_{\rho g^{\ast}}$ of $g^{\ast}$, while the second one computes $\ybar^*_{k+1} = \yb^{\ast}_{\beta_{k+1}}(\Ab\xbar^{k+1};\dot{\yb})$.

When $g = \delta_{\set{\cb}}$, the indicator of $\set{\cb}$ in the constrained problem \eqref{eq:constr_cvx},  we have 
\begin{equation*}
\yb^{\ast}_{\beta_k}(\Ab\xbar^k;\dot{\yb}) = \nabla{b^{*}_{\Yc}}\left(\beta_k^{-1}(\Ab\xbar^k - \cb), \dot{\yb}\right)~~\text{and}~~\ybar^{k+1} := \yhat^k +  \gamma_{k+1}\left(\Ab\xb^{\ast}_{\gamma_{k+1}}(\yhat^k; \dot{\xb}) - \cb\right).
\end{equation*}
The first dual step only requires one matrix-vector multiplication $\Ab\xb$.
Clearly, by Step~\ref{eq:2d_step7}, it follows that $\Ab\xbar^{k+1} - \cb = (1-\tau_k)(\Ab\xbar^k - \cb) + \tau_k(\Ab\xhat^{\ast}_{k+1} - \cb)$, 
and by Step~\ref{eq:2d_step3}, we have $\ybar^{\ast}_k = \yb^{\ast}_{\beta_k}(\Ab\xbar^k;\dot{\yb})=  \nabla{b^{\ast}_{\Yc}}\left(\beta_k^{-1}(\Ab\xbar^k - \cb),\dot{\yb}\right)$, which is equivalent to $\Ab\xbar^k - \cb = \beta_k\nabla{b_{\Yc}}(\ybar^{\ast}_k,\dot{\yb})$.
Hence, $\Ab\xbar^{k+1} - \cb = (1-\tau_k)\beta_k\nabla{b}_{\Yc}(\ybar^{\ast}_k,\dot{\yb}) + \frac{\tau_k}{\gamma_{k+1}}(\ybar^{k+1} - \yhat^k)$ due to Step~\ref{eq:2d_step6}.
Finally, we can derive an update rule for $\bar{\yb}^{\ast}_{k+1}$ as
\begin{equation}\label{eq:update_yopt_k}
\ybar^{*}_{k+1} := \nabla{b_{\Yc}^{*}}\Big(\beta_{k+1}^{-1}\big( (1-\tau_k)\beta_k\nabla{b}_{\Yc}(\ybar^{\ast}_k, \dot{\yb})  + \frac{\tau_k}{\gamma_{k+1}}(\ybar^{k+1} - \yhat^k)\big), \dot{\yb}\Big).
\end{equation}

\paragraph{Per-iteration complexity of Algorithm~\ref{alg:pd_alg2}}
From the above analysis, we can conclude that the per-iteration complexity of Algorithm~\ref{alg:pd_alg1} consists of
\begin{itemize}
\item One adjoint matrix-vector multiplication $\Ab^{\top}\yb$, and one mirror step in $f$ at Step \ref{eq:2d_step5} to compute $\xhat^{\ast}_{k+1}$.
If $f$ is decomposable, then Step~\ref{eq:2d_step5} can be  implemented \textit{in parallel}.
\item One matrix-vector multiplication $\Ab\xb$, one proximal step of $g^{\ast}$ at Step \ref{eq:2d_step6} to compute $\ybar^{k  + 1}$, and one mirror step of $g^{\ast}$ at Step~\ref{eq:2d_step3} to compute $\ybar^{\ast}_{k+1}$. 
If $g = \delta_{\set{\cb}}$ and $p_{\Yc}(\cdot) := (1/2)\norm{\cdot}_2^2$, then computing $\ybar^{k  + 1}$ using \eqref{eq:update_yopt_k}  requires only one $Ax$.
\end{itemize}

\subsection{Convergence analysis}
The following theorem shows the convergence of Algorithm~\ref{alg:pd_alg2}.
For the constrained setting \eqref{eq:constr_cvx}, we still have the lower bound on $f(\xbar^k) - \fopt$ as in Theorem \ref{th:convergence_accPD_alg}, i.e. $-\norm{\yopt}_{\Yc}\Vert\Ab\xbar^k \!-\! \cb\Vert_{\Yc,\ast} \leq f(\xbar^k) - f^{\star}$ for any $\xbar^k\in\Xc$ and $\yopt\in\Yopt$.

\begin{theorem}\label{th:convergence_A1_const}
Suppose that $g= \delta_{\set{\cb}}$.
Let $b_{\Xc}$ be chosen such that $L_{b_{\Xc}} = 1$, and  $\{\wbar^k\}$  be the sequence generated by Algorithm \ref{alg:pd_alg2} for solving \eqref{eq:constr_cvx}, where $\gamma_1 > 0$ is given.
Then, the following bounds for \eqref{eq:constr_cvx} hold:
\begin{equation}\label{eq:convergence_A2b}
{\!\!\!}\left\{\begin{array}{ll}
f(\xbar^k) \!-\! f^{\star} &\geq - \norm{\yopt}_\Yc \norm{A \bar x^k - \cb}_{\Yc,*},\vspace{1ex}\\
f(\xbar^k) - f^{\star} &\leq \frac{2\gamma_1}{k+1}b_{\Xc}(\xb^{\star},\dot{\xb}) + \frac{\bar{L}_{\Ab}}{\gamma_1k}b_{\Yc}(\yopt,\dot{\yb}) + \norm{\yopt}_{\Yc}\norm{\Ab\xbar^k - \cb}_{\Yc,\ast},\vspace{1ex}\\
\Vert\Ab\xbar^k \!- \cb\Vert_{\Yc,\ast}  &\leq \frac{\bar{L}_{\Ab}}{\gamma_1k} L_{b_{\Yc}} \Big[ \norm{\yopt - \dot{\yb}}_{\Yc}  + \big(\norm{\yopt - \dot{\yb}}_{\Yc}^2 + \frac{8\gamma_1^2}{L_{b_{\Yc}}}b_{\Xc}(\xb^{\star},\dot{\xb}) \big)^{1/2}\Big].
\end{array}\right.{\!\!\!}
\end{equation}
\end{theorem}

\begin{proof}
This set of inequalities is a consequence of Lemmas~\ref{le:excessive_gap_aug_Lag_func} and~\ref{le:maintain_excessive_gap2}  using $\beta_k \leq \frac{\beta_1}{k}$, $\gamma_k\leq \frac{2\gamma_1}{k+1}$ and $\frac{\gamma_k}{\beta_k} \leq \frac{4\gamma_1^2k}{k+1} \leq 4\gamma_1^2$.
\end{proof}

\begin{theorem}\label{th:convergence_A1}
Suppose that $g$ is Lipschitz continuous as in Theorem~\ref{th:convergence_accPD_alg}.
Let $b_{\Xc}$ be chosen such that $L_{b_{\Xc}} = 1$, and  $\{\wbar^k\}$  be the sequence generated by Algorithm \ref{alg:pd_alg2} for solving \eqref{eq:primal_cvx}, where $\gamma_1 > 0$ is given.
Then, the following convergence bound holds
\begin{equation}\label{eq:convergence_2d_Px}
P(\xbar^k) - \Popt \leq \frac{2\gamma_1}{k+1}b_{\Xc}(\xb^{\star}, \dot{\xb}) + \frac{2\bar{L}_{\Ab}}{\gamma_1 k}D_\Yc.
\end{equation}
\end{theorem}

\begin{proof}
Since $S_{\beta}(\xb;\dot{\yb}) \leq G_{\gamma\beta}(\wb;\dot{\wb}) + \gamma b_{\Xc}(\xopt,\dot{\xb})$, using Lemma~\ref{le:maintain_excessive_gap2} we can show that $S_{\beta_k}(\xbar^k;\dot{\yb}) \leq G_{\gamma_k\beta_k}(\wbar^k;\dot{\wb}) + \gamma_k b_{\Xc}(\xopt, \dot{\xb}) \leq \gamma_k b_{\Xc}(\xopt,\dot{\xb})$.
Similar to the proof of Theorem~\ref{th:convergence_accPD_alg}, we obtain the bound \eqref{eq:convergence_2d_Px} for the objective residual of \eqref{eq:primal_cvx}.
\end{proof}

Similar to Theorem~\ref{th:convergence_accPD_alg}, we can simplify the bound \eqref{eq:convergence_A2b} to obtain a simple bound as in \eqref{eq:simple_bound1}, where we omit the details here.
The choice of $\gamma_1$ and $\beta_1$ in Theorem \ref{th:convergence_A1} also trades off the primal objective residual and the primal feasibility gap.

\subsection{The choice of smoothers}
For this algorithm, one needs to choose a norm $\norm{\cdot}_\Xc = \norm{\cdot}_\Sb$ and a smoother $p_\Xc$ such that $p_{\Xc}$ is strongly convex with respect to the norm  $\norm{\cdot}_\Sb$.
One possibility is to choose $\norm{\cdot}_\Sb$ in order to have a simple formula for $\hat{\xb}^{\ast}_{k+1} =  \xb^{*}_\gamma(\hat y^k; \dot x)$. 
A classical choice is a diagonal $\Sb$ and $b_\Xc(\cdot,\dot{\xb}) = \frac 12 \norm{\cdot - \dot{\xb}}_\Sb^2$ is a quadratic function for a given $\dot{\xb} \in\Xc$.

If  $f$ is decomposable as $f(\xb) = \sum_{i=1}^Nf_i(\xb_i)$ and we choose $b_{\Xc}(\xb, \dot{\xb}) := \sum_{i=1}^Nb_{\Xc_i}(\xb_i, \dot{\xb}_i)$, 
then the computation of $\hat{\xb}^{\ast}_{k+1}$ at Step~\ref{eq:2d_step5} of Algorithm~\ref{alg:pd_alg2} can be carried out in \textit{parallel}. 

Another possibility is to choose $\Sb = \Ab$ and $p_\Xc(\cdot) = \frac 12 \norm{\cdot}_\Sb^2$.
In that case, the computation of $x^*_{k+1}$ may require an iterative sub-solver 
but we are allowed to take $\dot x = \xopt$. Indeed, as $A\xopt = c$, we have that
for all $\xb$, $b_\Xc(x, \xopt) = \frac 12 \norm{x - \xopt}_{\Ab}^2 = \frac 12 (Ax - c)^\top (A x -c)$. Hence, we can consider $\xopt$ as a center even though we do not know it.
We shall develop the consequences of such a choice in the Section~\ref{subsec:alsmoother}.

\section{Special instances of  the primal-dual gap reduction framework}\label{sec:variants}
We specify our \ref{eq:pd_scheme_2d} scheme to handle two special cases: augmented Lagrangian method and strongly convex objective.
Then, we provide an extension of our algorithms to a general cone constraint.

\subsection{Accelerated smoothing augmented Lagrangian gap reduction method}
\label{subsec:alsmoother}
The augmented Lagrangian (AL) method is a classical optimization technique, and has widely been used in various applications due to its emergingly practical performance.
In this section, we customize  Algorithm~\ref{alg:pd_alg2} using \ref{eq:pd_scheme_2d} to solve the constrained convex problem \eqref{eq:constr_cvx}.
The inexact variant of this algorithm can be found in our early technical report \cite[Section 5.3]{Tran-Dinh2014a}.

\paragraph{The augmented Lagrangian smoother}

We choose here $p_{\Xc}(\cdot) = \norm{\cdot}^2_\Xc = \norm{\cdot}^2_A$, $p_{\Yc}(\cdot) = \norm{\cdot}^2_\Yc = \norm{\cdot}_{\Id}^2$ and $\dot{x} = \xopt$ and $b_{\Xc}(\xb,\dot{\xb}) := (1/2)\norm{\Ab(\xb - \xopt)}_{\Yc,\ast}^2 = (1/2)\norm{\Ab\xb - \cb}_{\Yc,\ast}^2$. This is indeed the augmented term for the Lagrange function of \eqref{eq:constr_cvx}.
Note that even though $\dot x$ is unknown, $b_{\Xc}(\xb,\dot{\xb})$
can be computed easily using the equality $A \xopt = c$.

We specify the primal-dual \ref{eq:pd_scheme_2d}  scheme with the augmented Lagrangian smoother for fixed $\gamma_{k+1} = \gamma_0 > 0$ as follows:
\begin{equation}\label{eq:1p2d_augLM}\tag{ASALGARD}
\left\{\begin{array}{ll}
\yhat^k       &:= (1-\tau_k)\ybar^k + \tau_k\yb^{\ast}_{\beta_k}(\Ab\xbar^k;\dot{\yb}), \vspace{0.75ex}\\
\xhat^{\ast}_{\gamma_0}(\hat{\yb}^k) &:= \displaystyle\argmin_{\xb\in\Xc}\big\{f(\xb) + \iprods{\yhat^k, \Ab\xb - \cb} + \frac{\gamma_0}{2}\norm{\Ab\xb - \cb}_{\Yc, *}^2\big\}, \vspace{0.75ex}\\
\ybar^{k+1} &:= \yhat^k + \gamma_0(\Ab\xhat^{\ast}_{\gamma_0}(\hat{\yb}^k) - \cb), \vspace{0.75ex}\\
\xbar^{k+1} &: = (1-\tau_k)\xbar^k + \tau_k\xhat^{\ast}_{\gamma_0}(\hat{\yb}^k),
\end{array}\right.
\end{equation}
where $\tau_k \in (0, 1)$, $\gamma_0 > 0$ is the penalty (or the primal smoothness) parameter, and $\beta_k$ is the dual smoothness parameter. 
As a result, this method is called \textit{Accelerated Smoothing Augmented Lagrangian GAp ReDuction} \eqref{eq:1p2d_augLM} scheme. 

This scheme consists of two dual steps at lines 1 and 3.
However, we can combine these steps as in \eqref{eq:update_yopt_k} so that it requires only one matrix-vector multiplication $\Ab\xb$.
Consequently, the  per-iteration complexity of \ref{eq:1p2d_augLM} remains essentially the same as the standard augmented Lagrangian method \cite{Bertsekas1996d}. 

\paragraph{The update rule for parameters}
In our augmented Lagrangian method, we only need to update $\tau_k$ and $\beta_k$ such that $\beta_{k+1} \geq (1-\tau_k)\beta_k$ and $\gamma_0\beta_k(1-\tau_k) \geq \tau_k^2$.
Using the equality in these conditions and defining $\tau_k := t_k^{-1}$, we can derive 
\begin{equation}\label{eq:update_tau_beta7}
t_{k+1} := \frac{1}{2}\Big(1 + \sqrt{1 + 4t_k^2}\Big)~~~\text{and}~~~\beta_{k+1} := \frac{(t_k -1)}{t_k}\beta_k.
\end{equation}
Here, we fix $\beta_1 > 0$ and choose $t_0 := 1$.

\paragraph{The algorithm template}
We modify Algorithm \ref{alg:pd_alg2} to obtain the following augmented Lagrangian variant, Algorithm \ref{alg:aug_Lag_alg}.

\begin{algorithm}[!ht]\caption{(\textit{Accelerated Smoothing Augmented Lagrangian GAp ReDuction} \eqref{eq:1p2d_augLM})}\label{alg:aug_Lag_alg}
\begin{normalsize}
\begin{algorithmic}[1]
\Statex{\hskip-4ex}\textbf{Initialization:} 
\State Choose an initial value $\gamma_0 > 0$ and $\beta_0 := 1$. Set $t_0 := 1$ and $\beta_1 := \gamma_0^{-1}$.
\State  Choose an initial point $(\xbar^0, \ybar^0)\in\Wc$.
\Statex{\hskip-4ex}\textbf{For}~{$k=0$ {\bfseries to} $k_{\max}$, \textbf{perform:}}
\vspace{1ex}
\State Update $\yb^{\ast}_{\beta_k}(\xbar^k;\dot{\yb}) :=  \nabla{b^{*}_{\Yc}}\left(\beta_k^{-1}(\Ab\xbar^k - \cb), \dot{\yb}\right)$. 
\State Update 
\begin{equation}\label{eq:augLag_cvx_subprob} 
\xhat^{\ast}_{\gamma_0}(\hat{\yb}^k) := \displaystyle\argmin_{\xb\in\Xc}\big\{ f(\xb) + \iprods{\yhat^k, \Ab\xb - \cb} + \frac{\gamma_0}{2}\norm{\Ab\xb - \cb}_{\Yc,\ast}^2 \big\}.
\end{equation}
\State Update $\ybar^{k+1} := \yhat^k + \gamma_0(\Ab\xhat^{\ast}_{\gamma_0}(\hat{\yb}^k) - \cb)$ and $\xbar^{k+1} : = (1-t_k^{-1})\xbar^k + t_k^{-1}\xhat^{\ast}_{\gamma_0}(\hat{\yb}^k)$.
\State Update $t_{k+1} := 0.5\left( 1 + \sqrt{1 + 4t_k^2}\right)$ and $\beta_{k+2} := (t_{k+1} - 1)t_{k+1}^{-1}\beta_{k+1}$.
\vspace{1ex}
\Statex{\hskip-4ex}\textbf{End for}
\end{algorithmic}
\end{normalsize}
\end{algorithm}

\noindent The main step of Algorithm \ref{alg:aug_Lag_alg} is the solution of the primal convex subproblem~\eqref{eq:augLag_cvx_subprob}.
In general, solving this subproblem remains challenging due to the non-separability of the quadratic term $\norm{\Ab\xb-\cb}^2_{\Yc,\ast}$.
We can numerically solve it by using either alternating direction optimization methods or other first-order methods.
The convergence analysis of inexact augmented Lagrangian methods can be found in \cite{Nedelcu2014}.

\paragraph{Convergence guarantee}
The following proposition shows the convergence of  Algorithm  \ref{alg:aug_Lag_alg}, whose proof is moved to Appendix \ref{apdx:th:convergence_A2}.

\begin{proposition}\label{th:convergence_A2}
Let  $\{\wbar^k\}$ be the sequence generated by Algorithm \ref{alg:aug_Lag_alg}.
Then, we have
\begin{equation}\label{eq:convergence_A1}
\left\{\begin{array}{rcl}
 &-\frac{8 L_{b_\Yc}\norm{\yopt}_\Yc \norm{\yopt - \dot{\yb}}_{\Yc}}{\gamma_0 (k+2)^2}  \leq f(\xbar^k) \!-f^{\star} \leq \frac{8 L_{b_\Yc}\norm{\yopt}_\Yc \norm{\yopt - \dot{\yb}}_{\Yc}+ 4b_{\Yc}(\yopt, \dot{\yb})}{\gamma_0 (k+2)^2},\vspace{1ex} \\
 &\Vert\Ab\xbar^k \!-\! c\Vert_{\Yc,\ast} \leq
 \frac{ 8 L_{b_\Yc} \norm{\yopt - \dot{\yb}}_{\Yc}}{\gamma_0(k+2)^2}. 
\end{array}\right.
\end{equation}
As a consequence, the worst-case iteration-complexity of Algorithm \ref{alg:aug_Lag_alg} to achieve an $\varepsilon$-primal solution $\xbar^k$ for \eqref{eq:constr_cvx} is $\mathcal{O}\left( \sqrt{\frac{b_{\Yc}(\yb^{\star},\dot{\yb})}{\gamma_0\varepsilon}}\right)$.
\end{proposition}

The estimate \eqref{eq:convergence_A1} guides us to choose a large value for $\gamma_0$ such that we obtain better convergence bounds.
However, if $\gamma_0$ is too large, then the complexity of solving the subproblem \eqref{eq:augLag_cvx_subprob} increases commensurately.
In practice, $\gamma_0$ is often updated using a heuristic strategy \cite{Bertsekas1996d,Boyd2011}.
In general settings, since the solution $\xhat^{\ast}_{k+1}$ computed by \eqref{eq:augLag_cvx_subprob} requires to solve a generic convex problem, it no longer has a closed form expression.

\subsection{The strongly convex objective case}
If the objective function $f$ of \eqref{eq:primal_cvx} is strongly convex with the convexity parameter $\mu_f > 0$, then it is well-known~\cite{Nesterov2005c} that its conjugate $f^{\ast}$ is smooth,
and its gradient $\nabla{f}^{\ast}(\cdot) := \xb^{\ast}(\cdot)$ is Lipschitz continuous with the Lipschitz constant $L_{f^{\ast}} := \mu_f^{-1}$, where $\xb^{\ast}(\cdot)$ is given by
\begin{equation}\label{eq:primal_cvx_subprob}
\xb^{\ast}(\ub) := \mathrm{arg}\max_{\xb\in\Xc}\set{\iprods{\ub, \xb} - f(\xb)}.  
\vspace{-1ex}
\end{equation}
In addition, if $f^{\ast}_{\Ab}(\cdot) := f^{\ast}(-\Ab^{\top}(\cdot))$, then $\nabla{f}_{\Ab}$ is Lipschitz continuous with $L_{f^{\ast}_{\Ab}} := \frac{\bar L_\Ab}{\mu_f} = \frac{\norm{\Ab}^2}{\mu_f}$.

\paragraph{The primal-dual update scheme}
In this subsection, we only illustrate the modification of \ref{eq:pd_scheme_2d} to solve the strongly convex primal problem \eqref{eq:primal_cvx} as
\vspace{-0.75ex}
\begin{equation}\label{eq:pd_scheme_2d_strong}
\left\{\begin{array}{ll}
\hat{\yb}^k          &  := (1 -  \tau_k)\bar{\yb}^k +  \tau_k\yb^{\ast}_{\beta_k}(\Ab\bar{\xb}^k;\dot{\yb})\vspace{0.75ex}\\
\bar{\xb}^{k + 1}   &  := (1 -  \tau_k)\bar{\xb}^k + \tau_k\xb^{\ast}(-\Ab^{\top}\hat{\yb}^k) \vspace{0.75ex}\\
\ybar^{k  + 1} & := \prox_{L_{f^{\ast}_{\Ab}}^{-1}g^{\ast}}\left( \yhat^k + L_{f^{\ast}_{\Ab}}^{-1}\Ab\xb^{\ast}(-\Ab^{\top}\hat{\yb}^k)\right).
\end{array}\right.
\tag{\textrm{ADSGARD$_\mu$}} 
\end{equation}
We note that we no longer have the dual smoothness parameter $\gamma_k$, which is fixed to $\mu_f > 0$.
Hence, the conditions \eqref{eq:pd_condition2} of Lemma \ref{le:maintain_excessive_gap2} reduce to $\beta_{k+1} \geq (1-\tau_k)\beta_k$ and $(1-\tau_k)\beta_k \geq L_{f^{\ast}_{\Ab}}\tau_k^2$.
From these conditions we can derive the update rule for $\tau_k$ and $\beta_k$ as in Algorithm \ref{alg:aug_Lag_alg}, which is
\begin{equation}\label{eq:update_tau_beta7_muf}
t_{k+1} := \frac{1}{2}\Big(1 + \sqrt{1 + 4t_k^2}\Big),~~~~\beta_{k+1} := \frac{(t_k -1)}{t_k}\beta_k~~~\text{and}~~~ \tau_k := t_k^{-1}.
\end{equation}
Here, we fix $\beta_1 := L_{f^{\ast}_{\Ab}} = \frac{\norm{\Ab}^2}{\mu_f}$ and choose $t_0 := 1$.

\paragraph{Convergence guarantee}
The following proposition shows the convergence of \ref{eq:pd_scheme_2d_strong}, whose proof is in Appendix \ref{apdx:co:strong_convex_convergence}.

\begin{proposition}\label{co:strong_convex_convergence}
Suppose that the objective $f$ of the constrained convex problem \eqref{eq:constr_cvx} is strongly convex with the convexity parameter $\mu_f > 0$.
Let $\set{\wbar^k}$ be generated by \ref{eq:pd_scheme_2d_strong}  using the update rule \eqref{eq:update_tau_beta7_muf}. 
Then, the following guarantees hold:
\begin{equation}\label{eq:strong_cvx_main_estimate}
\left\{\begin{array}{rcl}
 &-\frac{8 L_{b_\Yc} \bar L_\Ab \norm{\yopt}_\Yc \norm{\yopt - \dot{\yb}}_{\Yc}}{\mu_f (k+2)^2}  \leq f(\xbar^k) \!-f^{\star} \leq \frac{8 L_{b_\Yc} \bar L_\Ab \norm{\yopt}_\Yc \norm{\yopt - \dot{\yb}}_{\Yc}+ 4b_{\Yc}(\yopt; \dot{\yb})}{\mu_f (k+2)^2},\vspace{1ex} \\
 &\Vert\Ab\xbar^k \!-\! \cb\Vert_{\Yc,\ast} \leq \frac{ 8 L_{b_\Yc} \bar{L}_\Ab \norm{\yopt - \dot{\yb}}_{\Yc}}{\mu_f(k+2)^2}. 
\end{array}\right.
\end{equation}
\end{proposition}
This result shows that \ref{eq:pd_scheme_2d_strong} has an $\mathcal{O}(1/k^2)$ convergence rate with respect to the objective residual and the feasibility gap.
We note that in both Propositions~\ref{th:convergence_A2} and \ref{co:strong_convex_convergence}, the bounds only depend on the quantities in the dual space and $\bar{L}_{\Ab}$.

\subsection{Extension to general cone constraints}\label{sec:extensions}
The theory presented in the previous sections can be extended to solve the following general constrained convex optimization problem:
\begin{equation}\label{eq:ineq_cvx_prob}
f^{\star} := \displaystyle\min_{\xb\in\Xc} \set{ f(\xb) \mid \Ab\xb - \cb \in \Kc },
\end{equation}
where $f$, $\Ab$ and $\cb$ are defined as in \eqref{eq:constr_cvx}, and $\Kc$ is a nonempty, closed and convex set in $\R^m$.

If $\Kc$ is  a nonempty, closed and convex set, then a simple way to process \eqref{eq:ineq_cvx_prob} is using a slack variable $\rb \in\Kc$ such that $\rb := \Ab\xb - \cb$ and $\zb := (\xb, \rb)$ as a new variable. 
Then, we can transform \eqref{eq:ineq_cvx_prob} into \eqref{eq:constr_cvx} with respect to the new variable $\zb$.
The primal subproblem corresponding to $\rb$ is defined as $\min\set{\iprods{-\yb, \rb} \mid \rb\in\Kc}$, which is equivalent to the support function $s_{\Kc}(\yb) := \displaystyle\sup\set{\iprods{\yb, \rb} \mid \rb\in\Kc }$ of $\Kc$. 
Consequently, the dual function becomes $\tilde{g}(\yb) := g(\yb) - s_{\Kc}(\yb)$, where $g(\yb) := \min\set{ f(\xb) + \iprods{\Ab\xb - \cb, \yb} \mid \xb\in\Xc}$.
Now, we can apply the algorithms presented in the previous sections to obtain an approximate solution $\zbar^k := (\xbar^k, \bar{\rb}^k)$ with a convergence guarantee on $f(\xbar^k) - \fopt$,  $\norm{\Ab\xbar^k - \bar{\rb}^k - \cb}_{\Yc,\ast}$, $\xbar^k\in\Xc$ and $\bar{\rb}^k \in\Kc$ as in Theorem~\ref{th:convergence_accPD_alg_const}, or Theorem~\ref{th:convergence_A1_const}.

If $\Kc$ is a cone (e.g., $\Kc := \R^m_{+}$, $\Kc := \mathcal{L}^m_{+}$ is a second order cone, or $\Kc := \mathcal{S}^m_{+}$ is a semidefinite cone), then with the choice $p_{\Yc}(\cdot) := (1/2)\norm{\cdot}_{\Yc}^2$, we can substitute the smoothed function $g_{\beta}$ in \eqref{eq:g_beta} to obtain the following one
\begin{equation}\label{eq:f_beta1}
\hat{g}_{\beta}(\Ab\xb,\dot{\yb}) := \max\set{ \iprods{\Ab\xb - \cb, \yb} - (\beta/2)\norm{\yb - \dot{\yb}}_{\Yc}^2 \mid \yb\in-\Kc^{\ast} },
\end{equation}
where $\Kc^{\ast}$ is the dual cone of $\Kc$, which is defined as $\Kc^{\ast} := \set{\zb \mid \iprods{\zb, \xb} \geq 0, ~\xb\in\Kc}$.
With this definition, we use the smoothed gap function $\hat{G}_{\gamma\beta}$ as $\hat{G}_{\gamma\beta}(\wb; \dot{\wb}) := \hat{P}_{\beta}(\xb;\dot{\yb}) - D_{\gamma}(\yb;\dot{\xb})$, where $D_{\gamma}(\yb;\dot{\xb}) := \min\set{ f(\xb) + \iprods{\Ab\xb - \cb, \yb} + \gamma b_{\Xc}(\xb,\dot{\xb}) \mid \xb\in\Xc}$ is the smoothed dual function defined as before, and $\hat{P}_{\beta}(\xb;\dot{\yb}) := f(\xb) + \hat{g}_{\beta}(\Ab\xb,\dot{\yb})$.

In principle, we can apply one of the two previous schemes to solve \eqref{eq:ineq_cvx_prob}. 
Let us demonstrate the \ref{eq:pd_scheme_2d} for this case.
Since $\Kc$ is a cone, we remain using the original scheme \eqref{eq:pd_scheme_2d} with the following changes:
\vspace{-0.75ex}
\begin{equation*}
\left\{\begin{array}{ll}
\yb_{\beta_k}^{*}(\Ab\bar{\xb}^k;\dot{\yb}) &:=  \mathrm{proj}_{-\Kc^{\ast}}\left(\dot{\yb} + \beta_k^{-1}(\Ab\xbar^k - \cb)\right),\vspace{1ex}\\
\ybar^{k+1} &:=  \mathrm{proj}_{-\Kc^{\ast}}\left(\yhat^k + \frac{\gamma_{k+1}}{\bar{L}_A}\left(\Ab\xb^{\ast}_{\gamma_{k+1}}(\yhat^k) - \cb\right)\right),
\end{array}\right.
\vspace{-0.75ex}
\end{equation*}
where $\mathrm{proj}_{-\Kc^{\ast}}$ is the projection onto the cone $-\Kc^{\ast}$.
In this case, we still have the convergence guarantee as in Theorem~\ref{th:convergence_A1} for the objective residual $f(\xbar^k) - \fopt$ and the primal feasibility gap $\dist{\Ab\xbar^k - \cb, \Kc}$, the Euclidean distance from $A\bar{x}^k - c$ to $\Kc$.
We note that if $\Kc$ is a self-dual conic cone, then $\Kc^{*} = \Kc$. 
Hence, $\yb_{\beta_k}^{*}(\Ab\bar{\xb}^k; \dot{\yb})$ and $\bar{\yb}^{k+1}$ can be either  efficiently computed or a closed form.

\subsection{Restarting techniques}
Similar to other accelerated gradient algorithms in \cite{Giselsson2014,Odonoghue2012,Su2014}, restarting ASGARD and ADSGARD may lead to
a better performance in practice. 
We discuss in this subsection how to restart these two algorithms using a fixed iteration restarting strategy \cite{Odonoghue2012}.

If we consider \ref{eq:ac_pd_scheme}, then, when a restart takes place,  we perform the following steps:
\begin{equation}\label{eq:asgard-restart}
\left\{\begin{array}{llll}
\tilde{\xb}^{k+1}& \leftarrow &&\bar{\xb}^{k+1}, \\
\dot{\yb} &\leftarrow && \yb^{\ast}_{\beta_{k+1}}(\Ab \bar{\xb}^{k+1}; \dot{\yb}),\\
\beta_{k+1}& \leftarrow &&\beta_1, \\
\tau_{k+1} &\leftarrow &&1.
\end{array}\right.
\end{equation}
Restarting the primal variable at $\bar{\xb}^{k+1}$ is classical, see, e.g.,~\cite{Odonoghue2012}.
For the dual center point $\dot{\yb}$, we suggest to restart it at the last dual variable computed. 
Indeed, by~\eqref{eq:S_bound}, we know that the distance between $\yb^{*}_{\beta_{k+1}}(\Ab\bar{\xb}^{k+1}; \dot{\yb})$ and the optimal solution $\yopt$ will remain bounded.
Hence, in the favorable cases, we will benefit from a smaller distance between the new center point and $\yopt$, while in the unfavorable cases, restarting should not affect too much the convergence.
In practice, however, we observe that $\yb^{*}_{\beta_{k+1}}(\Ab\bar{\xb}^{k+1}; \dot{\yb})$ converges to the dual solution $\yopt$.
We note that the restarting strategy \eqref{eq:asgard-restart} does not increase the per-iteration complexity of the algorithm.

For \ref{eq:pd_scheme_2d}, we suggest to restart it using the following steps:
\begin{equation}\label{eq:adsgard-restart}
\left\{\begin{array}{llll}
\hat{\yb}^{k+1}& \leftarrow &&\bar{\yb}^{k+1}, \\
\dot{\yb} &\leftarrow && \bar{\yb}^{k+1},\\
\dot{\xb} &\leftarrow && \xb^{\ast}_{\gamma_{k+1}}(\hat{\yb}^k; \dot{\xb}),\\
\beta_{k+1}& \leftarrow &&\beta_1, \\
\gamma_{k+1} & \leftarrow &&\gamma_1, \\
\tau_{k+1} &\leftarrow &&1.
\end{array}\right.
\end{equation}
Understanding the actual consequences of the restart procedure as well as designing other conditions for restarting are still  open questions, even for the unconstrained case. 
Yet, we observe  that it often significantly improves the convergence speed in practice. 

\section{Numerical experiments}\label{sec:num_experiments}
In this section, we provide some key examples to illustrate the advantages of our new algorithms compared to existing state-of-the-arts.
While other numerical experiments can be found in our technical reports \cite{Tran-Dinh2014a}, we instead focus some extreme cases where existing methods may encounter arbitrarily slow convergence rate due to lack of theory, while our methods exhibits an $\mathcal{O}(1/k)$ rate as predicted by the theory.
We then compare our methods with \cite{Nesterov2005c} and provide one application to illustrate the advantages of the proposed algorithms.

\subsection{A degenerate linear program}
We aim at comparing different algorithms to solve the following simple linear program:
\begin{equation}\label{eq:lp_exam}
\left\{\begin{array}{ll}
\displaystyle\min_{\xb \in \R^n}& 2 \xb_n     \\
\textrm{s.t.} & \textstyle \sum_{k=1}^{n-1} \xb_k  = 1, \\
              &  \xb_n - \textstyle \sum_{k=1}^{n-1} \xb_k = 0 \quad (2 \leq j \leq d), \\
              &  \xb_n \geq 0.
\end{array}\right.
\end{equation}
The second inequality is repeated $d-1$ times, which makes the problem degenerate.
Yet, qualification conditions hold since this is a feasible and bounded linear program.
This fits into our framework with $f(\xb) := 2 \xb_n + \delta_{\set{\xb_n\geq0}}(\xb_n)$,  $\Ab\xb := [\sum_{k=1}^{n-1} \xb_k; \xb_n - \sum_{k=1}^{n-1} \xb_k; \cdots; \xb_n - \sum_{k=1}^{n-1} \xb_k]$,
$\cb := (1, 0, \cdots, 0)^{\top} \in \R^d$ and  $g(\cdot) := \delta_{\set{\cb}}(\cdot)$.
A primal and dual solution can be found explicitly and by playing with the sizes $n$ and $d$ of the problem, one can control the degree of degeneracy.

In this test, we choose $n = 10$ and $d = 200$. 
We implement both \ref{eq:ac_pd_scheme} and \ref{eq:pd_scheme_2d} and their restart variants.
In Figure~\ref{fig:lp-example}, we compare our methods against the Chambolle-Pock method~\cite{Chambolle2011}. 
We can see that the Chambolle-Pock method struggles with the degeneracy while \ref{eq:ac_pd_scheme} still exhibits an $\mathcal{O}(1/k)$ sublinear convergence rate as predicted by our theory.

In Figure~\ref{fig:lp-example2}, we compare methods requiring the resolution of a  nontrivial optimization subproblem at each iteration.
In this case,  the inversion of a rank deficient linear system, we thus compare \ref{eq:1p2d_augLM} with and without restart against ADMM~\cite{Boyd2011}. 
For ADMM, we selected the step-size parameter by sweeping from small values to large values and choosing the one that gives us the fastest performance.
Again, our algorithm resists to the degeneracy and restarting strategies improves  the performance, while ADMM has very slow convergence rate.

\begin{figure}[ht!]
\includegraphics[width=0.49\linewidth]{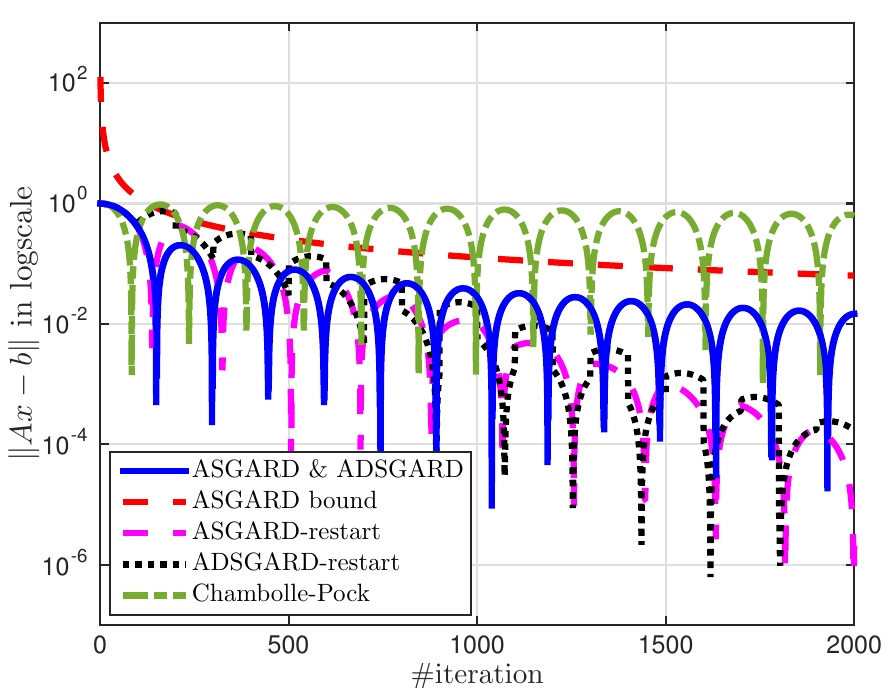}
\includegraphics[width=0.49\linewidth]{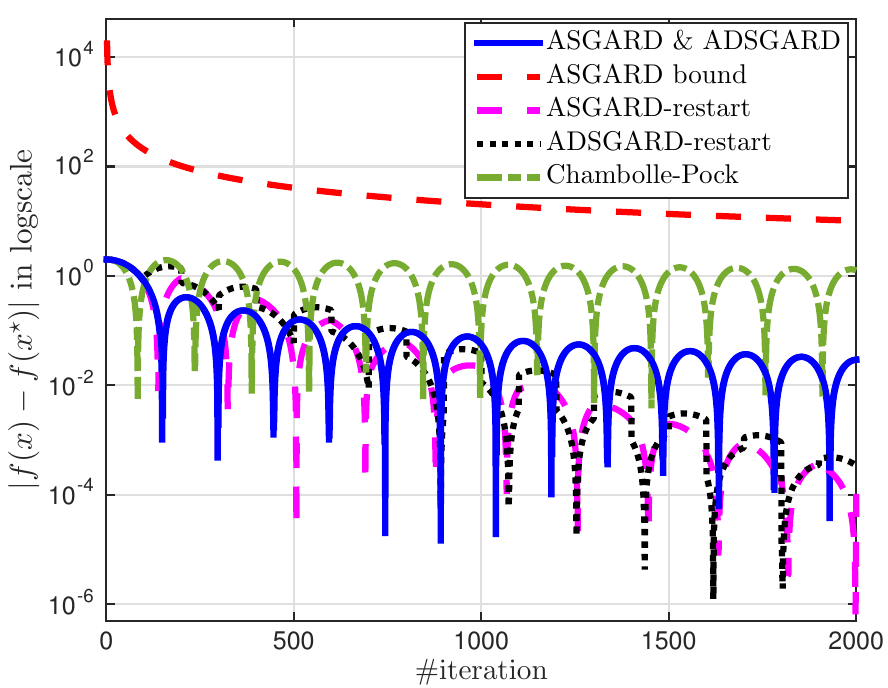}
\vspace{-1ex}
\caption{Comparison of the absolute feasibility violation $($left$)$ and the absolute objective residual $($right$)$ for \ref{eq:ac_pd_scheme} $($solid blue line$)$, 
\ref{eq:ac_pd_scheme} with a restart every 100 iterations using \eqref{eq:asgard-restart} $($dashed pink line$)$, 
\ref{eq:pd_scheme_2d} with a restart every 100 iterations using \eqref{eq:adsgard-restart} $($black dotted line$)$, and Chambolle-Pock $($green dash-dotted line$)$.
The dashed red line is the theoretical bound of \ref{eq:ac_pd_scheme} $($Theorem~\ref{th:convergence_accPD_alg}$)$.
 \ref{eq:pd_scheme_2d} leads to similar results as \ref{eq:ac_pd_scheme} on this linear program~\eqref{eq:lp_exam}: the difference is not perceptible on the figure.}
\label{fig:lp-example}
\end{figure}

\begin{figure}[ht!]
\vspace{-1ex}
\includegraphics[width=0.49\linewidth]{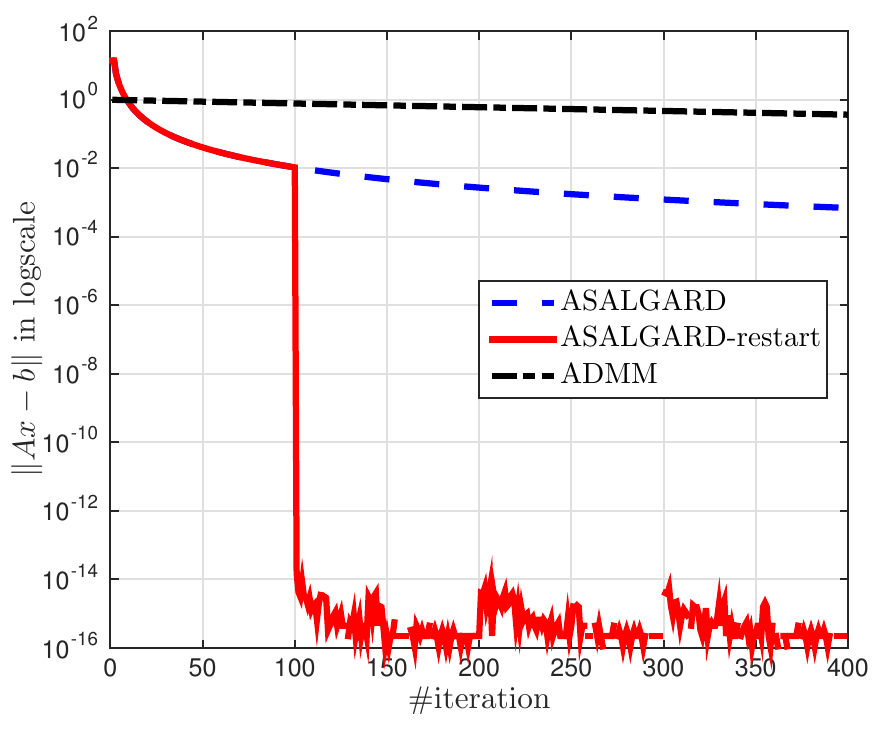}
\includegraphics[width=0.49\linewidth]{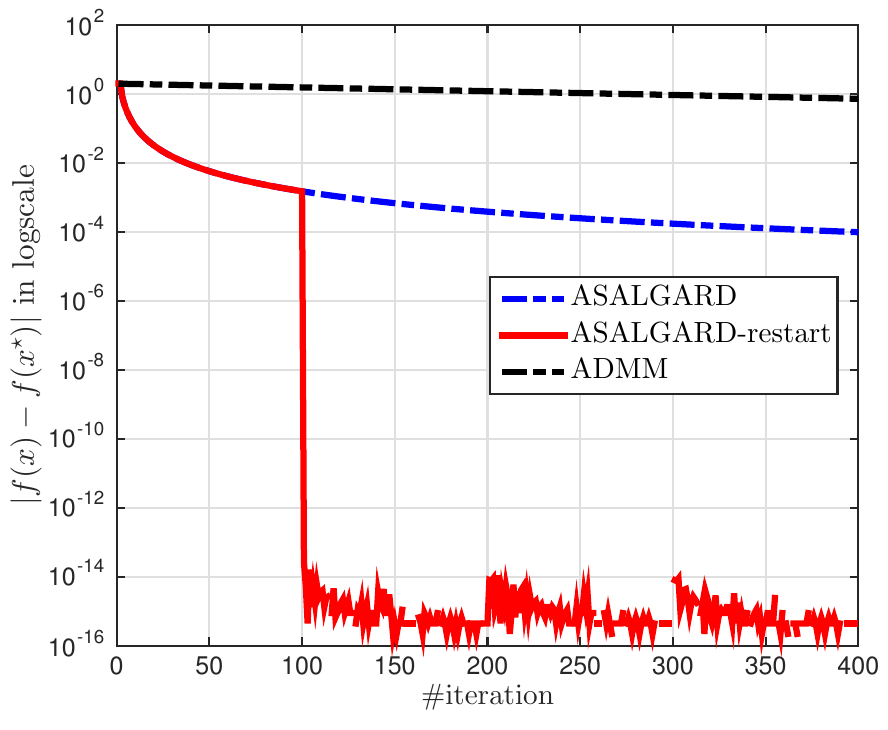}
\vspace{-1ex}
\caption{Comparison of the absolute feasibility violation $($left$)$ and the absolute objective residual $($right$)$ for \ref{eq:1p2d_augLM} $($solid blue line$)$, \ref{eq:1p2d_augLM} with a restart every 100 iterations $($dashed red line$)$, and  ADMM $($black dotted line$)$.}\label{fig:lp-example2}
\vspace{-2ex}
\end{figure}

\subsection{Generalized convex feasibility problem}
Given $N$ nonempty, closed and convex sets $\Xc_i \subseteq \R^n$ for $i=1,\cdots, N$, we consider the following optimization problem:
\begin{equation}\label{eq:min_Xc}
\min_{\xb := (\xb_1^{\top}, \cdots, \xb_N^{\top})^{\top} \in\R^{Nn}} \set{ f(\xb) := \sum_{i=1}^Ns_{\Xc_i}(\xb_i)  \mid \sum_{i=1}^N\Ab^{\top}_i\xb_i = \boldsymbol{0}^m },
\end{equation}
where $s_{\Xc_i}$ is the support function of $\Xc_i$, and $\Ab_i\in\R^{n\times m}$ is given   for $i=1,\cdots, N$.

It is trivial to show that the dual problem of \eqref{eq:min_Xc} is the following generalization of  a convex feasibility problem:
\begin{equation}\label{eq:feasibility_cvx}
\text{Find}~\yb^{\star}\in\R^m~\text{such that:}~\Ab_i\yb^{\star} \in \Xc_i ~(i=1,\cdots, N).
\end{equation}
Clearly, when $\Ab_i = \Id$ the identity matrix, \eqref{eq:feasibility_cvx} becomes the classical  convex feasibility problem.
When $\Ab_i = \Id$ for some $i\in\set{1,\cdots, N}$ and $\Ab_i = \Ab$, otherwise, \eqref{eq:feasibility_cvx} becomes a multiple-set split feasibility problem considered in the literature.
Assume that \eqref{eq:feasibility_cvx} has a solution and $N\geq 2$. 
Hence, \eqref{eq:min_Xc} and \eqref{eq:feasibility_cvx} satisfy Assumption~A.\ref{as:A1}.

Our aim is to apply Algorithm~\ref{alg:pd_alg1} and Algorithm~\ref{alg:pd_alg2} to solve the primal problem \eqref{eq:min_Xc}, and compare them with the most state-of-the-art  ADMM algorithm with multiple blocks \cite{deng2013parallel}.
Clearly, with nonorthogonal $\Ab_i$, the primal subproblem of computing $\xb_i$ in the parallel-ADMM scheme \cite{deng2013parallel} does not have a closed form solution, we need to solve it iteratively up to a given accuracy. 
In addition, by a change of variable, we can rescale the iterates such that ADMM does not depend on the penalty parameter when solving \eqref{eq:min_Xc}.
With the use of Euclidean distance for our smoother, Algorithm~\ref{alg:pd_alg1} and Algorithm~\ref{alg:pd_alg2} can solve the primal subproblem \eqref{eq:x_ast} in $\xb_i$ with a closed form solution, which only requires one projection onto $\Xc_i$.

The first experiment is for $N=2$. We choose  $\Xc_1 := \set{\yb\in\R^n \mid \epsilon \yb_1 - \sum_{j=2}^n\yb_j \leq 1}$ and $\Xc_2 := \set{\yb\in\R^n \mid \sum_{i=2}^n\yb_j \leq -1}$ to be two half-planes, where $\epsilon > 0$ is fixed. The constant $\epsilon$ represents the angle between these half-planes.
It is well-known \cite{Tran-Dinh2015} that the ADMM algorithm can be written equivalently to an alternating projection method on the dual space.
The convergence of this algorithm strongly depends on the angle between these sets. By varying $\epsilon$, we observe the convergence speed of ADMM is also varying, while our algorithms seem not to depend on $\epsilon$.
Figure \ref{fig:cfp_exam_10000} shows the convergence rate on the absolute feasibility gap $\Vert\sum_{i=1}^N\xb_i\Vert_2$ of three algorithms for $n = 10,000$. 
Since the objective value is always zero, we omit its plot here.

\begin{figure}[ht!]
\includegraphics[width=0.495\linewidth]{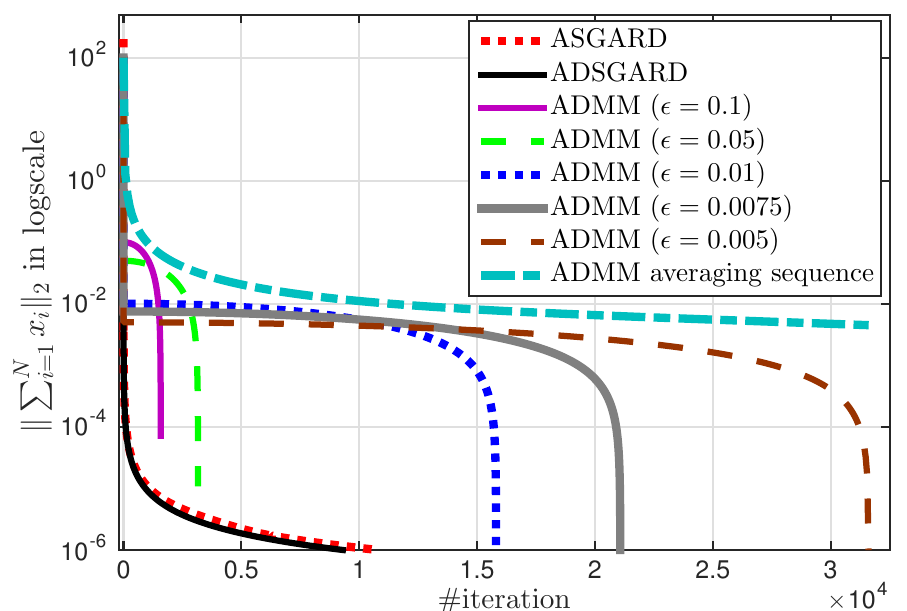}
\includegraphics[width=0.495\linewidth]{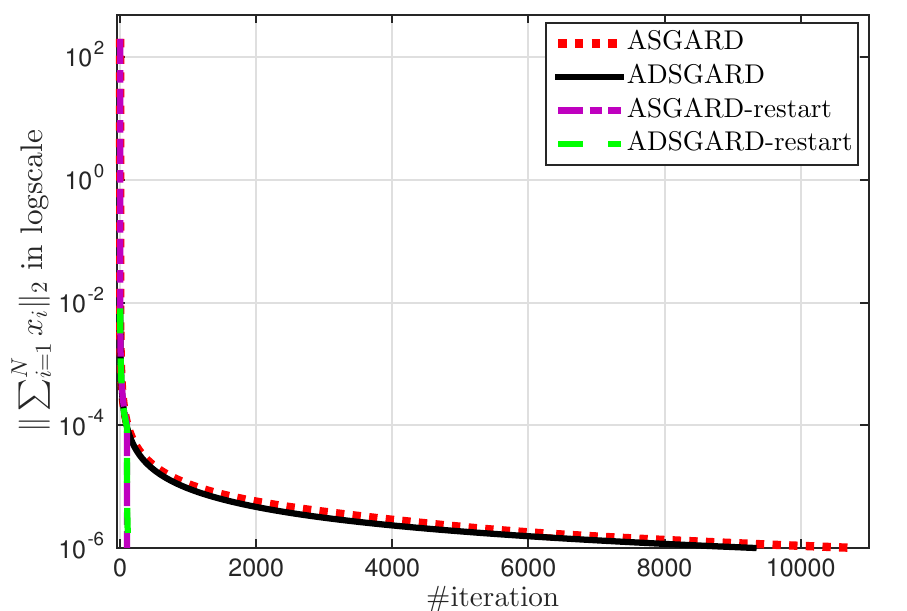}
\vspace{-1ex}
\caption{Comparison of Algorithm~\ref{alg:pd_alg1}, Algorithm~\ref{alg:pd_alg2} and ADMM with different values of $\epsilon$ $($left$)$. Comparison of Algorithm~\ref{alg:pd_alg1}, Algorithm~\ref{alg:pd_alg2} and their restart variant $($restarting after every 100 iterations$)$ $($right$)$. The number of variables is $20,000$.}\label{fig:cfp_exam_10000}
\vspace{-1ex}
\end{figure}
The theoretical version of Algorithm~\ref{alg:pd_alg1} and Algorithm~\ref{alg:pd_alg2}  exhibits a convergence rate slightly better than $\mathcal{O}(1/k)$  and  is independent of $\epsilon$, while ADMM can be arbitrarily  slow as $\epsilon$ decreases. 
ADMM very soon drops to a certain accuracy and then is saturated at that level before it converges. Algorithm~\ref{alg:pd_alg1} and Algorithm~\ref{alg:pd_alg2} also quickly converge to the $10^{-5}$ accuracy level and then make a slower progress to achieve the $10^{-6}$ accuracy, but still obeys our theoretical guarantee. 
We notice that the averaging sequence of ADMM converges at the $\mathcal{O}(1/k)$ rate but it remains far away from our theoretical rate in Algorithm~\ref{alg:pd_alg1} and Algorithm~\ref{alg:pd_alg2} due to a big constant factor.
If we combine these two algorithms with our restart strategy, both algorithms need $102$ iterations to reach the desired accuracy. 
We can see that Algorithm~\ref{alg:pd_alg1}  performs very similar to Algorithm~\ref{alg:pd_alg2}.
We can also observe that the performance of  our algorithms depends on $\bar{L}_{\Ab}$ and initial points, but it is relatively independent of the geometric structure of problems as opposed to the ADMM for solving the generalized convex feasibility problem \eqref{eq:min_Xc}.

Now, we extend to the cases of $N=3$ and $N=4$, where we add two more sets $\Xc_3$ and $\Xc_4$. 
We choose $\Xc_3 := \set{\yb\in\R^n \mid 0.5\epsilon \yb_1 - \sum_{j=2}^n\yb_j = 1}$ to be a hyperplane in $\R^n$, and $\Xc_4 := \set{\yb\in\R^n \mid -\yb_1 + \sum_{j=3}^n\yb_j \leq 1}$ to be a half-plane in $\R^n$. 
We test our algorithms and the multiblock-ADMM method in \cite{deng2013parallel} again. 
\begin{figure}[ht!]
\includegraphics[width=0.495\linewidth]{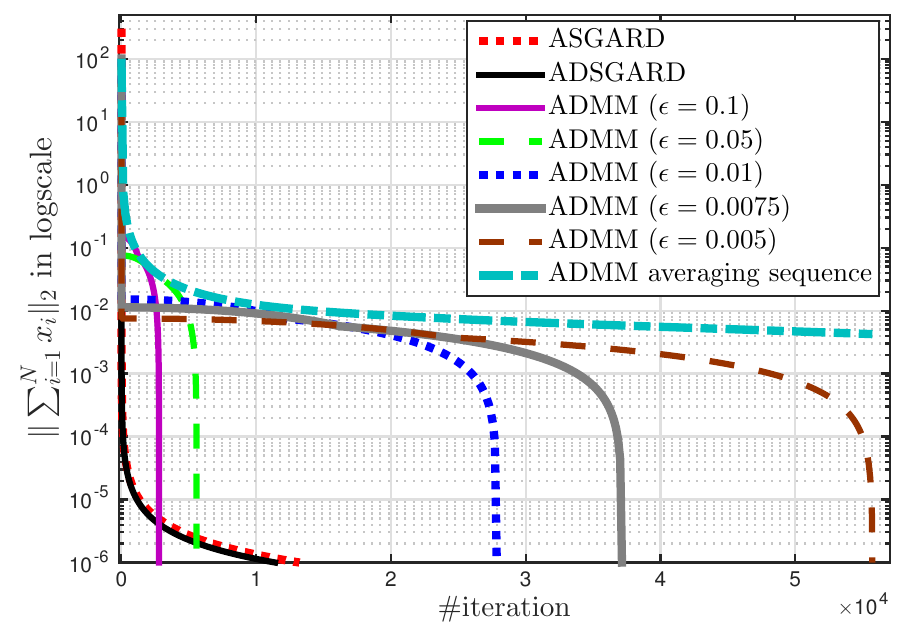}
\includegraphics[width=0.495\linewidth]{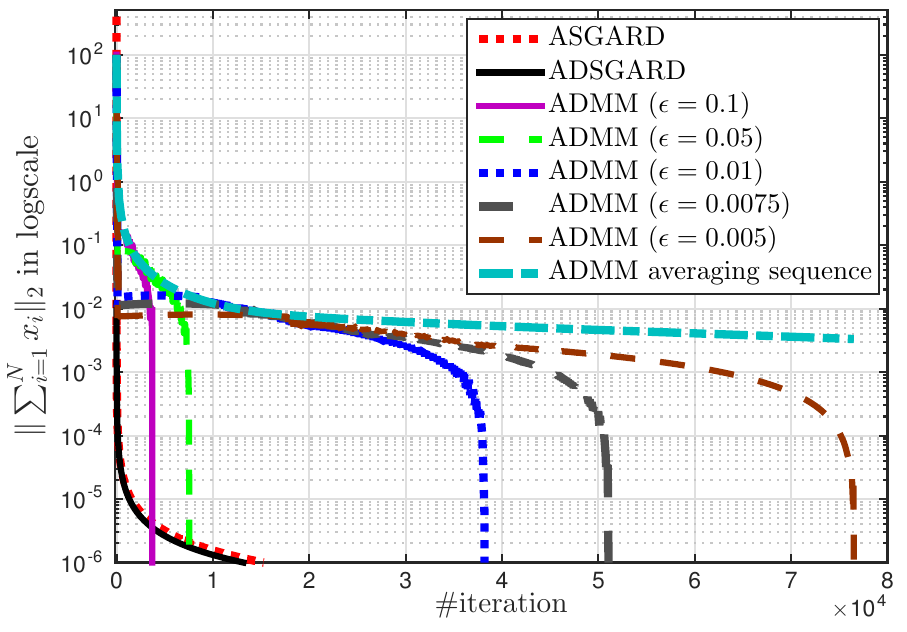}
\vspace{-2ex}
\caption{Comparison of Algorithm~\ref{alg:pd_alg1}, Algorithm~\ref{alg:pd_alg2} and ADMM with different values of $\epsilon$. The left-plot is for $N=3$ and the right one is for $N=4$.}\label{fig:cfp_exam2_10000}
\vspace{-3ex}
\end{figure}
The results are plotted in Figure~\ref{fig:cfp_exam2_10000} for the case $n = 10,000$.
In both cases, the ADMM algorithm still makes a slow progress as $\epsilon$ is decreasing and $N$ is increasing. 
 Algorithm~\ref{alg:pd_alg1} and Algorithm~\ref{alg:pd_alg2} seem to scale slightly to $N$, the number of blocks. 
We note that since $\Ab_i =\Id$ for $i=1,\cdots, N$. The per-iteration complexity of three algorithms in our experiment is essentially the same.

\subsection{A comparison with \cite{Nesterov2005c}}\label{subsec:compare_55}
To see the advantages of our homotopy strategy, we consider the following square-root LASSO problem considered in the literature, e.g., \cite{Belloni2011}:
\begin{equation}\label{eq:sqrt_lasso}
P^{\star} = \min_{\xb\in\R^n} \set{ P(\xb) := \tfrac{1}{\sqrt{m}}\Vert\Ab\xb - \bb\Vert_2 + \lambda\Vert\xb\Vert_1 },
\end{equation}
where $\Ab\in\R^{m\times n}$, $b\in\R^m$, and $\lambda > 0$ is a given regularization parameter.
As suggested in \cite{Belloni2011}, we can choose $\lambda := \frac{1.1}{\sqrt{m}}\Phi^{-1}(1-0.5\alpha/n)$, where $\alpha = 0.05$ and $\Phi$ is the standard normal distribution.
By letting $f(\xb) :=  \lambda\Vert\xb\Vert_1$, and $g(\ub) := \Vert\ub - \bb\Vert_2 = \max\set{ \iprods{\ub,\yb} - \iprods{\bb,\yb} \mid \Vert\yb\Vert_2 \leq 1}$, we can easily check that $f$ and $g$ satisfy Assumption~\ref{as:A1}. 
Clearly, the conjugate function $g^{\ast}(\yb) := \iprods{\bb,\yb} + \delta_{\mathcal{B}_2(0; 1)}(\yb)$, where $\delta_{\mathcal{B}_2(0; 1)}$ is the indicator function of the $\ell_2$-norm ball $\Bc_2(0; r) := \set{\yb\in\R^m \mid \Vert\yb\Vert_2 \leq r}$.

Since the solutions of \eqref{eq:sqrt_lasso} are sparse, we apply Algorithm~\ref{alg:pd_alg1} to solve this problem and compare it with Nesterov's method in \cite{Nesterov2005c}.
Let us choose $b_{\Xc}(\xb,\dot{\xb}) := \frac{1}{2}\Vert\xb -\dot{\xb}\Vert_2^2$ and $b_{\Yc}(\yb,\dot{\yb}) := \frac{1}{2}\Vert\yb-\dot{\yb}\Vert_2^2$ with $\dot \xb = 0$ and $\dot \yb = 0$.
In this case, $\yb^{\ast}_{\beta}(\xb;\dot{\yb})$ can be computed as the projection on the unit $\ell_2$-ball, while $\xb^{\ast}_{\gamma}(\yb;\dot{\xb})$ is computed from the proximal operator of the $\ell_1$-norm (a soft-thresholding operator). 
We initialize the algorithms at $x^0 = 0$.

Let us try to tune the smoothness parameter in both algorithms.
Nesterov suggested to tune it as follows. Given an iteration budget $K$ and an a priori bound on the distance to the solution, choose the smoothness parameter that minimizes the known theoretical bound.
In Figure~\ref{fig:sqrt_lasso} below, this corresponds to ``$\beta$ with guarantees''.
For this square-root LASSO problem,  we  take $K := 10^5$. 
Theoretically, we can show that $\norm{\xb^{\star} - \xb^0}_2 \leq \norm{\xb^{\star}}_1 \leq \frac{\norm{b}_2}{ \lambda \sqrt{m}}$.
We can also compute $D_{\Yc} := \frac{1}{2}$, where $\Yc := \set{\yb\in\R^m \mid \Vert\yb\Vert_2 \leq 1}$ is an $\ell_2$-unit ball.
The theoretical bound derived from \cite{Nesterov2005c} becomes $P(\xb^K) - P^{\star} \leq \frac{4 \norm{A} \norm{b}_2 \sqrt{D_\Yc}}{\lambda \sqrt{m} (K+1)}$.
Similarly, in our algorithms, we set $\beta_1$ and $\gamma_1$ as suggested by \eqref{eq:convergence_accPD1a}, and \eqref{eq:convergence_2d_Px}, respectively and we obtain a slightly better theoretical bound. 

\begin{figure}[ht!]
\begin{center}
\includegraphics[width=0.92\linewidth]{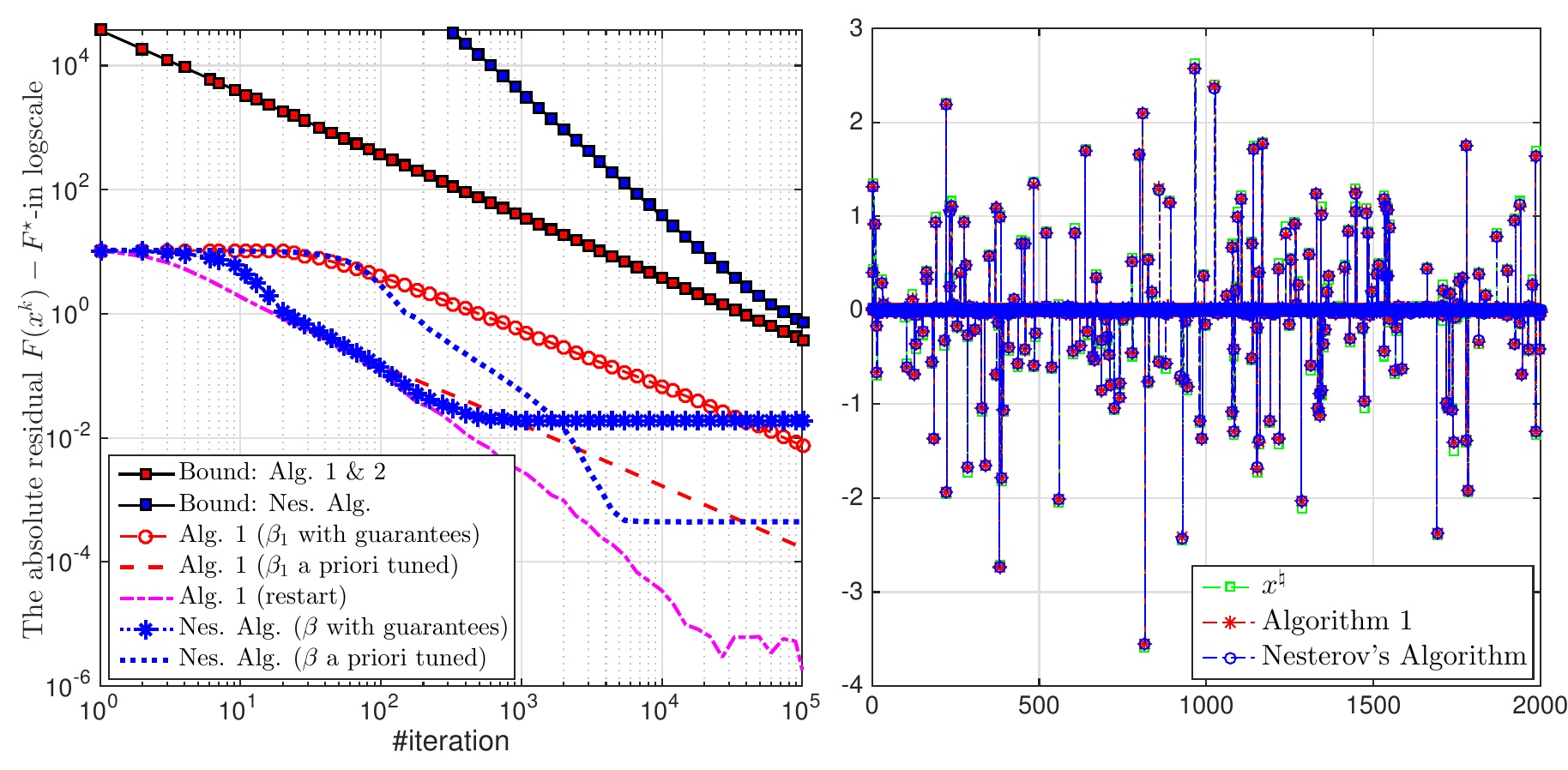}
\vspace{-2ex}
\caption{Comparison of Algorithm~\ref{alg:pd_alg1} $($Alg.~1$)$ and Nesterov's smoothing algorithm in \cite{Nesterov2005c} $($Nes.~Alg.$)$. 
Left: Convergence of the algorithms and their theoretical bounds; Right: Recovered solutions and $\xb^{\natural}$. 
\ref{eq:ac_pd_scheme}~$($red color plots$)$ features a steady $\mathcal{O}(1/k)$ decrease. 
Nesterov's smoothing algorithm $($blue color plots$)$ has a  slower start than \ref{eq:ac_pd_scheme}, then enters a quick decrease phase and finally stagnates when it has reached the minimum of the smoothed problem. 
At the end of the iteration budget,  \ref{eq:ac_pd_scheme} returns a more accurate solution. 
This behavior can be seen for  both parameter tuning strategies and is consistent with the theoretical bound.}\label{fig:sqrt_lasso}
\vspace{-1ex}
\end{center}
\end{figure}

In practice, the a priori bound on the distance $\norm{x^\star- x^0}_2$ may be very conservative. 
Hence, we will also use the quantity $\frac{\norm{b}_2}{\lambda\sqrt{m n}}$ as an estimate of $R_0 := \norm{x^\star - x^0}_2$. 
In Figure~\ref{fig:sqrt_lasso}, this corresponds to ``$\beta$ a priori tuned''.

We generate matrix $\Ab$ using standard Gaussian distribution $\Nc(0, 1)$ with $25\%$ correlated columns. 
The true parameter $\xb^{\natural}$ is a given $s$-sparse vector. We generate the observation $\bb$ as $\bb := \Ab\xb^{\natural} + \Nc(0,0.005)$, where the last term represents Gaussian noise.

Figure~\ref{fig:sqrt_lasso} illustrates the performance of the two algorithms for solving \eqref{eq:sqrt_lasso}, where $m=700$, $n=2000$ and $s = 200$.
The left plot shows the convergence behavior of both algorithms and their theoretical bounds. 
We clearly see that for each tuning strategy, Algorithm \ref{alg:pd_alg1} reaches a smaller final objective value than the one in \cite{Nesterov2005c}.
Moreover, Nesterov's method has the disadvantage of stagnating after a  given moment while Algorithm \ref{alg:pd_alg1} makes steady progress.
Restarting the method every $25$ iterations gives improvement again in the performance.
We can finally observe that both algorithms have better performance than their theoretical worst-case bounds.
Figure~\ref{fig:sqrt_lasso} (right) shows the solutions of both algorithms, and compares them with the true parameter $\xb^{\natural}$.
We see that the obtained solutions fit well $\xb^{\natural}$, and they are both sparse solutions.

\subsection{Application to  image reconstruction}
In this example, we propose to use  the following total variation (TV) norm optimization formulation to reconstruct images from compressive measurements $b$ obtained via a linear operator $\Lc$:
\begin{equation}\label{eq:basis_pursuit}
\min_{\Zb\in\R^{p_1\times p_2}}\big\{f(\Zb) :=  \Vert\Db(\Zb)\Vert_1 \mid \Lc(\Zb) = \bb \big\},
\end{equation}
where $\Db$ is 2D discrete gradient operator, $\Lc:\R^{p_1\times p_2}\to\R^n$ is a linear transformation obtained  from a subsampling-FFT transformation, and $\bb\in\R^n$ is a compressive measurement vector.
We first reformulate problem \eqref{eq:basis_pursuit} into the form \eqref{eq:constr_cvx} using a splitting trick as follows:
\begin{equation}\label{eq:basis_pursuit2}
f^{\star} := \left\{\begin{array}{ll}
\displaystyle\min_{\xb := [\ub^{\top}, \vec{\Zb}^{\top}]^{\top}} &\big\{ f(\xb) := \Vert\ub\Vert_1 \big\}\\
\mathrm{s.t.} & \Lc(\Zb) = \bb, ~~\Db(\Zb) - \ub = 0.
\end{array}\right.
\end{equation}
We now apply Algorithm~\ref{alg:pd_alg1} and Algorithm~\ref{alg:pd_alg2} to solve this problem and compare them with Chambolle-Pock's method in \cite{Chambolle2011,vu2013variable}. We also compare our methods with a line-search variant of the Chambolle-Pock method recently proposed in \cite{malitsky2016first}. 
We note that our algorithms and the standard Chambolle-Pock method have the same per-iteration complexity. 
We first test all the algorithms on two MRI images: \texttt{MRI-brain-tumor} and \texttt{MRI-of-knee}.\footnote{These images are from \href{https://radiopaedia.org/cases/4090/studies/6567}{https://radiopaedia.org/cases/4090/studies/6567} and \href{https://www.nibib.nih.gov/science-education/science-topics/magnetic-resonance-imaging-mri}{https://www.nibib.nih.gov}}
We follow the procedure in \cite{knoll2011adapted} to generate the samples using a sample rate of $20\%$. 
Then, the vector of measurements $\bb$ is computed from $\bb := \Lc(\Zb^{\natural})$, where $\Zb^{\natural}$ is the original image.
Our experiment is implemented in Matlab 2014b running on a MacBook Pro (Retina, 2.7 GHz Intel Core i5, 16GB 1867 MHz).

In Chambollle-Pock's method, we use the parameters as suggested in  \cite{Chambolle2011} with $\tau = \sigma = \Vert\Ab\Vert^{-1}$ (see \cite{malitsky2016first}). 
For the line-search variant of the of Chambollle-Pock's method in \cite{malitsky2016first} (denoted by \texttt{Linesearch CP}), we tune the parameters to obtain the best performance on a set of sample images. 
These parameters are set to $\beta = 10^5$, $\mu = 0.7$ and $\delta  = 0.99$.
Since we aim at reducing the feasibility gap, as guided by our theoretical results above, we use $\beta_1 = 10^{-3}\Vert\Ab\Vert$ in our algorithms. 
\begin{table}[ht!]
\newcommand{\cell}[1]{{\!\!}#1{\!\!\!}}
\begin{center}
\begin{footnotesize}
\caption{Performance and results of the five algorithms on two MRI images}\label{tbl:MRI_results}
\begin{tabular}{| c | rrrrr | rrrrr |}\hline
& \multicolumn{5}{|c|}{\texttt{MRI-knee ($650\times 650$)}} &  \multicolumn{5}{|c|}{\texttt{MRI-brain-tumor ($630\times 611$)}} \\ \hline
\cell{Algorithms} & \cell{$f(\Zb^k)$} & \cell{$\frac{\Vert\Lc(\Zb^k) -\bb\Vert}{\norm{\bb}}$} & \cell{\texttt{Error}} & \cell{PSNR} & \cell{Time[s]} & \cell{$f(\Zb^k)$} & \cell{$\frac{\Vert\Lc(\Zb^k) - \bb\Vert}{\norm{\bb}}$} & \cell{\texttt{Error}} & \cell{PSNR} & \cell{Time[s]} \\ \hline
\cell{ASGARD} &  \cell{39.927} & \cell{2.426e-03} & \cell{2.305e-02} & \cell{90.07} & \cell{172.85} & \cell{54.186} & \cell{2.206e-03} & \cell{3.179e-02} & \cell{85.81} & \cell{101.94} \\
\cell{ASGARD-restart} & \cell{40.126} & \cell{6.443e-04} & \cell{2.290e-02} & \cell{90.12} & \cell{173.41} & \cell{54.655} & \cell{7.122e-04} & \cell{3.109e-02} & \cell{86.00} & \cell{103.07} \\
\cell{ADSGARD} & \cell{39.372} & \cell{3.580e-03} & \cell{2.336e-02} & \cell{89.95} & \cell{210.36} & \cell{53.727} & \cell{3.353e-03} & \cell{3.238e-02} & \cell{85.65} & \cell{115.36} \\
\cell{Chambolle-Pock} &  \cell{39.931} & \cell{3.710e-02} & \cell{6.689e-02} & \cell{80.82} & \cell{160.12} & \cell{54.408} & \cell{5.748e-02} & \cell{1.232e-01} & \cell{74.04} & \cell{107.71} \\
\cell{Linesearch CP} & \cell{41.291} & \cell{4.514e-03} & \cell{2.563e-02} & \cell{89.15} & \cell{469.35} & \cell{54.720} & \cell{4.005e-03} & \cell{3.520e-02} & \cell{84.92} & \cell{317.09} \\
\hline
\end{tabular}
\end{footnotesize}
\end{center}
\end{table}
The performance and results of these algorithms are summarized in Table~\ref{tbl:MRI_results}, where $\texttt{Error} := \frac{\Vert\Zb^k - \Zb^{\natural}\Vert_F}{\Vert\Zb^{\natural}\Vert_F}$ presents the  error between the original image $\Zb^{\natural}$ to the reconstruction $\Zb^k$ after $k=500$ iterations.

As we can see both Algorithm~\ref{alg:pd_alg1} and Algorithm~\ref{alg:pd_alg2} have comparable performance with the line-search variant of Chambolle-Pock's method in terms of accuracy, while outperform the standard variant. In fact, our standard ASGARD method is still slightly better than this line-search version. The ASGARD with restart gives the best performance in terms of accuracy as well as PSNR (peak signal-to-noise ratio). 
The \texttt{Linesearch CP} is more than three times slower than the other methods in this experiment. 

The   reconstructed images are  revealed in Figure~\ref{fig:MRI_image}. 
As we can see from this plot, the quality of recovery image is very close to the original image for the sampling rate of $20\%$.
\begin{figure}[ht!]
\begin{center}
\includegraphics[width=1\linewidth]{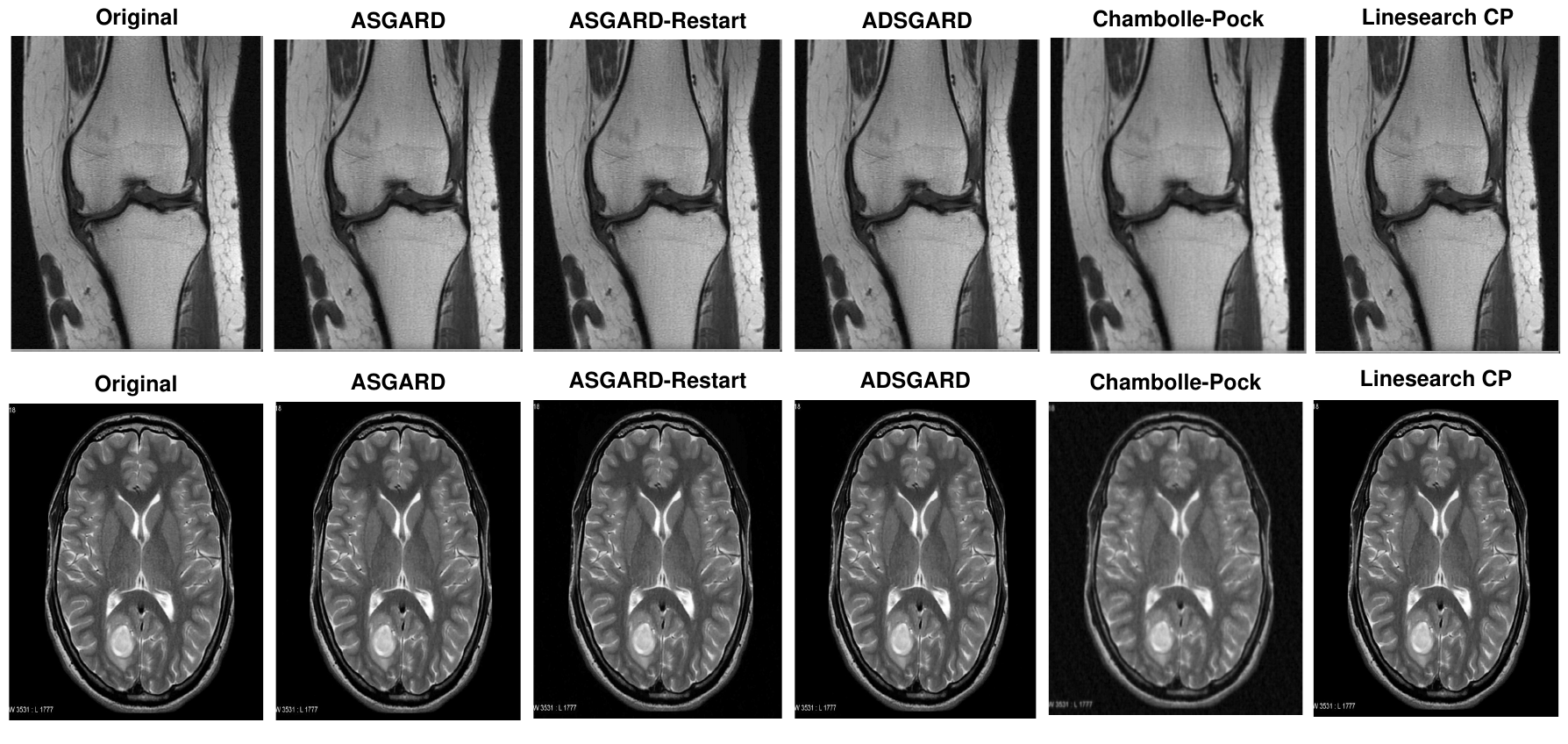}
\vspace{-4ex}
\caption{The original image and the reconstructed images of the five algorithms.}\label{fig:MRI_image}
\vspace{-1ex}
\end{center}
\end{figure}
Our algorithms give slightly higher PSNR for both images.
Chambolle-Pock's algorithm has the worst performance compared to the others.
However, the performance of this algorithm can slightly be changed and depends on the tuning strategy of the step-size parameters \cite{Chambolle2011} which is often very hard to tune a priori in practice without using heuristic strategy.

\section{A comparison between our results and existing methods}\label{sec:comparison}
We have presented a new primal-dual framework and two main algorithms (one with a primal flavor and one with a dual flavor) together with two special cases.
Now, let us summarize the main differences between our approach and existing methods in the literature.

The composite convex problem \eqref{eq:primal_cvx} can be written as a convex-concave saddle point problem:
\begin{equation}\label{eq:SP}
\min_{\xb\in\R^n}\max_{\yb\in\R^m}\set{ \Phi(\xb, \yb) := f(\xb) + \iprods{\Ab\xb, \yb} - g^{\ast}(\yb)}.
\end{equation}
The optimality condition of this problem is a maximal monotone inclusion, and can be reformulated as a variational inequality (VIP) \cite{Bauschke2011,Facchinei2003,Rockafellar2004}.
While \eqref{eq:SP} is classical, it has broad applications in image processing, machine learning, game theory among many others \cite{Bauschke2011,Combettes2011,Facchinei2003,Nemirovskii2004}.
Recent development in solution methods for solving \eqref{eq:SP} has attracted a great attention. Let us briefly survey some notable works which we find most related to ours.

Nemirovskii in \cite{Nemirovskii2004} proposed an averaging scheme to solve \eqref{eq:SP} based on its VIP formulation. 
His algorithm requires a proximal step at each iteration, which is usually not easy to compute in applications. 
He proved an $\mathcal{O}(1/k)$ convergence rate in an ergodic sense for the primal-dual gap function under the boundedness of both the primal and dual domains.
Nesterov in \cite{Nesterov2007a} proposed a similar method to solve VIP that covers \eqref{eq:SP} as a special case. 
By using smoothing techniques, he could prove an $\mathcal{O}(1/k)$ convergence rate for the gap function in an ergodic sense as in  \cite{Nemirovskii2004}. 
While this method has a simple subproblem, it requires  the underlying operator to be Lipschitz continuous, and both primal and dual domain are bounded.

One of the most celebrated  works for solving nonsmooth convex problem \eqref{eq:primal_cvx} is due to Nesterov in \cite{Nesterov2005c}.
By combining both the smoothing technique and his accelerated gradient-type method, he proposed an algorithm to solve  \eqref{eq:primal_cvx} just using proximal operators of $f$ and $g^{\ast}$. 
The method achieves an $\mathcal{O}(1/\varepsilon)$ complexity to obtain  an $\varepsilon$-solution. 
However, this algorithm has two disadvantages. 
First, it requires the boundedness of both the primal and dual domains. 
Second, the smoothness parameter depends on both the accuracy $\varepsilon$ and the diameter of the primal and dual domains. 
An improvement was proposed in \cite{Nesterov2005d} to remove the second disadvantage. 
But this algorithm requires a symmetric update which leads to a different per-iteration complexity than \cite{Nesterov2005c}.

Another remarkable work was proposed by Chambolle and Pock in \cite{Chambolle2011}. 
This algorithm solves \eqref{eq:SP} just using the proximal operator of $f$ and $g^{\ast}$. 
Similar to \cite{Nemirovskii2004}, they also proved an $\mathcal{O}(1/k)$ convergence rate on the gap function in an ergodic sense requiring the boundedness of both the primal and dual domain. 
An improvement on the parameter range was proposed in \cite{He2012}. However, the convergence guarantee remains preserved under the same assumptions.
Extensions of \cite{Chambolle2011} can be found in several papers, including \cite{combettes2012primal,Condat2013,malitsky2016first}.

Shefi and Teboulle provided a comprehensive study on the convergence rate of proximal-type methods for solving \eqref{eq:primal_cvx} in \cite{Shefi2014}, which extended the work \cite{Chen1994}.
They discussed different variants of the primal-dual proximal-type methods including Chambole-Pock's scheme, alternating minimization algorithms (AMA), and alternating direction methods of multipliers (ADMM).
With a proper choice of metric, they showed an $\mathcal{O}(1/k)$-convergence guarantee on the primal-dual gap function in an ergodic sense.
Their convergence guarantee indeed unifies several schemes, but is different from our results in this paper. 
First, they used different metric for proximal terms depending on $\Ab$. This makes the subproblem much harder to solve.
Second, they provided a guarantee for convergence rate on the gap function, which is not obvious to transform it into a separated guarantee for the objective and constraints in constrained convex optimization settings such as \eqref{eq:constr_cvx}. 
Moreover, the rate on the feasibility in the constrained setting reduces to $\mathcal{O}(1/\sqrt{k})$ (see \cite[Theorem 5]{Shefi2014}).
Finally, the gap function is defined on a given domain and it is not clear how to choose the radius of this domain.

Other methods for solving \eqref{eq:SP} can be found in the literature including \cite{Chambolle2011,Chen2013a,Davis2014a,Davis2014,Davis2014b,he2016accelerated,Ouyang2013,Ouyang2014}.
Each method requires different structure assumptions and achieves different guarantees mostly in an ergodic or averaging sense.
For instance, in \cite{Chen2013a}, the authors required $f$ to have a  Lipschitz gradient which is much more limited than our assumptions.
The authors in \cite{he2016accelerated} specified a so-called hybrid proximal extragradient (HPE) framework proposed in \cite{solodov1999hybrid} to solve \eqref{eq:SP}. 
While this method achieves an $\mathcal{O}(1/\varepsilon)$ complexity without any boundedness assumption on the domains, it is rather complicated due to a double loop of inner and outer iterations.

Regarding the constrained setting \eqref{eq:constr_cvx}, primal-dual first-order methods for directly solving  large-scaled settings of this problems are also well-developed.
Let us briefly discuss some of these methods here.
A natural approach is due to dual gradient-type methods. 
Such methods often use directly the subgradients of the dual function or smooth the dual function using proximity terms \cite{Necoara2008,necoara2014iteration}. 
While the former gives a slow convergence rate, which is $\mathcal{O}(1/\sqrt{k})$, the latter uses Nesterov's smoothing technique in \cite{Nesterov2005c}, and therefore faces the same drawbacks. In addition, for the setting \eqref{eq:constr_cvx}, the dual domain is unbounded, which leads to a difficulty to estimate the worst-case complexity bounds. 
Another approach is using penalty or augmented Lagrangian as considered in \cite{Lan2013,Lan2013b,Nedelcu2014}, which often leads a two loop algorithms and is much more complicated to control the specified parameters and accuracies in practice.
Alternating direction methods are perhaps the most common use for solving \eqref{eq:constr_cvx}, see \cite{Boyd2011,Chambolle2011,Davis2014a,Davis2014,Davis2014b,Ouyang2013,Ouyang2014}.
This method often requires an additional structure assumption such as $f$ is the sum of two separable convex functions. 
Without this structure, auxiliary variables need to be introduced, which is again equivalent to the two block case, see, e.g., \cite{Boyd2011,wang2013solving,Wei20131}.
Two common methods in this direction are alternating minimization algorithm (AMA) and alternating direction method of multipliers (ADMM).
While AMA can be viewed as a forward-backward splitting scheme for the dual formulation~\eqref{eq:dual_cvx}, and requires strict assumptions to guarantee convergence (e.g., one objective component is strongly convex), ADMM is equivalent to the Douglas-Rachford splitting method applying to the dual, and has a convergence guarantee under mild assumptions.
Recently, the convergence rate of ADMM has been attracted a great attention. There are vast of papers studied ADMM and its variants, including \cite{Boyd2011,Davis2014,He2012a,Ouyang2013,Ouyang2014}.
We would also like to mention that during the revision process of this paper, variants of ALM and ADMM with similar convergence rates as ours 
have been proposed in~\cite{xu2016accelerated}.

In summary, this paper has tried to overcome several issues we have mentioned above in recent primal-dual methods. 
Let us high-light the following characterizations.

\paragraph{Problem structure assumptions}
Our approach requires the convexity and the existence of primal and dual solutions of \eqref{eq:primal_cvx}. 
In the unconstrained setting \eqref{eq:primal_cvx}, we require $g$ to be Lipschitz continuous, which is often the case in practice.
We argue that such assumptions are mild for \eqref{eq:primal_cvx} and \eqref{eq:dual_cvx} and can be verified a priori. 
We emphasize that existing primal-dual methods in \cite{Chambolle2011,Davis2014a,Davis2014, Davis2014b,Ouyang2013,Ouyang2014,Shefi2014} require other structure assumptions on either $f$, such as Lipschitz gradient, strong convexity, error bound conditions, or the boundedness of both the primal and dual feasible sets, which may not be satisfied for  \eqref{eq:primal_cvx}, and especially for \eqref{eq:constr_cvx} \cite{Chambolle2011,Davis2014a,Davis2014,Davis2014b,Ouyang2013,Ouyang2014}.

\paragraph{Convergence characterization} 
We characterize an $\mathcal{O}(1/k)$-convergence rate for both the objective residual $f(\xb^k) - \fopt$ and the feasibility violation $\Vert\Ab\xb^k-\cb\Vert_{\Yc,\ast}$ for the constrained convex problem \eqref{eq:constr_cvx}.
Our finding is the best-known result under very mild assumptions and low per-iteration complexity.
In the composite form \eqref{eq:primal_cvx}, we also achieve the same $\mathcal{O}(1/k)$-convergence rate as in the seminar work \cite{Nesterov2005c}.
Tables~\ref{tbl:compare1} and \ref{tbl:compare2} compare our theoretical convergence rate results with the most recent selected algorithms in the literature for solving \eqref{eq:primal_cvx} and \eqref{eq:constr_cvx}, respectively.
\begin{landscape}
\centering 
\begin{table}[ht!]
\vspace{-3ex}
\caption{A comparison of convergence rates between our algorithms and selected existing methods for solving \eqref{eq:primal_cvx} and \eqref{eq:constr_cvx}. 
Here, all algorithms do not involve any large matrix inversion or ``complex'' convex subproblem, and $w_k := \frac{1}{k} \sum_{l=1}^k w^l$; $K$ is the iteration budget $($see Subsection~\ref{subsec:compare_55}$)$; and $\sigma$ is the step-size in \cite{Chambolle2011}.}\label{tbl:compare1}
\begin{tabular}{|l|p{1.9in}|p{1.9in}|p{2.8in}|}
\hline
\multicolumn{1}{|c|}{Paper}& $\dom{f}$ bounded\newline and $g$ Lipschitz & {~~~~~~~~~~~~~~~}$g$ Lipschitz & \hspace{20ex}$g = \delta_{\set{\cb}}$\newline{ \color{white}.~}\hspace{10ex}(optimality and feasibility) \\
\hline
Nesterov~\cite{Nesterov2005c} & $P(x^k) - P^{\star} \! \leq\! \frac{2 \sqrt{\bar L_A D_{\Xc} D_\Yc}}{K}\big(1 \! + \! \frac{K^2}{k^2}\big)  $ & $ P(x^k) - P^{\star}\leq \frac{\epsilon}{2} + \frac{8 \bar L_A \norm{x^0 - x^\star}^2 D_\Yc}{\epsilon (k+1)^2} $ & not applicable \\[2ex]
Chambolle-Pock~\cite{Chambolle2011}\hspace{-0.5em}~& $G(w_k) \leq \frac{\sigma \bar L_A D_\Xc + \sigma^{-1} D_\Yc}{k}$& convergence &  convergence \\[2ex]
ASGARD (Sec. \ref{sec:alg_scheme})&$P(x^k) - P^{\star} \leq \frac{2\sqrt{2}\sqrt{\bar L_A D_\Yc D_\Xc}}{k}$&
~\hspace{-0.5em}$P(x^k) - P^{\star} \leq \frac{\bar L_A}{2\beta_1 k}\norm{\bar x^0 \!-\! x^\star}^2
+ \frac{2 \beta_1}{k}D_\Yc$\hspace{-0.5em}~ &~\hspace{-1em} $|f(x^k)\! - \!f^\star| \leq \frac{\bar L_A}{\beta_1k}\norm{\bar x^0\! -\! x^\star}^2 \!+\! \frac{3 \beta_1}{k} \norm{\dot y\! -\! y^\star}^2 \! +\! \frac{\beta_1}{k} \norm{y^\star}^2 $ \hspace{-1em}~ \\
           &       &      & \!\!$\norm{A x^k - c}_{\Yc,\ast} \leq \frac{\beta_1}{k+1}\big(2 \norm{\dot y - y^\star} + \frac{\sqrt{\bar L_A}}{\beta_1}\norm{\bar x^0 - x^\star} \big)  $ \\[2ex]
ADSGARD (Sec. \ref{sec:algorithm3})\hspace{-0.5em}~& $G(w^k) \leq \frac{2\sqrt{\bar L_A D_\Yc D_\Xc}}{k}$  &~\hspace{-0em}$G(w^k) \leq \frac{\gamma_1}{k+1}\norm{\dot x \!-\! x^\star}^2 + \frac{2\bar{L}_{\Ab}}{\gamma_1 k}D_\Yc $\hspace{0em}~&\!\!$|f(x^k)\! -\! f^\star| \leq \frac{3\gamma_1}{k} \norm{\dot x \!-\! x^\star}^2  \!+\! \frac{2 \bar L_A}{\gamma_1k} \norm{\dot y \!-\! y^\star}^2 \!+\! \frac{L_A}{\gamma_1 k}\norm{y^\star}^2 $\!\!\\
& & &  $\norm{A x^k - c}_{\Yc,\ast}\leq \frac{\bar L_A}{\gamma_1 k}\big( 2 \norm{\dot y - y^\star} + 2 \gamma_1 \norm{\dot x - x^\star} \big)$\\[1ex]
\hline
\end{tabular}
\end{table}
\begin{table}[ht!]
\vspace{-0ex}
\centering
\caption{A comparison of convergence rates between our algorithms and selected existing methods for solving \eqref{eq:constr_cvx}. 
Here, all algorithms may involve  ``complex'' convex subproblems or matrix inversions; and $\rho$ is the penalty parameter in \cite{Lan2013b} and \cite{Monteiro2012b}.}\label{tbl:compare2}
\begin{tabular}{|l|c|} \hline
\multicolumn{1}{|c|}{Paper} & $g = \delta_{\set{\cb}}$ \\ \hline
ALM~\cite{Lan2013b} & $|f(x^k) - f^\star| \leq \frac{6}{\rho \sqrt{k}}\norm{y^\star}\norm{y^0 - y^\star}$\\
& $\norm{A x^k - c}_{\Yc,\ast} \leq \frac{3}{\rho \sqrt{k}}\norm{y^0-y^\star}$ \\[2ex]

ADMM$^1$~\cite{Monteiro2012b}$^2$ & $|f_1((x_1)_k)+f_2((x_2)_k) - f_1(x_1^\star) - f_2(x_2^\star)|
 \leq \frac{6 + 4 \sqrt{2}}{k} \big(\frac{1}{\rho}\norm{y^0 - y^\star}^2 + \rho \norm{x_1^0 - x_1^\star}^2_{A_1^* A_1}\big) $  \\
& $\norm{A_1 (x_1)_k + A_2 (x_2)_k - c} \leq \frac 2k \sqrt{\frac{1}{\rho^2}\norm{y^0 - y^\star}^2 + \norm{x_1^0 - x_1^\star}^2_{A_1^* A_1}} $  \\[2ex]

ASALGARD (Sec. \ref{subsec:alsmoother}) &  $|f(x^k) - f^\star|\leq \frac{10 \norm{y^\star}_\Yc \norm{\dot y - y^\star}_\Yc }{\gamma_0 (k+1)^2}$\\
&  $\norm{A x^k - c}_{\Yc,\ast} \leq \frac{8 \norm{\dot y - y^\star}_\Yc}{\gamma_0(k+2)^2} $ \\[1ex]
\hline
\end{tabular}
\vspace{-2ex}
\end{table}
\vspace{-1ex}
{\footnotesize \em $^1$Note that ADMM splits the objective in 2 parts, which may make it very efficient for some problems. 
$^2$The original result is stated by means of $\epsilon$-subdifferentials.}
\end{landscape}

\paragraph{Decomposition methods} 
Our algorithms naturally support decomposable structures in $f$ without either reformulating the problem as in ADMM or requiring additional assumptions as in parallel and multi-block ADMM \cite{lin2015sublinear}.
Both algorithms simply require only one proximal operator of $f$ and $g^{\ast}$, one matrix vector multiplication, and one adjoint per iteration.

\paragraph{Smoothing and smoothness parameter updates}
In contrast to proximal-type approaches in \cite{Monteiro2010a,Monteiro2011,Shefi2014,solodov1999hybrid} where the subproblem is often more complicated to solve, we instead exploit Nesterov's smoothing technique \cite{Nesterov2005c} which allows us to use proximal operators of $f$ and $g^{\ast}$.
However, we use differentiable smoothing functions as compared to Nesterov's smoothing technique in \cite{Nesterov2005c}. 
We propose explicit rules to update the smoothness parameters simultaneously at each iteration. 
We emphasize that this is one of the key contributions of this paper. 
To the best of our knowledge, this is the first adaptive primal-dual algorithms for smoothness parameters without sacrificing the $\mathcal{O}(1/k)$ rate and requiring additional assumptions.

\paragraph{Averaging vs.\ non-averaging} 
Most existing methods employ either non-weighted averaging \cite{Chambolle2011,He2012a,He2012,Shefi2014} or weighted averaging schemes \cite{Davis2014,Ouyang2013,Ouyang2014} to guarantee the $\mathcal{O}(1/k)$ rate on the primal sequence. 
While we also provide a weighted averaging scheme (Algorithm~\ref{alg:pd_alg2}), we alternatively derive a method (Algorithm~\ref{alg:pd_alg1}) without any averaging in the primal for solving \eqref{eq:primal_cvx}.
The non-averaging schemes are important since taking average may destroy key structures, such as the sparsity or low-rankness in sparse or low-rank optimization. 
Our weighted averaging scheme has increasing weight at the later iterates compared to non-weighted averaging schemes  \cite{Chambolle2011,He2012a,He2012,Shefi2014}. 
As indicated in \cite{Davis2014a,Davis2014b}, weighted averaging schemes  has better performance guarantee than non-weighted ones. 

We have attempted  to review various primal-dual methods which are most related to our work. 
It is still worth mentioning other primal-dual methods that are based on augmented Lagrangian methods such as alternating direction methods (e.g., AMA, ADMM and their variants)  \cite{Boyd2011,Lan2013,Lan2013b,Tseng1991a}, Bregman and other splitting methods \cite{Bauschke2011,combettes2012primal,Esser2010a,Goldstein2013,Monteiro2010,Monteiro2011,Monteiro2012b}, and using variational inequality frameworks \cite{Chambolle2011,He2012,He2012b}.
While most of these works have not considered the global convergence rate of the proposed algorithms, a few of them characterized the convergence rate in unweighted averaging schemes or used a more general variational inequality/monotone inclusion framework to study \eqref{eq:primal_cvx}, \eqref{eq:dual_cvx} and \eqref{eq:constr_cvx}. Hence, the results achieved in these papers are distinct from our findings.

\begin{footnotesize}
\section*{Acknowledgments}
This work is supported in part by the NSF-grant No. DMS-1619884, USA; and 
the European Commission under Grant ERC Future Proof, SNF 200021-146750, and SNF CRSII2-147633.
We are thankful Ahmet Alacaoglu, Baran Gozc\"{u}, and Alp Yurtsever for their help on the last numerical example, and Van Quang Nguyen for his careful proofreading.
\end{footnotesize}

\appendix
\normalsize
\section{The proof of theoretical results}\label{sec:appendix}
This section provides the full proof of Lemmas and Theorems  in the main text.

\subsection{Technical results}\label{apdx:tec_results}
We first prove the following basic lemma, which will be used to analyze the convergence of our algorithms in the main text.

\begin{lemma}\label{lem:move_beta}
Let $h$ be a proper, closed and convex function defined on $\Zc$, and $h^{\ast}$ is its Fenchel conjugate. 
Let $b_{\Zc}$ be a Bregman distance as defined in \eqref{eq:bregman_dist} with a weighted norm. 
We define a smoothed approximation of $h$ as
\begin{equation}\label{eq:h_beta}
h_{\beta}(\zb;\dot{\zb}) := \max_{\hat{\zb}\in\Zc}\set{ \iprods{\zb, \hat{\zb}} - h^{\ast}(\hat{\zb}) - \beta b_{\Zc}(\hat{\zb}, \dot{\zb})},
\end{equation}
where $\dot{\zb}\in\Zc$ is fixed and $\beta > 0$ is a smoothness parameter.
We also denote by $\zb^{\ast}_{\beta}(\zb;\dot{\zb})$ the solution of the maximization problem in \eqref{eq:h_beta}.
Then, the following facts hold:
\begin{enumerate}
\item[$\mathrm{(a)}$] We have a relation between the partial derivatives of $(\zb, \beta) \mapsto h_{\beta}(\zb; \dot{\zb})$ as
\begin{equation*}
\frac{\partial{h_{\beta}}(\zb; \dot{\zb})}{\partial{\beta}}(\beta) = -b_{\Zc}(\zb_\beta^{\ast}(\zb; \dot{\zb}), \dot{\zb}) = - b_{\Zc}(\nabla{h}_\beta(\zb; \dot{\zb}), \dot{\zb}).
\end{equation*}

\item[$\mathrm{(b)}$] For all $\zb\in\Zc$, $\beta \mapsto h_{\beta}(\zb; \dot{\zb})$ is convex, and for $\beta_{k+1} ,  \beta_k > 0$ and $\bar{\zb} \in \Zc$, we have
\begin{align}\label{eq:convexity_h_beta}
h_{\beta_{k+1}}(\bar{\zb}; \dot{\zb}) &\leq h_{\beta_k}(\bar{\zb}; \dot{\zb}) - (\beta_k - \beta_{k+1}) \frac{\partial{h_{\beta}}(\bar{\zb};  \dot{\zb})}{\partial{\beta}}(\beta_{k+1})  \\
&= h_{\beta_k}(\bar{\zb}; \dot{\zb}) + (\beta_k - \beta_{k+1})b_{\Zc}(\nabla{h}_{\beta_{k+1}}(\bar{\zb}; \dot{\zb}), \dot{\zb}). \notag
\end{align}

\item[$\mathrm{(c)}$]  $h_{\beta}(\cdot; \dot{\zb})$ has a $1/\beta$-Lipschitz gradient in $\norm{\cdot}_{\Zc,*}$. 
Hence, for all $\bar{\zb}, \hat{\zb} \in \Zc$, we have
\begin{align}\label{eq:taylor_nabla_h_beta}
&h_{\beta}(\bar{\zb}; \dot{\zb}) \leq h_{\beta} (\hat{\zb}; \dot{\zb}) + \iprods{\nabla{h}_\beta(\hat{\zb}; \dot{\zb}), \bar{\zb} - \hat{\zb}} + \frac{1}{2\beta} \norm{\bar{\zb} - \hat{\zb}}_{\Zc,*}^2.\\
\label{eq:cocoercivity_nabla_h_beta}
&h_{\beta}(\hat{\zb}; \dot{\zb}) +\iprods{\nabla{h}_{\beta}(\hat{\zb}; \dot{\zb}), \bar{\zb} - \hat{\zb}} \leq h_{\beta}(\bar{\zb}; \dot{\zb}) - \frac{\beta}{2} \norm{\nabla{h}_{\beta}(\hat{\zb};\dot{\zb}) - \nabla{h}_\beta(\bar{\zb}; \dot{\zb})}_{\Zc}^2.
\end{align}

\item[$\mathrm{(d)}$] Both functions $h$ and $h_{\beta}$ evaluated at different points $\zb, \hat{\zb} \in \Zc$ satisfy  
\begin{equation}\label{eq:magic_return_to_zero}
h_{\beta}(\hat{\zb}; \dot{\zb}) + \iprods{\nabla{h}_\beta(\hat{\zb}; \dot{\zb}), \zb - \hat{\zb}}  \leq h(\zb) - \beta b_{\Zc}(\nabla{h}_\beta(\hat{\zb}; \dot{\zb}), \dot{\zb}).
\end{equation}

\item[$\mathrm{(e)}$] If $\norm{\cdot}_{\Zc}$ is derived from a scalar product, then, for all $\tau > 0$, $\bar{\zb}, \hat{\zb} \in \Zc$, we have
\begin{equation} \label{eq:combine_tau}
{\!}(1 \!-\! \tau) \norm{\nabla{h}_{\beta}(\hat{\zb}; \dot{\zb}) \!-\! \nabla{h}_{\beta}(\bar{\zb}; \dot{\zb}) }^2_{\Zc} +  \tau \norm{\nabla{h}_{\beta}(\hat{\zb}; \dot{\zb}) \!-\! \dot{\zb}}^2_{\Zc} \geq \tau (1\!-\! \tau) \norm{\nabla{h}_{\beta}(\bar{\zb}; \dot{\zb}) \!-\! \dot{\zb}}^2_{\Zc}.{\!\!\!\!\!\!}
\end{equation}

\item[$\mathrm{(f)}$] We can control the influence of a change in the center points from $\dot{\zb}_1$ to $\dot{\zb}_2$ using the following estimate: 
\begin{equation}\label{eq:move_ydot}
\begin{array}{ll}
h_{\beta}(\zb; \dot{\zb}_2) &\leq~ h_{\beta}(\zb; \dot{\zb}_1) - \frac{\beta}{2}\norm{\zb^{\ast}_{\beta}(\zb; \dot{\zb}_1) - \zb^{\ast}_{\beta}(y; \dot{\zb}_2)}^2_{\Zc} \vspace{0.75ex}\\
&   +~ \beta \left[ b_{\Zc}(\zb^{\ast}_{\beta}(\zb; \dot{\zb}_2), \dot{\zb}_1) -  b_{\Zc}(\zb^{\ast}_{\beta}(\zb; \dot{\zb}_2), \dot{\zb}_2)\right].
\end{array}
\end{equation}
\end{enumerate}
\end{lemma}

\begin{proof}
We prove from item~$\mathrm{(a)}$ to item $\mathrm{(f)}$ as follows.

\indent$\mathrm{(a)}$
Since $h_{\beta}(\yb)$ is defined by the maximization of a strongly convex program in \eqref{eq:h_beta}, where the function in the $\max$ operator is linear in $\beta$ and convex in $\zb$, the minimizer $\zb^{\ast}_{\beta}(\zb;\dot{\zb})$ is unique. 
By the classical marginal derivative theorem \cite{Rockafellar1970}, the function is differentiable with respect to $\beta$ and $\zb$. 
In addition, $\nabla_{\zb}{h_{\beta}}(\zb;\dot{\zb}) = \zb^{\ast}_{\beta}(\zb;\dot{\zb})$.

\indent$\mathrm{(b)}$
The function $\beta \mapsto h_{\beta}(\zb; \dot{\zb})$ is the maximization of a linear function in $\beta$ indexed by $y$ and $\dot y$. Hence, it is convex. 
The remaining statement follows by the convexity of $h_{\beta}$ with respect to  $\beta$ and item~$\mathrm{(a)}$.

\indent$\mathrm{(c)}$
Since $\beta b_{\Zc}(\cdot, \dot{\zb})$ is $\beta$-strongly convex in the weighted-norm $\norm{\cdot}_{\Zc}$, $h_{\beta}(\cdot; \dot{\zb})$ is $\frac{1}{\beta}$-Lipschitz~\cite{Nesterov2005c} in the corresponding dual norm. The inequalities \eqref{eq:taylor_nabla_h_beta}  and \eqref{eq:cocoercivity_nabla_h_beta} are classical  for convex functions with Lipschitz gradient \cite{Nesterov2004}.

\indent$\mathrm{(d)}$
Let us denote here $\hat{\zb}^{*}_{\beta} := \zb^{*}_{\beta}(\hat{\zb}; \dot{\zb})$. Then, we can derive
\begin{align*}
h_{\beta}(\hat{\zb};\dot{\zb}) + \iprods{\nabla{h}_{\beta}(\hat{\zb};\dot{\zb}), \zb  - \hat{\zb}} &= \big( \iprods{\hat{\zb}, \hat{\zb}^{\ast}_{\beta}} - h^{\ast}(\hat{\zb}^{\ast}_{\beta}) - \beta b_{\Zc}(\hat{\zb}^{\ast}_{\beta},\dot{\zb}) \big) + \iprods{\hat{\zb}^{\ast}_{\beta}, \zb -\hat{\zb}} \\
& = \iprods{\zb, \hat{\zb}^{\ast}_{\beta}} - h^{\ast}(\hat{\zb}^{\ast}_{\beta}) - \beta b_{\Zc}(\hat{\zb}^{\ast}_{\beta}, \dot{\zb}) \\
& \leq  \max_{\ub \in \Zc}\set{\iprods{\zb , \ub} - h^{\ast}(\ub)} - \beta b_{\Zc}(\zb^{\ast}_{\beta}, \dot{\zb}) \\
& = h(\zb) - \beta b_{\Zc}(\nabla{h}_{\beta}(\hat{\zb};\dot{\zb}), \dot{\zb}).
\end{align*}

\indent$\mathrm{(e)}$
The elementary equality $\norm{(1-\tau) \ab + \tau \cb}^2 = (1-\tau) \norm{\ab}^2 + \tau \norm{\cb}^2 - \tau (1-\tau) \norm{\ab-\cb}^2$ directly implies the result for any norm $\norm{\cdot}$ deriving from a scalar product.

\indent$\mathrm{(f)}$
Let us denote by $\zb^{\ast}_{\beta,1} = \zb^{*}_{\beta}(\zb; \dot{\zb}_1)$ and $\zb^{\ast}_{\beta,2} := \zb^{\ast}_{\beta}(\zb;\dot{\zb}_2)$. 
Using the definition of $h_{\beta}$ in \eqref{eq:h_beta} and its optimality condition, we can derive
\begin{align*}
h_{\beta}(\zb; \dot{\zb}_2) &= \max_{\hat{\zb} \in \Zc} \set{\iprods{\zb, \hat{\zb}} - h^{\ast}(\hat{\zb}) - \beta b_{\Zc}(\hat{\zb}, \dot{\zb}_2)} 
= \iprods{\zb, \zb^{\ast}_{\beta,2}}  - h^{\ast}(\zb^{\ast}_{\beta,2}) -\beta b_{\Zc}(\zb^{\ast}_{\beta,2}, \dot{\zb}_2) \\
& = \Big( \iprods{\zb,\zb^{\ast}_{\beta,2}} - h^{\ast}(\zb^{\ast}_{\beta,2})  - \beta b_{\Zc}(\zb^{\ast}_{\beta,2}, \dot{\zb}_1) \Big) + \beta b_{\Zc}(\zb^{\ast}_{\beta,2}, \dot{\zb}_1) - \beta b_{\Zc}(\zb^{\ast}_{\beta,2}, \dot{\zb}_2) \\
&\leq{\!}  \iprods{\zb, \zb^{\ast}_{\beta,1}} \!-\! h^{\ast}(\zb^{\ast}_{\beta,1})  \!-\! \beta b_{\Zc}(\zb^{\ast}_{\beta,1}, \dot{\zb}_1)
\!-\! \frac{\beta}{2}\norm{\zb^{\ast}_{\beta,1} \!-\! \zb^{\ast}_{\beta,2}}^2_{\Yc} \!+\! \beta b_{\Zc}(\zb^{\ast}_{\beta,2}, \dot{\zb}_1) \!-\! \beta b_{\Zc}(\zb^{\ast}_{\beta,2}, \dot{\zb}_2) \\
& = h_{\beta}(\zb; \dot{\zb}_1) - \frac{\beta}{2}\norm{\zb^{\ast}_{\beta,1} -\zb^{\ast}_{\beta,2}}^2_{\Yc}   + \beta \left(b_{\Zc}(\zb^{\ast}_{\beta,2}, \dot{\zb}_1) -  b_{\Zc}(\zb^{\ast}_{\beta,2}, \dot{\zb}_2)\right),
\end{align*}
which proves \eqref{eq:move_ydot}.
\end{proof}

\subsection{The proof of Lemma \ref{le:excessive_gap_aug_Lag_func}: Key bounds for approximate solutions}\label{apdx:le:excessive_gap_aug_Lag_func}

We consider the smooth objective residual $S_{\beta}(\xb;\dot{\yb}) := \big(f(\xb) + g_\beta(\Ab\xb; \dot{\yb})\big) - \big(  f(\xopt)  + g(\Ab\xopt) \big)$. 
By using the definition of $g_{\beta}$, we can derive that
\begin{align}\label{eq:hbeta2hstar}
g_\beta(\Ab\xb; \dot{\yb}) &= \max_{\hat{\yb} \in \Yc} \set{ \iprods{\Ab\xb, \hat{\yb}} - g^{\ast}(\hat{\yb}) - \beta b_{\Yc}(\hat{\yb}, \dot{\yb})} \notag \\
&\geq \iprods{\Ab\xb , \yopt} - g^{*}(\yopt) - \beta b_{\Yc}(\yopt, \dot{\yb}) \notag\\
& = \iprods{\Ab\xb - \Ab \xopt, \yopt} + \iprods{\Ab\xopt , \yopt} - g^*(\yopt) - \beta b_{\Yc}(\yopt, \dot{\yb}) \notag \\
& = \iprods{\Ab(\xb - \xopt), \yopt} + g(\Ab \xopt) - \beta b_{\Yc}(\yopt, \dot{\yb}),
\end{align}
where the last line is the equality case in the Fenchel-Young inequality using the fact that $\Ab\xopt \in \partial{g}^{*}(\yopt)$.
Similarly, we have 
\begin{align}\label{eq:f_ast_gamma_est}
f^{\ast}_\gamma(-\Ab^\top\yb; \dot{\xb}) &= \max_{\hat{\xb} \in \Xc}\set{ \iprods{ -\Ab^\top\yb, \hat{\xb}} - f(\hat{\xb}) - \gamma b_{\Xc}(\hat{\xb}, \dot{\xb})} \notag\\
&\geq \iprods{\Ab^\top(\yopt - \yb), \xopt} + f^{\ast}(-\Ab^\top \yopt) - \gamma b_{\Xc}(\xopt, \dot{\xb}).
\end{align}
Combining \eqref{eq:f_ast_gamma_est}, the definition \eqref{eq:smoothed_gap_func} of $G_{\gamma\beta}(\cdot;\dot{\wb})$, and the strong duality condition \eqref{eq:strong_duality}, we can show that
\begin{align}\label{eq:GandS}
G_{\gamma\beta}(\wb;\dot{\wb}) &:= P_{\beta}(\xb;\dot{\yb}) - D_{\gamma}(\yb;\dot{\xb}) \notag\\
& = f(\xb) + g_{\beta}(\Ab\xb; \dot{\yb}) + f^{*}_{\gamma}(-\Ab^{\top}\yb; \dot{\xb}) + g^{*}(\yb) \notag \\
&\overset{\tiny\eqref{eq:strong_duality}}{=} S_{\beta}(\xb;\dot{\yb}) + f^{\ast}_{\gamma}(-\Ab^\top\yb; \dot{\xb}) + g^{\ast}(\yb) - f^{*}(-\Ab^\top\yopt) - g^{*}(\yopt)\notag \\
& \overset{\tiny\eqref{eq:f_ast_gamma_est}}{\geq} S_{\beta}(\xb;\dot{\yb}) + \iprods{\Ab^\top(\yopt - \yb), \xopt} + g^{*}(\yb) - g^{*}(\yopt) - \gamma b_{\Xc}(\xopt, \dot{\xb}) \notag \\
& \geq S_{\beta}(\xb;\dot{\yb}) -  \gamma b_{\Xc}(\xopt, \dot{\xb}), 
\end{align}
where the last inequality holds because $g^{*}$ is convex and $\Ab \xopt \in \partial{g}^*(\yopt)$ due to \eqref{eq:opt_cond}. This proves the first inequality of \eqref{eq:S_bound}.

Since $b_{\Yc}(\cdot, \dot{\yb})$ is $1$-strongly convex with respect to the weighted-norm, using the optimality condition of the maximization problem in \eqref{eq:g_beta} at $\yb := \yopt$, and $\ub := \Ab\xb$, we obtain
\begin{equation}\label{eq:opt_cond_max2}
g_{\beta}(\Ab\xb;\dot{\yb}) \geq \iprods{\Ab\xb,\yopt} - g^{\ast}(\yopt) - \beta b_{\Yc}(\yopt, \dot{\yb}) + \frac{\beta}{2}\snorm{\yb^{\ast}_{\beta}(\Ab\xb;\dot{\yb}) - \yopt}_{\Yc}^2. 
\end{equation}
By \eqref{eq:opt_cond}, we have $-\Ab^{\top}\yopt\in\partial{f}(\xopt)$. Using this and the convexity of $f$, we have  $f(\xb) \geq f(\xopt) - \iprods{\Ab(\xb - \xopt),\yopt}$. 
Summing up the last inequality and \eqref{eq:opt_cond_max2}, then using the definition of $S_{\beta}(\xb;\dot{\yb})$, we obtain
\begin{align*}
 \frac{\beta}{2}\snorm{\yb^{\ast}_{\beta}(\Ab\xb;\dot{\yb}) \!-\! \yopt}_{\Yc}^2 \leq \beta b_{\Yc}(\yopt, \dot{\yb}) \!+\! S_{\beta}(\xb;\dot{\yb}) \!+\! g(\Ab\xopt) \!+\! g^{\ast}(\yopt) \!-\! \iprods{\Ab\xopt,\yopt} \leq \beta b_{\Yc}(\yopt, \dot{\yb}) + S_{\beta}(\xb;\dot{\yb}),
\end{align*}
which implies the second estimate in  \eqref{eq:S_bound}, where the last inequality is due to the Fenchel-Young equality $g(\Ab\xopt) + g^{\ast}(\yopt) = \iprods{\Ab\xb^{\star}, \yb^{\star}}$, and $\Ab\xopt\in\partial{g^{\ast}}(\yopt)$.

Now, we consider the choice $g(\cdot) := \delta_{\set{\cb}}(\cdot)$ in the constrained setting \eqref{eq:constr_cvx}.
Under Assumption A.\ref{as:A1}, any $\wopt := (\xopt, \yopt) \in\Wopt$ is a saddle point of the Lagrange function $\mathcal{L}(\xb, \yb) := f(\xb) + \iprod{\Ab\xb - \cb, \yb}$, i.e., $\Lc(\xb^{\star},\yb) \leq \Lc(\xb^{\star}, \yopt) \leq \Lc(\xb, \yopt)$ for all $\xb\in \Xc$ and $\yb\in\R^m$. The dual function $D$ in \eqref{eq:dual_cvx} becomes $D(\yb) := -f^{\ast}(-\Ab^{\top}\yb) - \cb^{\top}\yb = \min_{\xb}\set{f(\xb) + \iprods{\Ab\xb - \cb,\yb}}$.
It leads to $D(\yb) \leq D(\yb^{\star}) = f(\xopt) \leq f(\xb) + \iprods{\yopt, \Ab\xb - \cb}$, and hence
\vspace{-0.75ex}
\begin{align}\label{eq:pd_lower_bound}
f(\xb)  -  D(\yb) \geq f(\xb) - f(\xopt) \geq \iprod{\cb - \Ab\xb, \yopt} \geq -\norm{\yopt}_{\Yc}\norm{\Ab\xb - \cb}_{\Yc, *},
\vspace{-0.75ex}
\end{align}
for all $(\xb, \yb)\in \Wc$, which proves \eqref{eq:lower_bound}.

Finally, we prove \eqref{eq:main_bound_gap_3}. 
Indeed, using the definition of $g$ and $g_{\beta}$, and $\Ab\xopt = \cb$, we can write
\begin{align*}
f(\xb) - f(\xopt) &= f(\xb) + g_\beta(\Ab\xb; \dot{\yb}) - f(\xopt) - g(\Ab\xopt)  - g_\beta(\Ab\xb; \dot{\yb})  + g(\Ab\xopt) \\ 
&= S_{\beta}(\xb;\dot{\yb}) - g_\beta(\Ab\xb; \dot{\yb})  + g(\Ab\xopt) \overset{\tiny\eqref{eq:hbeta2hstar}}{\leq} S_{\beta}(\xb;\dot{\yb}) - \iprods{\Ab(\xb - \xopt) , \yopt} + \beta b_{\Yc}(\yopt, \dot{\yb}) \\
& \overset{\tiny\eqref{eq:GandS}}{\leq} G_{\gamma\beta}(\wb;\dot{\wb}) + \iprods{\cb - \Ab\xb, \yopt} + \beta b_{\Yc}(\yopt, \dot{\yb}) + \gamma b_{\Xc}(\xopt, \dot{\xb}).
\end{align*}
We then use the second inequality of \eqref{eq:pd_lower_bound} to get
\begin{align}\label{eq:bound_f}
\iprods{ \yopt, \cb - \Ab\xb} &\leq f(\xb) - f(\xopt) =  S_{\beta}(\xb;\dot{\yb}) - g_\beta(\Ab\xb; \dot{\yb})  + g(\Ab\xopt) = S_{\beta}(\xb;\dot{\yb}) - g_\beta(\Ab\xb; \dot{\yb}),
\end{align}
where $g(\Ab\xopt) = 0$ due to the feasibility of $\xopt$, i.e., $\Ab\xopt = \cb$. 
Now, it is obvious that
\begin{equation*}
g_\beta(\Ab\xb; \dot{\yb}) := \sup_{\hat{\yb} \in \Yc} \set{ \iprods{\Ab\xb - \cb, \hat{\yb}} - \beta b_{\Yc}(\hat{\yb}, \dot{\yb}) }\geq \iprods{\Ab\xb - \cb, \yopt} - \beta b_{\Yc}(\yopt, \dot{\yb}).
\end{equation*}
Hence, combining this estimate and \eqref{eq:bound_f} we obtain the first inequality in \eqref{eq:main_bound_gap_3}. 

As $\nabla{b}_{\Yc}(\cdot, \dot{\yb})$ is $L_{b_{\Yc}}$-Lipschitz continuous, $b_{\Yc}(\dot{\yb}, \dot{\yb}) = 0$ and $\nabla{b}_{\Yc}(\dot{\yb}, \dot{\yb}) = 0$, we have 
\vspace{-0.75ex}
\begin{align*}
g_\beta(\Ab\xb; \dot{\yb}) & = \sup_{\hat{\yb} \in \Yc} \set{ \iprods{\Ab\xb - \cb, \hat{\yb}} - \beta b_{\Yc}(\hat{\yb}, \dot{\yb}) } \geq \sup_{\hat{\yb} \in \Yc} \set{ \iprods{\Ab\xb - \cb, \hat{\yb}} - \frac{\beta L_{b_{\Yc}}}{2} \norm{ \hat{\yb} - \dot{\yb}}_{\Yc}^2} \\
& = \frac{1}{2\beta L_{b_{\Yc}}} \norm{\Ab\xb - \cb}_{\Yc, \ast}^2 + \iprods{\dot{\yb}, \Ab\xb - \cb}. 
\end{align*}
The last equality comes from the formula of the Fenchel conjugate of the squared norm.
Combining this inequality and \eqref{eq:bound_f}, we obtain
\begin{align*} 
\iprods{\yopt, \cb - \Ab x} &\leq  S_{\beta}(\xb;\dot{\yb}) - \frac{1}{2L_{b_y}\beta}\norm{\Ab x - \cb}^2_{\Yc, *} - \iprods{\dot y, \Ab x - \cb}
\end{align*}
Rearranging this expression and using the Cauchy-Schwarz inequality, we obtain $-\norm{\yopt - \dot{\yb}}_{\Yc}\norm{\Ab\xb - \cb}_{\Yc, *} \leq S_{\beta}(\xb;\dot{\yb}) - (2L_{b_{\Yc}}\beta)^{-1}\norm{\Ab\xb - \cb}^2_{\Yc,*}$, which leads to
\begin{align*}
\norm{\Ab\xb - \cb}^2_{\Yc, *} - 2\beta L_{b_{\Yc}}\norm{\yopt - \dot{\yb}}_{\Yc}\norm{\Ab\xb - \cb}_{\Yc, *} - 2L_{b_{\Yc}}\beta S_{\beta}(\xb;\dot{\yb}) \leq 0.
\end{align*}
Let $t := \Vert\Ab\xb - \cb\Vert_{\Yc, *}$. 
The last inequality becomes $t^2 - 2\beta L_{b_{\Yc}} \norm{\yopt - \dot{\yb}}_{\Yc}t - 2L_{b_{\Yc}}\beta S_{\beta}(\xb;\dot{\yb}) \leq 0$.
This inequation in $t$ has solution. Hence, $\norm{\yopt - \dot{\yb}}_{\Yc}^2 + 2L_{b_{\Yc}}^{-1}\beta^{-1}S_{\beta}(\xb;\dot{\yb}) \geq 0$ and
\begin{equation*}
t := \norm{\Ab\xb - \cb}_{\Yc,*} \leq \beta L_{b_{\Yc}}\Big[ \norm{\yopt - \dot{\yb}}_{\Yc}  + \big(\norm{\yopt - \dot{\yb}}_{*}^2 + 2L_{b_{\Yc}}^{-1}\beta^{-1}S_{\beta}(\xb;\dot{\yb}) \big)^{1/2}\Big],
\end{equation*}
which is the second estimate of \eqref{eq:main_bound_gap_3}. 
\Eproof

\vspace{-2ex}
\subsection{The convergence analysis of the \ref{eq:ac_pd_scheme} method}\label{apdx:subsec:convergence_analysis1}
In this appendix, we provide the full convergence analysis of \ref{eq:ac_pd_scheme}.
First, we prove a key inequality to guarantee the optimality gap reduction condition.

\begin{lemma}\label{le:key_estimate_1a}
Let us define $S_{\beta_k}(\bar{\xb}^k;\dot{\yb}) := P_{\beta_k}(\xbar^k;\dot{\yb}) - P^{\ast} = f(\bar{\xb}^{k}) + g_{\beta_{k}}(\Ab\bar{\xb}^{k}; \dot{\yb}) - f(\xopt) - g(\Ab\xopt)$. 
If $\tau_k \in (0, 1]$, then 
\begin{align}\label{eq:key_estimate_1a}
S_{\beta_{k+1}}(\bar{\xb}^{k+1};\dot{\yb}) &+ \frac{\bar{L}_\Ab \tau_k^2}{2\beta_{k+1}} \norm{\xtilde^{k+1} - \xopt}_{\Xc}^2  \leq (1-\tau_k) S_{\beta_k}(\bar{\xb}^k;\dot{\yb})  + \frac{\bar{L}_\Ab \tau_k^2}{2\beta_{k+1}} \norm{\xtilde^{k} - \xopt}_{\Xc}^2 \notag\\
& + \frac{(1-\tau_k)}{2}\left[ (\beta_k - \beta_{k+1})L_{b_{\Yc}}- \beta_{k+1}\tau_k\right] \big\Vert\nabla{g}_{\beta_{k+1}}(\Ab\bar{\xb}^k; \dot{\yb}) - \dot{\yb}\big\Vert^2_{\Yc}.
 \end{align}
\end{lemma}

\begin{proof}
Using Lemma~\ref{lem:move_beta} with $h := g$, $h_{\beta} := g_{\beta}$, $\Zc := \Yc$, and $\zb := \Ab\xb$, we can  proceed as
\begin{align}\label{eq:primal_comparison}
f(\bar x^{k+1}) &+ g_{\beta_{k+1}}(\Ab\xbar^{k+1}; \dot{\yb}) \overset{\eqref{eq:taylor_nabla_h_beta}}{\leq} f(\bar x^{k+1}) + g_{\beta_{k+1}}(\Ab\hat{\xb}^{k}; \dot{\yb}) + \iprods{\nabla{g}_{\beta_{k+1}}(\Ab\hat{\xb}^k; \dot{\yb}), \Ab \xbar^{k+1} - \Ab\hat{\xb}^k} \notag\\
&  + \frac{1}{2\beta_{k+1}}\norm{\Ab\hat{\xb}^k - \Ab\bar{\xb}^{k+1}}^2_{\Yc, *} \notag \\
&\hspace{-0.7em}\overset{\nabla{g}_{\beta} = \yb^*_\beta}{\leq} f(\bar{\xb}^{k+1}) + g_{\beta_{k+1}}(\Ab \hat{\xb}^{k}; \dot{\yb}) + \iprods{\Ab^{\top}\yb^{\ast}_{\beta_{k+1}}(\Ab\hat{\xb}^k; \dot{\yb}), \xbar^{k+1} - \xhat^k }  + \frac{\bar{L}_\Ab}{2 \beta_{k+1}}\norm{\xhat^k - \xbar^{k+1}}^2_{\Xc} \\
&{\!\!\!} \hspace{-1.2em}\overset{\tiny\text{def. of }~\xbar_{k\!+\!1}}{\leq}{\!\!\!\!\!} f(\xb) \!+\! g_{\beta_{k\!+\!1}}(\Ab \xhat^{k}; \dot{\yb}) + \iprods{\Ab^{\top}\yb^{\ast}_{\beta_{k\!+\!1}}(\Ab\xhat^k; \dot{\yb}), \xb \!-\! \xhat^k}  + \frac{\bar{L}_\Ab}{2 \beta_{k\!+\!1}}\left[\norm{\xhat^k \!-\! \xb}^2_{\Xc}  - \norm{\xbar^{k\!+\!1} \!-\! \xb}^2_{\Xc}\right], \notag
\end{align}
where the last inequality comes from the definition of $\xbar^{k+1}$ by using its optimality condition and the functions value at $\xb \in \Xc$. 

Our next step is to choose $\xb := (1-\tau_k)\xbar^k + \tau_k\xopt$. In this case, we have 
\begin{equation*}
\begin{array}{lll}
\xb - \hat{\xb}^k  &= (1-\tau_k)\bar{\xb}^k + \tau_k \xopt - (1-\tau_k)\bar{\xb}^k - \tau_k \tilde{\xb}^k &= \tau_k (\xopt - \tilde{\xb}^k), \vspace{0.75ex}\\
\xb - \hat{\xb}^k  &= (1-\tau_k)\xbar^k + \tau_k\xopt - (1-\tau_k)\xhat^k - \tau_k\xhat^k &= (1-\tau_k)(\bar{\xb}^k - \hat{\xb}^k) + \tau_k (\xopt - \hat{\xb}^k), \vspace{0.75ex}\\ 
\xb - \xbar^{k+1} &= (1-\tau_k)\xbar^k + \tau_k\xopt - \hat{\xb}^k - \tau_k(\tilde{\xb}_{k+1} - \tilde{\xb}_k) &= \tau_k (\xopt - \tilde{\xb}^{k+1}).
\end{array}
\end{equation*}
Now, we plug these expressions into \eqref{eq:primal_comparison} and using the convexity of $f$, we can derive
\begin{align*}
f(\xbar^{k+1}) &+ g_{\beta_{k+1}}(\Ab\xbar^{k+1}; \dot{\yb})   \leq (1-\tau_k)f(\xbar^k) + \tau_kf(\xopt) + g_{\beta_{k+1}}(\Ab\xhat^{k}; \dot{\yb}) \nonumber\\
&\qquad \qquad + \tau_k \iprods{\Ab^{\top}\yb^{\ast}_{\beta_{k+1}}(\Ab\xhat^k; \dot{\yb}), \xopt - \xhat^k}  + (1-\tau_k) \iprods{\Ab^{\top}\yb^{\ast}_{\beta_{k+1}}(\Ab\xhat^k; \dot{\yb}), \xbar^k - \xhat^k} \nonumber\\
&\qquad \qquad + \frac{\bar{L}_\Ab \tau_k^2}{2 \beta_{k+1}}\norm{\tilde{\xb}^k - \xopt}^2_{\Xc}  - \frac{\bar{L}_\Ab \tau_k^2}{2 \beta_{k+1}}\norm{\tilde{\xb}^{k+1} - \xopt}^2_{\Xc} \\
&\hspace{-1em}\overset{\tiny\eqref{eq:cocoercivity_nabla_h_beta}+\eqref{eq:magic_return_to_zero}}{\leq} 
(1-\tau_k)f(\xbar^k) + \tau_k f(\xopt) + \tau_k g(\Ab\xopt) - \tau_k \beta_{k+1} b_{\Yc}(\nabla{g}_{\beta_{k+1}}(\Ab\xhat^k; \dot{\yb}), \dot{\yb}) \nonumber\\
& \qquad \qquad + (1-\tau_k)g_{\beta_{k+1}}(\Ab\xbar^k; \dot{\yb}) - (1-\tau_k)\frac{\beta_{k+1}}{2}\norm{\nabla{g}_{\beta_{k+1}}(\Ab\xhat^k; \dot{\yb}) -\nabla{g}_{\beta_{k+1}}(\Ab\xbar^k; \dot{\yb}) }^2_{\Yc} \nonumber\\
& \qquad \qquad  + \frac{\bar{L}_\Ab \tau_k^2}{2 \beta_{k+1}}\norm{\tilde{\xb}^k - \xopt}^2_{\Xc}  - \frac{\bar{L}_\Ab \tau_k^2}{2 \beta_{k+1}}\norm{\tilde{\xb}^{k+1} - \xopt}^2_{\Xc}\\
&\overset{\tiny\eqref{eq:convexity_h_beta}}{\leq} (1-\tau_k)f(\xbar^k) + \tau_k f(\xopt) + \tau_k g(\Ab\xopt) + (1-\tau_k)g_{\beta_{k}}(\Ab\xbar^k; \dot{\yb})  \nonumber\\
& \qquad \qquad - \frac{\tau_k \beta_{k+1}}{2} \norm{\nabla{g}_{\beta_{k+1}}(\Ab\xhat^k; \dot{\yb}) - \dot{\yb}}^2_{\Yc} + (1-\tau_k)(\beta_k - \beta_{k+1}) b_{\Yc}(\nabla{g}_{\beta_{k+1}}(\Ab\xbar^k; \dot{\yb}), \dot{\yb}) \\
& \qquad \qquad- (1-\tau_k)\frac{\beta_{k+1}}{2}\norm{\nabla{g}_{\beta_{k+1}}(\Ab\xhat^k; \dot{\yb}) -\nabla{g}_{\beta_{k+1}}(\Ab\xbar^k; \dot{\yb}) }^2_{\Yc} \\
& \qquad \qquad  + \frac{\bar{L}_\Ab \tau_k^2}{2 \beta_{k+1}}\norm{\tilde{\xb}^k - \xopt}^2_{\Xc}  - \frac{\bar{L}_{\Ab} \tau_k^2}{2 \beta_{k+1}}\norm{\tilde{\xb}^{k+1} - \xopt}^2_{\Xc}.
\end{align*}
By using \eqref{eq:combine_tau} from Lemma~\ref{lem:move_beta}, we can further estimate this inequality as
\begin{align*}
f(\xbar^{k+1}) + g_{\beta_{k+1}}(\Ab\xbar^{k+1}; \dot{\yb}) & \overset{\tiny\eqref{eq:combine_tau}}{\leq} (1-\tau_k)f(\xbar^k) + (1-\tau_k)g_{\beta_{k}}(\Ab\xbar^k; \dot{\yb}) + \tau_k f(\xopt) + \tau_k g(\Ab\xopt) \nonumber\\
&+ (\beta_k - \beta_{k+1})(1-\tau_k)b_{\Yc}(\nabla{g}_{\beta_{k+1}}(\Ab\xbar^k; \dot{\yb}), \dot{\yb}) \nonumber\\
& - \frac{\beta_{k+1}}{2}\tau_k(1-\tau_k) \norm{\nabla{g}_{\beta_{k+1}}(\Ab\xbar^k; \dot{\yb}) - \dot{\yb}}^2_{\Yc} \nonumber\\
&  + \frac{\bar{L}_{\Ab} \tau_k^2}{2 \beta_{k+1}}\norm{\tilde{\xb}^k - \xopt}^2_{\Xc}  - \frac{\bar{L}_{\Ab} \tau_k^2}{2 \beta_{k+1}}\norm{\tilde{\xb}^{k+1} - \xopt}^2_{\Xc}.
\end{align*}
Finally, using the $L_{b_{\Yc}}$-Lipschitz continuity of $\nabla{b}_{\Yc}$ in the weighted-norm $\norm{\cdot}_{\Yc}$ and the fact that $\nabla{b}_{\Yc}(\dot{\yb}, \dot{\yb}) = 0$, we obtain \eqref{eq:key_estimate_1a} from the last derivation.
\end{proof}

\subsubsection{The proof of Lemma~\ref{lem:choose_parameters_asgard}: Small smoothed primal optimality gap}\label{apdx:lem:choose_parameters_asgard}
Let us denote $S_{\beta_k}(\bar{\xb}^k; \dot{\yb}) := P_{\beta_{k+1}}(\xbar^k; \dot y) - \Popt = f(\xbar^{k}) + g_{\beta_{k}}(\Ab\xbar^{k}; \dot{\yb}) - f(\xopt) - g(\Ab\xopt)$. 
Using \eqref{eq:key_estimate_1a} from Lemma~\ref{le:key_estimate_1a}, we have
\begin{align}\label{eq:lm31_proof1}
S_{\beta_{k+1}}(\bar{\xb}^{k+1};\dot{\yb}) &+ \frac{\bar{L}_{\Ab} \tau_k^2}{2\beta_{k+1}} \norm{\xtilde^{k+1} - \xopt}^2_{\Xc}  \leq (1-\tau_k) S_{\beta_k}(\bar{\xb}^k;\dot{\yb})  + \frac{\bar{L}_{\Ab} \tau_k^2}{2\beta_{k+1}} \norm{\xtilde^{k} - \xopt}^2_{\Xc} \notag\\
& + \frac{(1-\tau_k)}{2}\big[ (\beta_k - \beta_{k+1})L_{b_{\Yc}} - \beta_{k+1}\tau_k\big] \norm{\nabla{g}_{\beta_{k+1}}(\Ab\xbar^k; \dot{\yb}) - \dot{\yb}}^2_{\Yc}.
\end{align}
In order to remove the last term in this estimate and to get a telescoping sum, we can impose the following conditions:
\begin{align}\label{eq:key_condition_1a}
(\beta_k - \beta_{k+1}) L_{b_{\Yc}} = \beta_{k+1} \tau_k ~~~\text{and}~~~ (1 - \tau_k)\frac{\beta_{k+1}}{\tau_k^2} = \frac{\beta_k}{\tau_{k-1}^2}.
\end{align}
By eliminating $\beta_k$ and $\beta_{k+1}$ from these equalities, we obtain $\tau_k^2(1+\tau_k/L_{b_{\Yc}}) = \tau_{k-1}^2(1-\tau_k)$.
Hence, we can compute $\tau_k$ by solving the cubic equation 
\begin{equation}\label{eq:cubic_eq}
p_3(\tau) := \tau^3/L_{b_{\Yc}} + \tau^2 + \tau_{k-1}^2 \tau - \tau_{k-1}^2 = 0.
\end{equation}
At the same time, we also obtain from \eqref{eq:key_condition_1a} an update rule $\beta_{k+1} := \frac{\beta_k}{1 + \frac{\tau_k}{L_{b_{\Yc}}}} < \beta_k$.

Now, we show that \eqref{eq:cubic_eq} has a unique positive solution $\tau_k\in (0, 1)$ for any $L_{b_{\Yc}} \geq 1$ and $\tau_{k-1}\in (0, 1]$.
We consider the cubic polynomial $p_3(\tau)$ defined by the left-hand side of \eqref{eq:cubic_eq}.
Clearly, for any $\tau > 0$, we have $p_3'(\tau) = 3\tau^2/L_{b_{\Yc}} + 2\tau + \tau_{k-1}^2 > 0$. Hence, $p_3(\cdot)$ is monotonically increasing on $(0, +\infty)$. 
In addition, since $p_3(0) = -\tau_{k-1}^2 < 0$ and $p_3(1) = 1/L_{b_{\Yc}} + 1 > 0$, the equation \eqref{eq:cubic_eq} has only one positive solution $\tau_k \in (0, 1)$.

Next, we show that $\tau_k \leq \frac{2}{k+2}$.
Indeed, by \eqref{eq:cubic_eq} we have $p_3(\tau) \geq \tau^2 +  \tau_{k-1}^2\tau - \tau_{k-1}^2 := p_2(\tau)$.
Since the unique positive root of $p_2(\tau) = 0$ is $\tilde{\tau}_k := \frac{\tau_{k-1}}{2}\left( \sqrt{ \tau_{k-1}^2 + 4} - \tau_{k-1}\right)$, we have $p_3(\tau) \geq p_2(\tilde{\tau}_k) = 0$ for $\tau \geq \tilde{\tau}_k$.
As $p_3(\tau)$ is monotonically increasing on $\R_{+}$, its positive solution $\tau_k$ must be in $(0, \tilde{\tau}_k]$. Hence, we have $\tau_k \leq \frac{\tau_{k-1}}{2}\left( \sqrt{ \tau_{k-1}^2 + 4} - \tau_{k-1}\right)$.
By induction, we can easily show that $\tau_k \leq \frac{2}{k+2}$.

We show by induction that $\tau_k \geq \frac{1}{k+1}$. 
First of all, by the choice of $\tau_0$, we have $\tau_0 = 1 \geq \frac{1}{0+1}$.
Suppose that $\tau_{k-1} \geq \frac{1}{k}$, we show that $\tau_k \geq \frac{1}{k+1}$. 
Assume by contradiction that $\tau_k < \frac{1}{k+1}$. Then, using \eqref{eq:key_condition_1a} we have
\begin{align*} 
\frac{1}{k^2} \leq \tau_{k-1}^2 = \tau_k^2 \frac{1 + \tau_k/L_{b_{\Yc}}}{1-\tau_k} < \frac{1}{(k+1)^2} \frac{1 + \frac{L_{b_{\Yc}}^{-1}}{k+1}}{1 - \frac{1}{k+1}} = \frac{1}{(k+1)^2} \frac{k + 1 + L_{b_{\Yc}}}{k}.
\end{align*}
This is equivalent to $(k+1)^2 < k (k+1+L_{b_{\Yc}})$, which contradicts the assumption that $L_{b_{\Yc}} = 1$ in Lemma~\ref{lem:choose_parameters_asgard}.
Hence, if $\tau_{k-1} \geq \frac{1}{k}$, then we have $\tau_{k} \geq \frac{1}{k+1}$. 
We have  $\frac{1}{k+1} \leq \tau_k \leq \frac{2}{k+2}$ for $k\geq 0$.

By the update rule $\beta_{k+1} := \frac{\beta_k}{1 + \frac{\tau_k}{L_{b_{\Yc}}}}$ of $\beta_k$, we can show that
\begin{equation*}
\beta_{k+1} = \frac{\beta_k}{1 + \tau_k/L_{b_{\Yc}}} \leq \beta_k \frac{k+1}{k+1 +L_{b_{\Yc}}^{-1}} \leq \beta_1 \prod_{l = 1}^k \frac{l+1}{l+1 +L_{b_{\Yc}}^{-1}} = \mathcal{O}\Big(\frac{1}{k^{1/L_{b_{\mathcal Y}}}}\Big)\underset{k \to \infty}{\longrightarrow} 0.
\end{equation*}
Clearly, if  $L_{b_{\Yc}} =1$, then $\beta_{k+1} = \frac{\beta_k}{1 + \tau_k} \leq \frac{k+1}{k+2}\beta_k \leq \frac{2\beta_1}{k+2}$ by induction.

Finally, we upper bound the ratio $\tau_k^2/\beta_{k+1}$ by using the second equality in \eqref{eq:key_condition_1a} as
\begin{equation*}
\frac{\tau_k^2}{\beta_{k+1}} = \frac{\tau_{k-1}^2}{\beta_{k}}(1-\tau_k) = \frac{\tau_{0}^2}{\beta_{1}}\prod_{l=1}^k(1-\tau_l) \leq \frac{\tau_{0}^2}{\beta_{1}}\prod_{l=1}^k (1-\frac{1}{l+1}) = \frac{\tau_{0}^2}{\beta_{1}}\prod_{l=1}^k \frac{l}{l+1} = \frac{\tau_{0}^2}{\beta_{1}(k+1)}.
\end{equation*}
Using these relations into \eqref{eq:lm31_proof1} and letting $S_k := S_{\beta_k}(\bar{\xb}^k;\dot{\yb})$, we obtain
\begin{align*}
\frac{\beta_{k\!+\!1} }{\tau_k^2} S_{k\!+\!1} + \frac{\bar{L}_{\Ab}}{2} \norm{\xtilde^{k\!+\!1} \!\!-\! \xopt}^2_{\Xc} \leq \frac{\beta_{k} }{\tau_{k\!-\!1}^2} S_k  + \frac{\bar{L}_{\Ab}}{2} \norm{\xtilde^{k} \!\!-\! \xopt}^2_{\Xc} \leq \frac{\beta_{0}(1 \!-\! \tau_0) }{\tau_{0}^2} S_0   \!+\! \frac{\bar{L}_\Ab}{2} \norm{\xtilde^{0} \!-\! \xopt}^2_{\Xc},
\end{align*}
we get \eqref{eq:key_est_a0} with noting that $\tau_0 = 1$, the bound on $\frac{\tau_k^2}{\beta_{k+1}}$ and $S_k := P_{\beta_{k}}(\xbar^k;\dot{\yb}) - \Popt$.
\Eproof

\subsection{The analysis of the update rule~\eqref{eq:update_tau_new}}\label{apdx:new_update_tau}
If we choose $p_{\Yc}(\yb) := \frac{1}{2}\Vert\yb\Vert_2^2$, then $b_{\Yc}(\yb,\dot{\yb}) = \frac{1}{2}\Vert\yb - \dot{\yb}\Vert_2^2$. 
We can compute $\yb^{\ast}_{\beta}(\ub;\dot{\yb})$ from \eqref{eq:yast_beta} explicitly as $\yb^{\ast}_{\beta}(\ub;\dot{\yb}) = \dot{\yb} + \frac{1}{\beta}(\ub - \cb)$, and $g_{\beta}(\ub;\dot{\yb})$ from \eqref{eq:g_beta} as $g_{\beta}(\ub;\dot{\yb}) = \frac{1}{2\beta}\Vert\ub - \cb\Vert_2^2 + \iprods{\dot{\yb}, \ub - \cb}$.
Hence, $g_{\beta_{k+1}}(\ub;\dot{\yb})  = g_{\beta_k}(\ub;\dot{\yb}) + \frac{(\beta_k - \beta_{k+1})}{\beta_{k+1}\beta_k}\Vert\ub - \cb\Vert_2^2$.
Using this relation into the proof of Lemma~\ref{le:key_estimate_1a} instead of \eqref{eq:convexity_h_beta}, we obtain 
\begin{align*} 
S_{\beta_{k+1}}(\bar{\xb}^{k+1};\dot{\yb}) &+ \frac{\bar{L}_\Ab \tau_k^2}{2\beta_{k+1}} \norm{\xtilde^{k+1} - \xopt}_{\Xc}^2  \leq (1-\tau_k) S_{\beta_k}(\bar{\xb}^k;\dot{\yb})  + \frac{\bar{L}_\Ab \tau_k^2}{2\beta_{k+1}} \norm{\xtilde^{k} - \xopt}_{\Xc}^2 \notag\\
& + \frac{(1-\tau_k)}{2\beta_k\beta_{k+1}}\left[\beta_{k+1} - (1-\tau_k)\beta_k\right] \norm{\nabla{g}_{\beta_{k+1}}(\Ab\bar{\xb}^k; \dot{\yb}) - \dot{\yb}}^2_{\Yc}.
 \end{align*}
Hence, if we choose $\beta_{k+1} = (1-\tau_k)\beta_k$, then, we can remove the last term in the above estimate. 
Combining this rule and the second condition of \eqref{eq:key_condition_1a}, we obtain the update rule \eqref{eq:update_tau_new}.
\Eproof


\subsection{The proof of Lemma~\ref{le:maintain_excessive_gap2}:  Gap reduction in \ref{eq:pd_scheme_2d}}\label{apdx:le:maintain_excessive_gap2}
For simplicity of notation, we denote by $f^{*}_k(y) := f^{\ast}_{\gamma_{k+1}}(-\Ab^{\top} y; \dot{\xb})$ using \eqref{eq:D_gamma}, $\ybar^{\ast}_k := \yb^{\ast}_{\beta_k}(\Ab\xbar^k;\dot{\yb})$, $\xhat^{\ast}_{k\!+\!1} := \xb^{\ast}_{\gamma_{k\!+\!1}}(\yhat^k; \dot{\xb})$ and $\xbar^{\ast}_{k\!+\!1} := \xb^{\ast}_{\gamma_{k\!+\!1}}(\ybar^k; \dot{\xb})$.
By  \eqref{eq:taylor_nabla_h_beta}, $\nabla f^{\ast}_{\gamma}$ is Lipschitz continuous with the Lipschitz constant $L_{f^{\ast}_{\gamma}} := \gamma^{-1}$ and thus $\nabla f^{\ast}_k$ is Lipschitz continuous with the Lipschitz constant $\gamma_{k+1}^{-1} \bar L_\Ab$.

First, using the optimality condition for problem \eqref{eq:F_beta}, we obtain
\vspace{-0.75ex}
\begin{align}\label{eq:lm41_proof1a}
f(\xbar^k) + \iprods{\Ab\xbar^k,\yb} - \beta_kb_{\Yc}(\yb, \dot{\yb}) - g^{\ast}(\yb) \leq P_{\beta_k}(\xbar^k;\dot{\yb}) - (\beta_k/2)\norm{\yb-\ybar^{\ast}_k}_{\Yc}^2.
\vspace{-0.75ex}
\end{align}
Second, using the definition of $f^{\ast}_{\gamma}(\cdot;\dot{\xb})$ in \eqref{eq:D_gamma}, we can show that
\vspace{-1ex}
\begin{align}\label{eq:lm41_proof1b}
\iprods{\Ab\xhat^{\ast}_{k\!+\!1},\yb} + f(\xhat^{\ast}_{k\!+\!1})  &= -\gamma_{k\!+\!1}b_{\Xc}(\xhat^{\ast}_{k\!+\!1}, \dot{\xb}) - f^{\ast}_k(\yhat^k) + \iprods{\Ab\xhat^{\ast}_{k\!+\!1}, \yb - \yhat^k}\notag\\
&=  -\gamma_{k\!+\!1}b_{\Xc}(\xhat^{\ast}_{k\!+\!1}, \dot{\xb}) - f^{\ast}_k(\yhat^k) - \iprods{\nabla f^{\ast}_k(\yhat^k), \yb - \yhat^k}.
\vspace{-1ex}
\end{align}
Third, using \eqref{eq:convexity_h_beta} for $f^{\ast}_{\gamma}$ and  the inequality~\eqref{eq:cocoercivity_nabla_h_beta} of $f^{\ast}_{\gamma_{k\!+\!1}}(\cdot;\dot{\xb})$, we can derive
\vspace{-1ex}
\begin{align}\label{eq:lm41_proof1c}
-D_{\gamma_k}(\ybar^k;\dot{\xb}) &= f^{\ast}_{\gamma_k}(-\Ab^\top \ybar^k;\dot{\xb}) + g^{\ast}(\ybar^k)  \notag \\
&\overset{\tiny\eqref{eq:convexity_h_beta}}{\geq} f^{\ast}_{\gamma_{k+1}}(-\Ab^{\top} \ybar^k; \dot{\xb}) + g^{\ast}(\ybar^k) - (\gamma_k-\gamma_{k\!+\!1})b_{\Xc}(\xbar^{\ast}_{k\!+\!1};\dot{\xb}) \notag\\
&\overset{\tiny\eqref{eq:cocoercivity_nabla_h_beta}}{\geq}  f^{\ast}_{\gamma_{k+1}}(-\Ab^{\top} \yhat^k; \dot{\xb}) + \iprods{\nabla f^{\ast}_{\gamma_{k+1}}(- \Ab^{\top} \yhat^k; \dot{\xb}), \Ab^{\top}(\hat{\yb}^k -  \ybar^k)} \notag \\
&\qquad\qquad + \frac{\gamma_{k\!+\!1}}{2}\norm{\nabla f^{\ast}_{\gamma_{k+1}}(- \Ab^{\top} \ybar^k; \dot{\xb})- \nabla f^{\ast}_{\gamma_{k+1}}(- \Ab^{\top} \yhat^k; \dot{\xb})}^2_{\Xc}\notag\\
&\qquad\qquad + g^{\ast}(\ybar^k) - (\gamma_k-\gamma_{k\!+\!1})b_{\Xc}(\xbar^{\ast}_{k\!+\!1}, \dot{\xb}) \notag \\
& =   f^{\ast}_k(\yhat^k) + \iprods{\nabla f^{\ast}_k(\yhat^k), \ybar^k - \hat{\yb}^k} 
 + \frac{\gamma_{k\!+\!1}}{2}\norm{\bar x^*_{k+1} - \hat x^*_{k+1} }^2_{\Xc}\notag\\
&\qquad\qquad + g^{\ast}(\ybar^k) - (\gamma_k-\gamma_{k\!+\!1})b_{\Xc}(\xbar^{\ast}_{k\!+\!1}, \dot{\xb}).
\end{align}
Then, by the definition of $\ybar^{k\!+\!1}$, we can write
\begin{align}
D_{\gamma_{k\!+\!1}}(\ybar^{k\!+\!1};\dot{\xb}) \!&=\! - g^{\ast}(\ybar^{k+1}) - f^{\ast}_{\gamma_{k+1}}(-\Ab^{\top} \ybar^{k+1}, \dot{\xb}) \notag \\
&\geq - g^{\ast}(\ybar^{k+1}) - f^{\ast}_{k}(\yhat^{k}) - \iprods{\nabla f^{\ast}_k(\yhat^k), \ybar^{k+1} \!-\! \yhat^k}  \!-\!  \frac{\bar L_\Ab}{2\gamma_{k+1}}\norm{\ybar^{k+1} \!-\!\yhat^k}_{\Yc}^2 \notag \\
&= -\min_{\ub\in\Yc}\! \set{\! g^{\ast}(\ub) + f^{\ast}_k(\yhat^k) \!+\! \iprods{\nabla f^{\ast}_k(\yhat^k), \ub \!-\! \yhat^k} \!+\!  \frac{\bar L_\Ab}{2\gamma_{k+1}}\norm{\ub \!-\!\yhat^k}_{\Yc}^2 \!}.{\!\!\!\!\!\!\!\!}\label{eq:lm41_proof1d}
\vspace{-1ex}
\end{align}
Using these relations, the definition of $\xbar^{k\!+\!1}$, and the convexity of $f$, we have
\begin{align*}
P_{\beta_{k\!+\!1}}(\xbar^{k\!+\!1};\dot{\yb}) & = f(\xbar^{k\!+\!1}) + \max_{\yb\in\Yc}\set{\iprods{\Ab\xbar^{k\!+\!1},\yb} - g^{\ast}(\yb) - \beta_{k\!+\!1}b_{\Yc}(\yb, \dot{\yb})} \notag\\
&\overset{\tiny\eqref{eq:pd_condition2}}{\leq} \max_{\yb\in\Yc}\Big\{ (1-\tau_k)\left[f(\xbar^k) + \iprods{\Ab\xbar^k,\yb} - \beta_kb_{\Yc}(\yb, \dot{\yb}) - g^{\ast}(\yb)\right] \notag\\
&\qquad + \tau_k\left[\iprods{\Ab\xhat^{\ast}_{k\!+\!1},\yb} + f(\xhat^{\ast}_{k\!+\!1}) - g^{\ast}(\yb)\right] \Big\}\notag\\
&\hspace{-1.5em}\overset{\tiny\eqref{eq:lm41_proof1a}+\eqref{eq:lm41_proof1b}}{\leq}  (1-\tau_k)P_{\beta_k}(\xbar^k;\dot{\yb}) - \tau_k\gamma_{k\!+\!1}b_{\Xc}(\xhat^{\ast}_{k\!+\!1}, \dot{\xb}) \notag\\
&\qquad - \min_{\yb\in\Yc}\Big\{  \tau_k f^{\ast}_k(\yhat^k) \!+\! \tau_k\iprods{\nabla f^{\ast}_k(\yhat^k), \yb \!-\! \yhat^k}  + \frac{(1 \!-\! \tau_k)\beta_k}{2}\norm{\yb \!-\! \ybar^{\ast}_k}_{\Yc}^2 \!+\! \tau_kg^{\ast}(\yb) \Big\}\notag\\
&\overset{\tiny\eqref{eq:lm41_proof1c}}{\leq} (1-\tau_k)\left[P_{\beta_k}(\xbar^k;\dot{\yb}) - D_{\gamma_k}(\ybar^k;\dot{\xb})\right] - \tau_k\gamma_{k\!+\!1}b_{\Xc}(\xhat^{\ast}_{k\!+\!1}, \dot{\xb})\notag\\
&\qquad -  \frac{(1-\tau_k)\gamma_{k\!+\!1}}{2}\norm{\bar x^*_{k+1} - \hat x^*_{k+1} }^2_{\Xc}  + (1\!-\! \tau_k)(\gamma_k\!-\! \gamma_{k\!+\!1})b_{\Xc}(\xbar^{\ast}_{k\!+\!1}, \dot{\xb})\notag\\
&\qquad- \min_{\yb\in\Yc}\Big\{ f^{\ast}_k(\yhat^k) + \iprods{\nabla f^{\ast}_k(\yhat^k), (1\!-\!\tau_k)\ybar^k \!+\! \tau_k\yb \!-\! \yhat^k}\notag\\
&\quad\qquad\qquad  +  \frac{(1\!-\! \tau_k)\beta_k}{2}\norm{\yb \! -\! \ybar^{\ast}_k}_{\Yc}^2 + g^{\ast}((1-\tau_k)\ybar^k + \tau_k\yb) \Big\}.
\end{align*}
Let us define the auxiliary term $\Tc_k$  as
\begin{align}\label{eq:Tk_term}
\begin{array}{ll}
\Tc_k &:= (1 \!-\! \tau_k)\frac{\gamma_{k\!+\!1}}{2}\norm{\bar x^*_{k+1} - \hat x^*_{k+1} }^2_{\Xc} - (1\!-\!\tau_k)(\gamma_k\!-\!\gamma_{k\!+\!1})b_{\Xc}(\xbar^{\ast}_{k\!+\!1}, \dot{\xb})\vspace{0.75ex}\\
& +  \tau_k\gamma_{k\!+\!1}b_{\Xc}(\xhat^{\ast}_{k\!+\!1}, \dot{\xb}).
\end{array}
\vspace{-1ex}
\end{align}
Now, we consider the change of variable $\ub := (1-\tau_k)\ybar^k + \tau_k\yb$ for $\yb\in\Yc$. Then, $\ub\in\Yc$, and $\ub - \yhat^k = \tau_k(\yb - \ybar^{\ast}_k)$. 
We have
\begin{align}
P_{\beta_{k\!+\!1}}(\xbar^{k\!+\!1};\dot{\yb})  &\leq (1-\tau_k)G_{\gamma_k\beta_k}(\wbar^k;\dot{\wb})  - \Tc_k\notag\\
&\qquad - \min_{\ub\in\Yc}\Big\{ f^{\ast}_k(\yhat^k) \!+\! \iprods{\nabla f^{\ast}_k(\yhat^k), \ub - \yhat^k} +  \frac{(1-\tau_k)\beta_k}{2 \tau_k^2}\norm{\ub -\yhat^k}_{\Yc}^2  + g^{\ast}(\ub) \Big\}\notag\\
&\hspace{-1.5em} \overset{\tiny\eqref{eq:lm41_proof1d}+\eqref{eq:pd_condition2}}{\leq} (1-\tau_k)G_{\gamma_k\beta_k}(\wbar^k;\dot{\wb}) + D_{\gamma_{k\!+\!1}}(\ybar^{k\!+\!1};\dot{\xb}) - \Tc_k, \label{eq:lm41_proof2}
\end{align}
Finally, we estimate  $\Tc_k$ in \eqref{eq:Tk_term} using the strong convexity of $b_\Xc(\cdot, \dot{\xb})$ as follows:
\begin{align}\label{eq:Tk_est}
2\Tc_k & \geq (1 \!-\! \tau_k)\gamma_{k\!+\!1}\Vert \xbar^{\ast}_{k\!+\!1} - \xhat^{\ast}_{k\!+\!1}\Vert_{\Xc}^2 + \tau_k\gamma_{k\!+\!1}\Vert \xhat^{\ast}_{k\!+\!1} - \dot{\xb}\Vert_{\Xc}^2 \notag\\
& - (1\!-\!\tau_k)(\gamma_k\!-\!\gamma_{k\!+\!1})L_{b_{\Xc}}\norm{\xbar^{\ast}_{k\!+\!1} - \dot{\xb}}_{\Xc}^2\notag\\
&\overset{\tiny\eqref{eq:combine_tau}}{\geq} (1-\tau_k)\left[\tau_k\gamma_{k\!+\!1}- (\gamma_k\!-\!\gamma_{k\!+\!1})L_{b_{\Xc}}\right]\norm{\xbar^{\ast}_{k\!+\!1} - \dot{\xb}}_{\Xc}^2\notag\\
&\overset{\tiny\eqref{eq:pd_condition2}}{\geq} 0.
\end{align}
Substituting \eqref{eq:Tk_est} into \eqref{eq:lm41_proof2}, we  get $G_{\gamma_{k\!+\!1}\beta_{k\!+\!1}}(\wbar^{k\!+\!1};\dot{\wb}) \leq (1-\tau_k)G_{\gamma_k\beta_k}(\wbar^k;\dot{\wb})$.

Note that this is valid for all $k \geq 1$. 
Using similar ideas together with the relations $\bar x^1 = \hat x^*_1$ and $\hat y^0 = \bar y^*_0$, we also get
\begin{equation*}
G_{\gamma_1, \beta_1}(\bar{\wb}^1; \dot{\wb}) \leq - \gamma_1 b_\Xc(\bar{\xb}^1, \dot{\xb}) + \frac{\bar L_\Ab}{2 \gamma_1} \norm{\bar{\yb}^{\ast}_1 - \bar{\yb}^{\ast}_0}^2_{\Yc} - \beta_1 b_\Yc(\bar{\yb}_1^{\ast}, \dot{\yb})
\end{equation*}
As $\beta_1\gamma_1 \geq \bar L_A$ and $\bar{\yb}^{\ast}_0 := \dot{\yb}$, we obtain $G_{\gamma_1\beta_1}(\bar{\wb}^1;\dot{\wb}) \leq 0$.

Next, we set the equality in three conditions of \eqref{eq:pd_condition2} to get $\gamma_{k+1} = \gamma_k(1+\tau_k/L_{b_{\Xc}})^{-1}$, $\beta_{k+1} = (1-\tau_k)\beta_k$ and $(1-\tau_k)\beta_k\gamma_{k+1} = \tau_k^2 \bar L_\Ab$. 
In particular, $\gamma_{k+1} \beta_{k+1} = \tau_k^2 \bar L_\Ab$ and thus $\gamma_1 \beta_1 = \bar L_\Ab$. By eliminating $\gamma_k$ and $\beta_k$, we obtain $\tau_k^3/L_{b_{\Xc}} + \tau_k^2 + \tau_{k-1}^2\tau_k - \tau_{k-1}^2 = 0$. 
Hence, similar to the proof of Lemma~\ref{lem:choose_parameters_asgard}, we can show that $\tau_k\in (0, 1)$ is the unique positive solution of the cubic equation $p_3(\tau) := \tau^3/L_{b_{\Xc}} + \tau^2 + \tau_{k-1}^2\tau - \tau_{k-1}^2 = 0$. In addition, $\frac{1}{k+1} \leq \tau_k \leq \frac{2}{k+2}$ for $k\geq 1$ and $\tau_0 = 1$.
If $L_{b_{\Xc}} = 1$, then $\gamma_{k+1} = \frac{\gamma_k}{1 + \tau_k} \leq \frac{\gamma_k(k+1)}{k+2} \leq \frac{2\gamma_1}{k+2}$.
Similarly, $\beta_{k+1} = (1-\tau_k)\beta_k \leq \frac{k}{k+1}\beta_k \leq \frac{\beta_1}{k+1}$.
Finally, we note that $\beta_{k+1} = \frac{\tau_k^2\bar L_\Ab}{\gamma_{k+1}} \geq \frac{\bar L_\Ab}{(k+1)^2}\frac{k+2}{2\gamma_1} \geq \frac{\bar L_\Ab}{2\gamma_1(k+1)}$.
\Eproof

\subsection{The proof of Proposition \ref{th:convergence_A2}: The accelerated augmented Lagrangian method}\label{apdx:th:convergence_A2}
First of all, with the choice of norm associated with the Lagrangian smoother, we have
\vspace{-0.5ex}
\begin{equation*}
\bar{L}_{\Ab} := \norm{\Ab}^2 = \max_{\xb \in \R^n} \set{ \frac{\norm{\Ab\xb}_{\Yc, *}^2}{\norm{\xb}_\Xc^2} } = \max_{\xb \in \R^n} \set{ \frac{\norm{\Ab\xb}_{\Yc,\ast}^2}{\norm{\Ab \xb}_{\Yc,\ast}^2} } = 1.
\vspace{-0.5ex}
\end{equation*}
Next, note that the conclusions of Lemma~\ref{le:maintain_excessive_gap2}
are valid for any semi-norm. In particular,
if we choose $\beta_1 \gamma_0 \geq \bar L_\Ab = 1$,  
\vspace{-0.5ex}
\begin{equation*}
\gamma_{k+1} = \gamma_0 \geq \frac{\gamma_0}{1 + \tau_k / L_{b_\Xc}} ,\quad  \beta_{k+1} = (1-\tau_k) \beta_k, \quad\text{ and }\quad \frac{\bar L_\Ab}{\gamma_0} = \frac{(1-\tau_k)\beta_k}{\tau_k},
\vspace{-0.5ex}
\end{equation*}
then $G_{\gamma_0, \beta_{k+1}}(\bar{\wb}^{k+1}; \dot{\wb}) \leq (1-\tau_k) G_{\gamma_0, \beta_{k}}(\bar{\wb}^{k}; \dot{\wb}) \leq 0$.

Eliminating $\beta_{k+1}$ and $\beta_k$ in these equalities, we get $\frac{\tau_{k+1}^2}{1 - \tau_{k+1}} = \tau_{k}^2$.
One can easily check by induction that $\beta_{k} =\beta_1 \prod_{l=1}^k (1-\tau_l) = \beta_1\frac{\tau_k^2}{\tau_0^2} = \frac{\tau_k^2}{\gamma_0}$ and $\tau_k \leq \frac{2}{k+2}$.
We then conclude using Lemma~\ref{le:excessive_gap_aug_Lag_func} and the fact that  $b_\Xc(\xopt, \dot x) = 0$ that
\begin{equation*}
S_{\beta_k}(\bar{\xb}^k;\dot{\yb}) \leq G_{\gamma_0\beta_k}(\bar{\wb}^k; \dot{\wb}) \leq 0,
\vspace{-1.5ex}
\end{equation*}
\begin{align*}
\norm{\Ab \bar{\xb}^k - \cb}_{\Yc, *} &\leq  \beta_k L_{b_{\Yc}} \Big[ \norm{\yopt - \dot{\yb}}_{\Yc}  + \big(\norm{\yopt - \dot{\yb}}_{\Yc}^2 + 2L_{b_{\Yc}}^{-1}\beta_k^{-1}S_{\beta_k}(\bar{\xb}^k; \dot{\yb}) \big)^{1/2}\Big] \leq \frac{8 L_{b_\Yc}\norm{\yopt - \dot{\yb}}_{\Yc}}{\gamma_0(k+2)^2},
\end{align*}
and
\begin{align*}
f(\bar \xb^k) - \fopt &\leq S_{\beta_k}(\bar \xb^k;\dot{\yb}) -\iprods{\yopt, \Ab\bar \xb^k -\cb} + \beta_k b_{\Yc}(\yopt, \dot{\yb}) \\
& \leq \norm{\yopt}_\Yc \norm{\Ab\bar \xb^k -\cb}_{\Yc, *} + \beta_k b_{\Yc}(\yopt, \dot{\yb}) \leq \frac{8 L_{b_\Yc}\norm{\yopt}_\Yc \norm{\yopt - \dot{\yb}}_{\Yc}+ 4b_{\Yc}(\yopt, \dot{\yb})}{\gamma_0 (k+2)^2}, \\
f(\bar \xb^k) - \fopt &\geq -\Vert\yopt\Vert_{\Yc}\Vert\Ab\xb - \cb\Vert_{\Yc, *} \geq 
- \frac{8 L_{b_\Yc}\norm{\yopt}_\Yc\norm{\yopt - \dot{\yb}}_{\Yc}}{\gamma_0 (k+2)^2}.
\end{align*}
The proposition is proved.
\Eproof

\subsection{The proof of Proposition \ref{co:strong_convex_convergence}: The strongly convex objective case}\label{apdx:co:strong_convex_convergence}
The proof follows the same arguments as the proof of Lemma~\ref{le:maintain_excessive_gap2}.
We only need to replace the Lipschitz continuity coefficient $\frac{\bar{L}_{\Ab}}{\gamma_{k+1}}$ by $L_{f^{*}_{\Ab}} = \frac{\bar{L}_{\Ab}}{\mu_f}$ in \eqref{eq:lm41_proof1d} and replace all other occurrences of $\gamma_{k+1}$ by zero.
Under a choice of parameters satisfying \eqref{eq:update_tau_beta7_muf}, we obtain the gap reduction condition 
\begin{equation*}
G_{0,\beta_{k+1}}(\wbar^{k+1}; \dot{\wb}) \leq (1-\tau_k)G_{0, \beta_k}(\wbar^k; \dot{\wb})  \leq 0,
\end{equation*} 
as in Lemma \ref{le:maintain_excessive_gap2}.
We can also check by induction that $\beta_k \leq \frac{4}{(k+2)^2} \frac{\bar{L}_{\Ab}}{\mu_f}$. 
Hence, we obtain the conclusion of Proposition~\ref{co:strong_convex_convergence} by using Lemma \ref{le:excessive_gap_aug_Lag_func}.
\Eproof

\bibliographystyle{siamplain}

\begin{thebibliography}{10}

\bibitem{Auslender1976}
{\sc A.~Auslender}, {\em {O}ptimisation: {M}\'{e}thodes {N}um\'{e}riques},
  Masson, Paris, 1976.

\bibitem{Bauschke2011}
{\sc H.~Bauschke and P.~Combettes}, {\em Convex analysis and monotone operators
  theory in {H}ilbert spaces}, Springer-Verlag, 2011.

\bibitem{Beck2009}
{\sc A.~Beck and M.~Teboulle}, {\em {A} fast iterative shrinkage-thresholding
  agorithm for linear inverse problems}, SIAM J. Imaging Sci., 2 (2009),
  pp.~183--202.

\bibitem{Beck2012a}
{\sc A.~Beck and M.~Teboulle}, {\em Smoothing and first-order methods: A
  unified framework}, SIAM J. Optim., 22 (2012), pp.~557--580.

\bibitem{Beck2014}
{\sc A.~Beck and M.~Teboulle}, {\em A fast dual proximal gradient algorithm for
  convex minimization and applications}, Oper. Res. Letter, 42 (2014),
  pp.~1--6.

\bibitem{Belloni2011}
{\sc A.~Belloni, V.~Chernozhukov, and L.~Wang}, {\em Square-root {LASSO}:
  {P}ivotal recovery of sparse signals via conic programming}, Biometrika, 94
  (2011), pp.~791--806.

\bibitem{BenTal2001}
{\sc A.~Ben-Tal and A.~Nemirovski}, {\em {L}ectures on modern convex
  optimization: {A}nalysis, algorithms, and engineering applications}, vol.~3
  of MPS/SIAM Series on Optimization, SIAM, 2001.

\bibitem{Bertsekas1989b}
{\sc D.~Bertsekas and J.~N. Tsitsiklis}, {\em {P}arallel and distributed
  computation: {N}umerical methods}, Prentice Hall, 1989.

\bibitem{Bertsekas1996d}
{\sc D.~P. Bertsekas}, {\em {C}onstrained {O}ptimization and {L}agrange
  {M}ultiplier {M}ethods}, Athena Scientific, 1996.

\bibitem{boct2012variable}
{\sc R.~I. Bo{\c{t}} and C.~Hendrich}, {\em A variable smoothing algorithm for
  solving convex optimization problems}, TOP, 23 (2012), pp.~124--150.

\bibitem{Boyd2011}
{\sc S.~Boyd, N.~Parikh, E.~Chu, B.~Peleato, and J.~Eckstein}, {\em Distributed
  optimization and statistical learning via the alternating direction method of
  multipliers}, Foundations and Trends in Machine Learning, 3 (2011),
  pp.~1--122.

\bibitem{Boyd2004}
{\sc S.~Boyd and L.~Vandenberghe}, {\em {C}onvex {O}ptimization}, University
  {P}ress, Cambridge, 2004.

\bibitem{cai2016convergence}
{\sc X.~Cai, D.~Han, and X.~Yuan}, {\em On the convergence of the direct
  extension of {ADMM} for three-block separable convex minimization models with
  one strongly convex function}, Comput. Optim. Appl.,  (2016), pp.~1--35.

\bibitem{Cevher2014}
{\sc V.~Cevher, S.~Becker, and M.~Schmidt}, {\em Convex optimization for big
  data: Scalable, randomized, and parallel algorithms for big data analytics},
  IEEE Signal Processing Magazine, 31 (2014), pp.~32--43.

\bibitem{Chambolle2011}
{\sc A.~Chambolle and T.~Pock}, {\em A first-order primal-dual algorithm for
  convex problems with applications to imaging}, J. Math. Imaging Vis., 40
  (2011), pp.~120--145.

\bibitem{Chandrasekaranm2012}
{\sc V.~Chandrasekaranm, B.~Recht, P.~A. Parrilo, and A.~S. Willsky}, {\em The
  convex geometry of linear inverse problems}, Foundations of Computational
  Mathematics, 12 (2012), pp.~805--849.

\bibitem{chen2016direct}
{\sc C.~Chen, B.~He, Y.~Ye, and X.~Yuan}, {\em The direct extension of {ADMM}
  for multi-block convex minimization problems is not necessarily convergent},
  Math. Program., 155 (2016), pp.~57--79.

\bibitem{Chen1994}
{\sc G.~Chen and M.~Teboulle}, {\em A proximal-based decomposition method for
  convex minimization problems}, Math. Program., 64 (1994), pp.~81--101.

\bibitem{Chen2013a}
{\sc Y.~Chen, G.~Lan, and Y.~Ouyang}, {\em Optimal primal-dual methods for a
  class of saddle-point problems}, SIAM J. Optim., 24 (2014), pp.~1779--1814.

\bibitem{Combettes2011}
{\sc P.~Combettes and J.-C. Pesquet}, {\em Signal recovery by proximal
  forward-backward splitting}, in Fixed-Point Algorithms for Inverse Problems
  in Science and Engineering, Springer-Verlag, 2011, pp.~185--212.

\bibitem{combettes2012primal}
{\sc P.~L. Combettes and J.-C. Pesquet}, {\em Primal-dual splitting algorithm
  for solving inclusions with mixtures of composite, lipschitzian, and
  parallel-sum type monotone operators}, Set-Valued Var. Anal., 20 (2012),
  pp.~307--330.

\bibitem{Condat2013}
{\sc L.~Condat}, {\em A primal–dual splitting method for convex optimization
  involving lipschitzian, proximable and linear composite terms}, J. Optim.
  Theory Appl., 158 (2013), pp.~460--479.

\bibitem{Davis2014a}
{\sc D.~Davis}, {\em Convergence rate analysis of primal-dual splitting
  schemes}, SIAM J. Optim., 25 (2015), pp.~1912--1943.

\bibitem{Davis2014}
{\sc D.~Davis}, {\em Convergence rate analysis of the
  forward-{D}ouglas-{R}achford splitting scheme}, SIAM J. Optim., 25 (2015),
  pp.~1760--1786.

\bibitem{Davis2014b}
{\sc D.~Davis and W.~Yin}, {\em Faster convergence rates of relaxed
  {P}eaceman-{R}achford and {ADMM} under regularity assumptions}, Mathematics
  of Operations Research,  (2014).

\bibitem{deng2013parallel}
{\sc W.~Deng, M.-J. Lai, Z.~Peng, and W.~Yin}, {\em Parallel multi-block {ADMM}
  with ${o}(1/k)$ convergence}, J. Scientific Computing, DOI:
  10.1007/s10915-016-0318-2 (2016).

\bibitem{Deng2012}
{\sc W.~Deng and W.~Yin}, {\em On the global and linear convergence of the
  generalized alternating direction method of multipliers}, J. Sci. Comput., 66
  (2016), pp.~889--916.

\bibitem{Esser2010a}
{\sc J.~E. Esser}, {\em Primal-dual algorithm for convex models and
  applications to image restoration, registration and nonlocal inpainting},
  {PhD} {T}hesis, University of California, Los Angeles, Los Angeles, USA,
  2010.

\bibitem{Facchinei2003}
{\sc F.~Facchinei and J.-S. Pang}, {\em Finite-dimensional variational
  inequalities and complementarity problems}, vol.~1-2, Springer-Verlag, 2003.

\bibitem{fercoq2016restarting}
{\sc O.~Fercoq and Z.~Qu}, {\em Restarting accelerated gradient methods with a
  rough strong convexity estimate}, arXiv preprint arXiv:1609.07358,  (2016).

\bibitem{Giselsson2014}
{\sc P.~Giselsson and S.~Boyd}, {\em {M}onotonicity and {R}estart in {F}ast
  {G}radient {M}ethods}, in IEEE Conference on Decision and Control, Los
  Angeles, USA, December 2014, CDC.

\bibitem{Goldstein2013}
{\sc T.~Goldstein, E.~Esser, and R.~Baraniuk}, {\em Adaptive primal-dual hybrid
  gradient methods for saddle point problems}, Tech. Report.,  (2013),
  pp.~1--26.
\newblock http://arxiv.org/pdf/1305.0546v1.pdf.

\bibitem{He2012b}
{\sc B.~He and X.~Yuan}, {\em Convergence analysis of primal-dual algorithms
  for saddle-point problem: from contraction perspective}, SIAM J. Imaging
  Sci., 5 (2012), pp.~119--149.

\bibitem{He2012a}
{\sc B.~He and X.~Yuan}, {\em On non-ergodic convergence rate of
  {D}ouglas--{R}achford alternating direction method of multipliers},
  Numerische Mathematik, 130 (2012), pp.~567--577.

\bibitem{He2012}
{\sc B.~He and X.~Yuan}, {\em On the ${O}(1/n)$ convergence rate of the
  {D}ouglas-{R}achford alternating direction method}, SIAM J. Numer. Anal., 50
  (2012), pp.~700--709.

\bibitem{he2016accelerated}
{\sc Y.~He and R.-D. Monteiro}, {\em An accelerated {HPE}-type algorithm for a
  class of composite convex-concave saddle-point problems}, SIAM J. Optim., 26
  (2016), pp.~29--56.

\bibitem{knoll2011adapted}
{\sc F.~Knoll, C.~Clason, C.~Diwoky, and R.~Stollberger}, {\em Adapted random
  sampling patterns for accelerated {MRI}}, Magnetic resonance materials in
  physics, biology and medicine, 24 (2011), pp.~43--50.

\bibitem{Lan2013}
{\sc G.~Lan and R.~Monteiro}, {\em Iteration complexity of first-order penalty
  methods for convex programming}, Math. Program., 138 (2013), pp.~115--139.

\bibitem{Lan2013b}
{\sc G.~Lan and R.~Monteiro}, {\em Iteration-complexity of first-order
  augmented {L}agrangian methods for convex programming}, Math. Program., 155
  (2016), pp.~511--547.

\bibitem{lin2015iteration}
{\sc T.~Lin, S.~Ma, and S.~Zhang}, {\em Iteration complexity analysis of
  multi-block {ADMM} for a family of convex minimization without strong
  convexity}, J. Sci. Comput.,  (2015), pp.~1--30.

\bibitem{lin2015global}
{\sc T.~Lin, S.~Ma, and S.~Zhang}, {\em On the global linear convergence of the
  admm with multiblock variables}, SIAM J. Optim., 25 (2015), pp.~1478--1497.

\bibitem{lin2015sublinear}
{\sc T.~Lin, S.~Ma, and S.-Z. Zhang}, {\em On the sublinear convergence rate of
  multi-block {ADMM}}, Journal of the Operations Research Society of China, 3
  (2015), pp.~251--274.

\bibitem{malitsky2016first}
{\sc Y.~Malitsky and T.~Pock}, {\em A first-order primal-dual algorithm with
  linesearch}, arXiv preprint arXiv:1608.08883,  (2016).

\bibitem{McCoy2014}
{\sc M.~B. McCoy, V.~Cevher, Q.~Tran-Dinh, A.~Asaei, and L.~Baldassarre}, {\em
  Convexity in source separation: {M}odels, geometry, and algorithms}, IEEE
  Signal Processing Magazine, 31 (2014), pp.~87--95.

\bibitem{Monteiro2010a}
{\sc R.~Monteiro and B.~Svaiter}, {\em On the complexity of the hybrid proximal
  extragradient method for the interates and the ergodic mean}, SIAM J. Optim.,
  20 (2010), pp.~2755--2787.

\bibitem{Monteiro2011}
{\sc R.~Monteiro and B.~Svaiter}, {\em Complexity of variants of {T}seng's
  modified {F-B} splitting and {K}orpelevich’s methods for hemivariational
  inequalities with applications to saddle-point and convex optimization
  problems}, SIAM J. Optim., 21 (2011), pp.~1688--1720.

\bibitem{Monteiro2012b}
{\sc R.~Monteiro and B.~Svaiter}, {\em Iteration-complexity of
  block-decomposition algorithms and the alternating direction method of
  multipliers}, SIAM J. Optim., 23 (2013), pp.~475--507.

\bibitem{Monteiro2010}
{\sc R.~Monteiro and B.~Svaiter}, {\em Iteration-complexity of
  block-decomposition algorithms and the alternating minimization augmented
  {L}agrangian method}, SIAM J. Optim., 23 (2013), pp.~475--507.

\bibitem{necoara2014iteration}
{\sc I.~Necoara and A.~Patrascu}, {\em Iteration complexity analysis of dual
  first-order methods for convex programming}, J. Optim. Theory Appl. (Arxiv
  preprint:1409.1462),  (2014).

\bibitem{Necoara2008}
{\sc I.~Necoara and J.~Suykens}, {\em Applications of a smoothing technique to
  decomposition in convex optimization}, IEEE Trans. Automatic control, 53
  (2008), pp.~2674--2679.

\bibitem{Nedelcu2014}
{\sc V.~Nedelcu, I.~Necoara, and Q.~Tran-Dinh}, {\em {C}omputational
  {C}omplexity of {I}nexact {G}radient {A}ugmented {L}agrangian {M}ethods:
  {A}pplication to {C}onstrained {MPC}}, SIAM J. Optim. Control, 52 (2014),
  pp.~3109--3134.

\bibitem{Nemirovskii2004}
{\sc A.~Nemirovskii}, {\em Prox-method with rate of convergence
  $\mathcal{O}(1/t)$ for variational inequalities with {L}ipschitz continuous
  monotone operators and smooth convex-concave saddle point problems}, SIAM J.
  Op, 15 (2004), pp.~229--251.

\bibitem{Nesterov1983}
{\sc Y.~Nesterov}, {\em A method for unconstrained convex minimization problem
  with the rate of convergence $\mathcal{O}(1/k^2)$}, Doklady AN SSSR, 269
  (1983), pp.~543--547.

\bibitem{Nesterov2004}
{\sc Y.~Nesterov}, {\em {I}ntroductory lectures on convex optimization: {A}
  basic course}, vol.~87 of Applied Optimization, Kluwer Academic Publishers,
  2004.

\bibitem{Nesterov2005d}
{\sc Y.~Nesterov}, {\em Excessive gap technique in nonsmooth convex
  minimization}, SIAM J. Optim., 16 (2005), pp.~235--249.

\bibitem{Nesterov2005c}
{\sc Y.~Nesterov}, {\em Smooth minimization of non-smooth functions}, Math.
  Program., 103 (2005), pp.~127--152.

\bibitem{Nesterov2007a}
{\sc Y.~Nesterov}, {\em Dual extrapolation and its applications to solving
  variational inequalities and related problems}, Math. Program., 109 (2007),
  pp.~319--344.

\bibitem{Nocedal2006}
{\sc J.~Nocedal and S.~Wright}, {\em {N}umerical {O}ptimization}, Springer
  Series in Operations Research and Financial Engineering, Springer, 2~ed.,
  2006.

\bibitem{Odonoghue2012}
{\sc B.~O'Donoghue and E.~Candes}, {\em {Adaptive Restart for Accelerated
  Gradient Schemes}}, Found. Comput. Math., 15 (2015), pp.~715--732.

\bibitem{Ouyang2013}
{\sc H.~Ouyang, N.~He, L.~Q. Tran, and A.~Gray}, {\em Stochastic alternating
  direction method of multipliers}, JMLR W\&CP, 28 (2013), pp.~80--88.

\bibitem{Ouyang2014}
{\sc Y.~Ouyang, Y.~Chen, G.~Lan, and E.~J. Pasiliao}, {\em An accelerated
  linearized alternating direction method of multiplier}, SIAM J. Imaging Sci.,
  8 (2015), pp.~644--681.

\bibitem{Rockafellar2004}
{\sc R.~Rockafellar and R.~Wets}, {\em {V}ariational {A}nalysis}, vol.~317,
  Springer, 2004.

\bibitem{Rockafellar1970}
{\sc R.~T. Rockafellar}, {\em {C}onvex {A}nalysis}, vol.~28 of Princeton
  Mathematics Series, Princeton University Press, 1970.

\bibitem{Shefi2014}
{\sc R.~Shefi and M.~Teboulle}, {\em {R}ate of {C}onvergence {A}nalysis of
  {D}ecomposition {M}ethods {B}ased on the {P}roximal {M}ethod of {M}ultipliers
  for {C}onvex {M}inimization}, SIAM J. Optim., 24 (2014), pp.~269--297.

\bibitem{solodov1999hybrid}
{\sc M.~Solodov and B.~Svaiter}, {\em A hybrid approximate
  extragradient--proximal point algorithm using the enlargement of a maximal
  monotone operator}, Set-Valued Var. Anal., 7 (1999), pp.~323--345.

\bibitem{Su2014}
{\sc W.~Su, S.~Boyd, and E.~Candes}, {\em A differential equation for modeling
  {N}esterov's accelerated gradient method: {T}heory and insights}, in Advances
  in Neural Information Processing Systems (NIPS), 2014, pp.~2510--2518.

\bibitem{TranDinh2014b}
{\sc Q.~Tran-Dinh and V.~Cevher}, {\em Constrained convex minimization via
  model-based excessive gap}, in Proc. the Neural Information Processing
  Systems (NIPS), vol.~27, Montreal, Canada, December 2014, pp.~721--729.

\bibitem{Tran-Dinh2014a}
{\sc Q.~Tran-Dinh and V.~Cevher}, {\em A primal-dual algorithmic framework for
  constrained convex minimization}, Tech. Report., LIONS,  (2014), pp.~1--54.

\bibitem{Tran-Dinh2015}
{\sc Q.~Tran-Dinh and V.~Cevher}, {\em {S}mooth alternating direction methods
  for nonsmooth constrained convex optimization}, Tech. Report. (LIONS, EPFL),
  (2015), \url{http://arxiv.org/abs/1507.03734}.

\bibitem{Tran-Dinh2013a}
{\sc Q.~Tran-Dinh, A.~Kyrillidis, and V.~Cevher}, {\em Composite
  self-concordant minimization}, J. Mach. Learn. Res., 15 (2015), pp.~374--416.

\bibitem{Tseng1991a}
{\sc P.~Tseng}, {\em Applications of splitting algorithm to decomposition in
  convex programming and variational inequalities}, SIAM J. Control Optim., 29
  (1991), pp.~119--138.

\bibitem{vu2013variable}
{\sc B.~C. Vu}, {\em A variable metric extension of the
  forward--backward--forward algorithm for monotone operators}, Numerical
  Functional Analysis and Optimization, 34 (2013), pp.~1050--1065.

\bibitem{Wainwright2014}
{\sc M.~J. Wainwright}, {\em Structured regularizers for high-dimensional
  problems: {S}tatistical and computational issues}, Annu. Rev. Stat. Appl., 1
  (2014), pp.~233--253.

\bibitem{wang2013solving}
{\sc X.~Wang, M.~Hong, S.~Ma, and Z.-Q. Luo}, {\em Solving multiple-block
  separable convex minimization problems using two-block alternating direction
  method of multipliers}, arXiv preprint arXiv:1308.5294,  (2013).

\bibitem{Wei20131}
{\sc E.~Wei and A.~Ozdaglar}, {\em On the $\mathcal{O}(1/k)$-convergence of
  asynchronous distributed alternating direction method of multipliers}, in
  Global Conference on Signal and Information Processing (GlobalSIP), IEEE,
  2013, pp.~551--554.

\bibitem{xu2016accelerated}
{\sc Y.~Xu}, {\em Accelerated first-order primal-dual proximal methods for
  linearly constrained composite convex programming}, arXiv preprint
  arXiv:1606.09155,  (2016).

\end{thebibliography}

\end{document}